\documentclass[11pt]{article}
\usepackage{amsmath,amssymb,amsthm,mathtools}
\usepackage[margin=1in]{geometry}
\usepackage{microtype,booktabs,longtable,array,enumitem,graphicx,xcolor}
\usepackage{tikz,listings}
\usepackage[colorlinks=true,linkcolor=blue!45!black,citecolor=blue!45!black,urlcolor=blue!45!black]{hyperref}
\usepackage[capitalise,noabbrev]{cleveref}
\usepackage{seqsplit}
\usepackage{etoolbox}
\newtheorem{theorem}{Theorem}[section]
\newtheorem{lemma}[theorem]{Lemma}
\newtheorem{proposition}[theorem]{Proposition}
\newtheorem{corollary}[theorem]{Corollary}
\theoremstyle{definition}\newtheorem{definition}[theorem]{Definition}
\newtheorem{remark}[theorem]{Remark}
\newcommand{\Ftwo}{\mathbb F_2}

\newcommand{\ind}{\mathbf 1}

\setlist{itemsep=3pt,topsep=5pt}
\AddToHook{env/theorem/begin}{\crefalias{section}{theorem}\crefalias{theorem}{theorem}}
\AddToHook{env/lemma/begin}{\crefalias{section}{lemma}\crefalias{theorem}{lemma}}
\AddToHook{env/proposition/begin}{\crefalias{section}{proposition}\crefalias{theorem}{proposition}}
\AddToHook{env/corollary/begin}{\crefalias{section}{corollary}\crefalias{theorem}{corollary}}
\AddToHook{env/definition/begin}{\crefalias{section}{definition}\crefalias{theorem}{definition}}
\AddToHook{env/remark/begin}{\crefalias{section}{remark}\crefalias{theorem}{remark}}
\hypersetup{pdftitle={Minimum orders of clique oddomorphism counterexamples},
 pdfauthor={Michal Čertík},pdfsubject={Main article and technical companion},bookmarksopen=true,bookmarksnumbered=true}
\makeatletter
\let\savedtitle\title
\let\savedauthor\author
\let\saveddate\date
\let\savedmaketitle\maketitle
\let\savedinternalmaketitle\@maketitle
\newcommand{\restoremaketitle}{\let\title\savedtitle\let\author\savedauthor\let\date\saveddate\let\maketitle\savedmaketitle\let\@maketitle\savedinternalmaketitle}
\makeatother
\title{Minimum orders of clique oddomorphism counterexamples}
\author{Michal \v{C}ert\'{\i}k\\ The World Bank Group\\ Washington, DC, USA\\ \href{mailto:michal.certik@gmail.com}{\texttt{michal.certik@gmail.com}}}
\date{8 November 2026}
\begin{document}
\maketitle
\begin{abstract}
An oddomorphism to $K_n$ is a proper colouring with $n$ colours in which every vertex is designated odd or even, every vertex has in each other colour class a number of neighbours of its own parity, and every colour class contains an odd number of odd vertices. We prove that an oddomorphism witness to $K_n$ on at most $n+6$ vertices contains a $K_n$ minor with every branch meeting the witness's specified odd set. The conclusion holds inside any graph containing the witness and therefore applies to weak oddomorphisms. The bound is sharp: for every $n\ge8$ the minimum order of a graph admitting a weak oddomorphism to $K_n$ but having no $K_n$ minor is exactly $n+7$. The extremal example is a fifteen-vertex graph with an oddomorphism to $K_8$, obtained from the doubled theta graph of Kriesell and Mohr by adding one vertex adjacent to the seven roots; we give a structural proof that it has no $K_8$ minor, including their counting argument for the rooted case. We also prove that every witness to $K_4$, of arbitrary order, has a $K_4$ minor with every branch meeting its original odd set. The roots need not be prescribed or have distinct colours. In six colours, we prove that a fibre of size at most two forces a $K_5$ minor rooted in the original odd set, with no bound on the other fibres. In particular every six-colour witness on at most seventeen vertices has such a model. A separate technical companion records further results on related root-apex graphs, an attachment theorem and a bounded nine-colour statement, together with their computational provenance; none of it is used here. The qualitative clique-target existence problem has been solved in recent work; our results concern exact source-order bounds and control of the original odd set.
\end{abstract}
\begingroup\small

\par\endgroup
\clearpage
\setcounter{tocdepth}{1}
\begingroup\small\tableofcontents\endgroup
\clearpage
\phantomsection\addcontentsline{toc}{part}{Main article}
\section{Results and their scope}\label{sec:intro}
The relation between parity-constrained graph homomorphisms and clique minors has two different quantitative aspects. One asks whether a sufficiently large complete target forces a prescribed clique minor without restricting the source. Another fixes the target and asks how small a counterexample can be. This paper studies the latter question through the excess $|V(W)|-n$ of a witness to $K_n$. We also prove an unrestricted rooted result for the target $K_4$ and an all-order rooted-five result for six-colour witnesses with a fibre of size at most two.

Roberson introduced oddomorphisms and their weak variant and related them to homomorphism counts of parity constructions \cite{Roberson2022}. The general implication from an oddomorphism to containment of its target as a minor is false: Kristj\'ansson gives counterexamples, including complete targets \cite{Kristjansson2026}. His examples for the target $K_n$ come from random $3$-lifts of $K_n$, using a result of Drier and Linial on minors in random lifts \cite{DrierLinial2006}; they have $3n$ vertices and exist for all sufficiently large $n$ \cite[Section~3.2]{Kristjansson2026}. We determine the least order of such a counterexample exactly: it is $n+7$ for every $n\ge8$.

In subsequent work, Kristj\'ansson refuted the Strong Roberson Conjecture itself and its immersion analogue \cite[Theorems~1 and~14]{KristjanssonCounterexamples2026}. Those counterexamples concern non-complete targets and do not determine the minimum order considered here.

\begin{theorem}[Rooted excess six]\label{thm:main}
Let $n\ge2$. Suppose a finite simple graph $F$ contains a subgraph $W$ with an oddomorphism to $K_n$ and specified odd set $P$. If $|V(W)|\le n+6$, then $W$ has a $K_n$ minor with every branch set meeting $P$. In particular, the same model exists in $F$.
\end{theorem}
This formulation is slightly stronger than a host-order bound: vertices outside $W$ do not count. In particular, a graph of order at most $n+6$ with a weak oddomorphism to $K_n$ satisfies the conclusion. A branch is only required to contain a specified odd vertex; it is not required to contain an odd number of such vertices, remain within a fibre, or have a prescribed colour.

The word \emph{odd} refers to the parity conditions of \cref{sec:normalization}: on the number of neighbours of a vertex in each other colour class, and on the number of odd vertices in each colour class. It is unrelated to odd minors, which occur in the odd-minor variant of Hadwiger's conjecture. There a graph contains an odd $K_n$ minor if it has $n$ vertex-disjoint trees, every two joined by an edge, whose vertices can be two-coloured so that the edges of the trees are bichromatic and the joining edges are monochromatic \cite{GeelenEtAl2009}. All minors in this paper are ordinary minors.

\begin{theorem}[Sharpness]\label{thm:sharp}
For $n\ge8$, let $m(n)$ be the minimum order of a graph admitting a weak oddomorphism to $K_n$ but having no $K_n$ minor. Then $m(n)=n+7$. An extremal example can be chosen to admit an oddomorphism. In particular, $m(8)=15$.
\end{theorem}
The construction uses the theta-based two-vertex-fibre graph of Kriesell and Mohr \cite[Section 3, pp.~8--9, immediately before Theorem 2]{KriesellMohr2022}. Its rooted obstruction is theirs (their Theorem~2 states it as the failure of a rooted routing property); it already shows that a uniform rooted excess-seven theorem fails at $n=7$, and we include their counting argument in \cref{lem:known-theta}. The assertion needed for \cref{thm:sharp} is stronger and different: adjoining a vertex adjacent to the seven designated roots produces no \emph{ordinary} $K_8$ minor. The rooted obstruction excludes only the models in which the new vertex forms a singleton branch; an ordinary model may have a different singleton branch or put the new vertex into a larger branch. We exclude every model structurally in \cref{prop:noK8} using a spanning-model reduction, singleton saturation, and eight elementary cases.

\begin{theorem}[Rooted four-colour witnesses]\label{thm:rooted-four}
Every oddomorphism witness to $K_4$ with specified odd set $P$ contains a $P$-rooted $K_4$ model, with no restriction on its order. Consequently, the same conclusion holds in any graph containing such a witness.
\end{theorem}
This is a set-rooted assertion: each of the four disjoint connected pairwise adjacent branches meets the original set $P$. It does not prescribe four particular roots or require the roots to have distinct colours. The proof in \cref{sec:rooted-four} is independent of \cref{thm:main}, uses no computation, and allows arbitrarily many odd vertices in a fibre.

\begin{theorem}[A small fibre in six colours]\label{thm:small-fibre}
Every oddomorphism witness to $K_6$ with specified odd set $P$ and a fibre of size at most two contains a $P$-rooted $K_5$ model. Consequently, every specified $K_6$ witness on at most seventeen vertices has such a model. The same model lies in every graph containing the witness.
\end{theorem}
There is no bound on the other five fibre sizes in the first assertion. Thus a counterexample to original-set-rooted $K_5$ forcing from six colours, if one exists, has all six fibres of size at least three and order at least eighteen. No example at eighteen or sharpness assertion is claimed. A singleton fibre was already sufficient by \cref{thm:rooted-four,lem:attachment}; the substantive extension is to a double. The proof in \cref{sec:small-fibre} restores original roots through separator quotients, reduces a hypothetical counterexample with a double to a four-connected witness, and applies the classical rural-society theorem. It is independent of \cref{thm:main} and uses no computation.

\paragraph{Relation to previous work.}
Three ingredients come from Kriesell and Mohr \cite{KriesellMohr2022}: the doubled graphs $Z(B)$; their routing theorem for graphs $B$ on at most six vertices \cite[Theorem~6, Section~5, p.~12]{KriesellMohr2022}; and the doubled theta graph with completed root pairs, together with its rooted obstruction. We use the routing theorem only when the core of a witness consists of at most six double fibres. It is a statement about the graphs $Z(B)$, not about arbitrary six-colour Kempe routing, and the companion gives direct proofs of these cases. Twin deletion and parity-preserving merging also occur in recent work \cite{NeuenSeppelt2026,JimenezMooreQuirozYoo2026}; we make no priority claim for those general ideas. For the target $K_5$, Neuen and Seppelt prove that the class of $K_5$-minor-free graphs is closed under weak oddomorphisms \cite[Theorem~5.22 in arXiv version~2]{NeuenSeppelt2026}. Hence a weak oddomorphism to $K_5$ forces a $K_5$ minor whatever the order of the source, but this does not give the rooted conclusion of \cref{thm:main}. Our contribution is the mixed-fibre excess-six theorem with its rooted conclusion, the exclusion of an ordinary $K_8$ minor in the fifteen-vertex graph (\cref{prop:noK8}), and the resulting exact value of $m(n)$. The additional contribution of \cref{thm:rooted-four} is the unrestricted conclusion for the full original odd set, including witnesses with several odd vertices in a fibre. A comparison with the primary sources cited here found no earlier statement of the threshold $m(n)=n+7$. That comparison was limited in scope and is not a certification of novelty.

For comparison with \cref{thm:rooted-four}, Kriesell and Mohr prove a rooted $K_4$ theorem for a prescribed transversal of a four-colouring when each pair of transversal vertices is joined in the corresponding two-colour graph \cite[Theorem~4]{KriesellMohr2022}. If each fibre of an oddomorphism witness has exactly one odd vertex, handshaking in each two-colour component gives those paths, so that case was already covered. With several odd vertices in a fibre, the witness equations do not supply a compatible prescribed transversal. Our proof instead pairs all odd-degree endpoints in each two-colour graph and uses a cyclic parity identity. The cycle principle is recorded in Groenland, Nederlof and Koana \cite[Lemma~4]{GroenlandNederlofKoana2024}, and the crossing-linkage implication is due to Fabila-Monroy and Wood \cite[Lemma~7]{FabilaMonroyWood2013}; these ingredients are not claimed as new. The comparison with these sources did not locate the general original-set-rooted statement, but it is not an exhaustive certification of novelty.

For \cref{thm:small-fibre}, the new point is the interaction of witness parity, whole-colour deletion and original-root-preserving supports at three-separators, which makes the four-connected clean-root argument applicable to an arbitrary-order double. The three-terminal support criterion is classical: Fabila-Monroy and Wood \cite[Lemma~1]{FabilaMonroyWood2013} attribute it to Wood and Linusson. The rural-society theorem is due to Robertson and Seymour \cite{RobertsonSeymour1990}; we use the precise formulation in \cite[Theorem~2.5]{KNTWBounded2012}. Planarity preservation under weak oddomorphisms is Roberson's \cite[Lemma~8.2 in arXiv version~1]{Roberson2022}. These ingredients are not new, and the targeted comparison with these sources is not an exhaustive priority determination. The theorem is also distinct from the five-colour prescribed-transversal question of Kriesell and Mohr \cite[Section~7]{KriesellMohr2022}: it has six colours and a small fibre, and requires five branches to meet the odd set rather than assigning five specified roots to them.

Roberson's Question~7 asks whether, for every $t$, some complete target $K_N$ forces a $K_t$ minor under a weak oddomorphism \cite{Roberson2022}. Kristj\'ansson, Moore, Neuen, Roberson, Seppelt and Young answer this question affirmatively in their working paper \emph{A Proof of the Weak Roberson Conjecture} \cite[Theorem~1.3]{WeakRoberson2026}, posted on 14 October 2026. Thus the qualitative clique-target existence problem is solved by that result; it is no longer an open problem motivating the present paper. Their theorem gives a finite target threshold for each prescribed minor order without a source-order bound. It does not give the exact extremal source order $m(n)=n+7$ proved here, nor does its stated conclusion require branches to meet a witness's original odd set. The working-paper status of that source is retained in the citation; none of our proofs uses its theorem.

The unrestricted same-target questions for $K_6$ and $K_7$ are not settled here. For $n=5$ the closure theorem of Neuen and Seppelt mentioned above rules out every counterexample, and for $n\le4$ ordinary forcing is already known; \cref{thm:rooted-four} strengthens the four-colour conclusion to the original odd set. We do not prove that every weak oddomorphism to $K_7$ forces a $K_6$ minor, or that every weak oddomorphism to $K_9$ forces a $K_8$ minor.

\paragraph{Organisation.} \Cref{sec:normalization} sets up witnesses and a root-preserving normalization. \Cref{sec:excess5,sec:excess-six} prove \cref{thm:main}: the excess-five bound and the seven fibre patterns of excess six, including mixed fibres of sizes up to five. \Cref{sec:sharpness} constructs the fifteen-vertex graph, proves that it has no $K_8$ minor and deduces \cref{thm:sharp}. \Cref{sec:rooted-four} proves \cref{thm:rooted-four} by a root-preserving separator reduction and a cyclic parity argument. \Cref{sec:small-fibre} proves \cref{thm:small-fibre} and its consequences through a double-preserving reduction and a clean-root theorem.

\paragraph{The technical companion.} A separate companion document, whose section and statement numbers continue those of this paper, contains material that is not needed for \cref{thm:main,thm:sharp}: three extremal families found within a specified thirty-one-member family of root-apex graphs; a criterion for an even vertex attached to eight double fibres to force a $K_8$ minor; a classification of when the apex of the three hard root-apex graphs can remain a singleton branch after one extra edge is added; and a bounded computer-assisted theorem for weak oddomorphism witnesses to $K_9$ of order at most nineteen. For each statement it says whether the evidence is structural, certificate-backed or inherited.

\section{Parity witnesses and reductions}\label{sec:normalization}
All graphs are finite and simple. A $K_t$ model is a family of $t$ nonempty pairwise disjoint connected vertex sets with an edge between every pair. For a specified set $P$, a model is \emph{$P$-rooted} if each branch meets $P$.

\begin{definition}
An \emph{oddomorphism witness to $K_n$} consists of a graph $W$, a partition $V_1,\ldots,V_n$ into independent sets, and $p:V(W)\to\{0,1\}$ satisfying
\begin{align}
 |N_W(v)\cap V_j|&\equiv p(v)\pmod2 &&(v\in V_i,\ j\ne i),\label{eq:parity}\\
 \sum_{v\in V_i}p(v)&\equiv1\pmod2 &&(1\le i\le n).\label{eq:charge}
\end{align}
Its odd set is $P=p^{-1}(1)$. The fibres are nonempty by \eqref{eq:charge}. A vertex supports a fibre if it has a neighbour in that fibre.
\end{definition}
The fibre map is an oddomorphism in the usual sense. In the standard definition, a \emph{weak oddomorphism} from $F$ to $K_n$ is a homomorphism on all of $F$ whose restriction to a possibly noninduced subgraph $W$ is an oddomorphism \cite[Definition 3.9]{Roberson2022}. Our witness-containment assertions do not require extending the fibre map to the host. Every parity calculation concerns $W$, not any extra host edges.

\begin{lemma}[Root-preserving normalization]\label[lemma]{lem:normalization}
Every specified witness $W$ has a minor witness $H$ to the same target such that every even vertex supports at least two other fibres. There exist disjoint connected sets $C_x\subseteq V(W)$ for $x\in V(H)$ such that
\[
 p_H(x)=|C_x\cap P|\pmod2,
 \qquad xy\in E(H)\Longrightarrow E_W(C_x,C_y)\ne\varnothing.
\]
Consequently any model in $H$ rooted in $p_H^{-1}(1)$ lifts to a $P$-rooted model in $W$.
\end{lemma}
\begin{proof}
Initially use singleton supports. Delete any isolated even vertex. If an even vertex $v$ of colour $r$ has neighbours only in colour $i$, put $S=N_H(v)$. Its size is positive and even. Contract the star on $\{v\}\cup S$ to a vertex $w$ of colour $i$. For each outside vertex $z$, retain $wz$ precisely when $|N_H(z)\cap S|$ is odd, deleting any other resulting edge. Every retained edge is available after contraction. No outside vertex is adjacent to $v$, and no outside vertex of colour $i$ is adjacent to $S$, so the colouring stays proper. Set $p_H(w)=\sum_{u\in S}p_H(u)\pmod2$.

Each outside vertex has the same degree parity into the new colour-$i$ fibre. For $j\ne i$ the parity at the new vertex is
\[
 d_{V'_j}(w)\equiv\sum_{u\in S}d_{V_j}(u)-\ind_{j=r}|S|
 \equiv\sum_{u\in S}p_H(u)=p_H(w)\pmod2.
\]
The subtraction removes the star edges to $v$ and is even. The colour-$i$ charge is unchanged and colour $r$ loses only an even vertex, so both retain odd charge. Set $C_w=C_v\cup\bigcup_{u\in S}C_u$. This is connected, preserves the asserted parity, and remains disjoint from all other supports. Every retained edge has a witnessing original edge. Each step strictly decreases order, so the procedure terminates.

For a connected branch $A\subseteq V(H)$, the union $\bigcup_{x\in A}C_x$ is connected in $W$. These unions preserve disjointness and crossings. If $A$ contains an odd $x$, then $|C_x\cap P|$ is odd and hence positive. This proves the lift.
\end{proof}
A witness satisfying the support condition of \cref{lem:normalization} is called \emph{normalized}. This is a reduction property, not an assumption about arbitrary restrictions to fewer fibres.

\begin{lemma}[Restriction and attachment]\label[lemma]{lem:attachment}
Deleting whole fibres leaves a witness on the remaining fibres with the same restricted indicator. A singleton fibre $\{s\}$ has $p(s)=1$ and, for every $v\in V(W)\setminus\{s\}$, satisfies $sv\in E(W)$ if and only if $p(v)=1$. If a connected set $A$ contains a whole fibre $V_i$, then $A$ is adjacent to every disjoint branch containing an odd vertex outside $V_i$.
\end{lemma}
\begin{proof}
Whole-fibre deletion changes none of the retained equations \eqref{eq:parity}--\eqref{eq:charge}. The singleton assertion follows from \eqref{eq:charge} and, for $v\ne s$ (necessarily in another fibre), from \eqref{eq:parity}, because degrees into $\{s\}$ are zero or one. An odd vertex outside $V_i$ has positive odd degree into $V_i$, proving the final assertion.
\end{proof}
Thus a rooted clique model on the nonsingleton fibres extends by all singleton fibres. Removing singletons also preserves normalization, since an even vertex has no singleton neighbour. We refer to the restriction to nonsingleton fibres as the \emph{core}. Whole-fibre restriction is a special case of \cite[Lemma 5.2]{Roberson2022}; the support bookkeeping above is recorded to fix the original roots through contractions.

\begin{lemma}[Double interfaces]\label[lemma]{lem:double}
In a double fibre write $D_i=\{o_i,e_i\}$ with $p(o_i)=1$ and $p(e_i)=0$. Between two doubles exactly two interfaces are possible: the edge $o_io_j$, or the three edges $o_ie_j,e_ie_j,e_io_j$.
\end{lemma}
\begin{proof}
Writing $b=[e_ie_j\in E]$, the two even parity equations give $[o_ie_j]=[e_io_j]=b$, and the odd equation gives $[o_io_j]=1-b$. These exhaust all four entries.
\end{proof}
The \emph{active graph} $B$ has an edge $ij$ for the three-edge interface. For an all-double core, normalization is equivalent to $\delta(B)\ge2$.

\section{The excess five bound}\label{sec:excess5}
\begin{theorem}\label{thm:excess-five}
A specified witness to $K_n$ on at most $n+5$ vertices has a $P$-rooted $K_n$ model.
\end{theorem}
\begin{proof}
Normalize by \cref{lem:normalization} and remove singleton fibres by \cref{lem:attachment}. The following cases construct the required rooted core model.

Write $N=n$ and $r=|V(H)|-N=\sum_i(|V_i|-1)\le5$ after normalization. If any even
vertex exists, its own fibre is nonsingleton because it has odd total
charge; each of at least two colours that it sees contains at least two
vertices because its positive degree into that fibre is even. There are
therefore at least three nonsingleton fibres. Partitioning $r\le5$ into
at least three positive excesses gives precisely
\[
 (2,2,2),\ (2,2,2,2),\ (3,2,2),\ (2,2,2,2,2),\
 (3,2,2,2),\ (3,3,2),\ (4,2,2).
\]
If all vertices are odd, every fibre size is odd. The possible nonsingleton
patterns are empty, $(3)$, $(5)$, or $(3,3)$. A single larger fibre supplies
one odd vertex. For two triples, every vertex has positive odd degree into
the other triple, so there is an edge; its endpoints form the two rooted
core branches. These cases and the displayed list exhaust the domain.

In a double write the unique odd and even vertices as $o_i,e_i$.
For two doubles the four parity equations uniquely imply, with
$b_{ij}=[e_ie_j\in E]$,
\[
 [o_ie_j]=[e_io_j]=[e_ie_j]=b_{ij},\qquad
 [o_io_j]=1-b_{ij}.
\]
Call $ij$ active when $b_{ij}=1$, and let $B$ be the active graph. When all core
fibres are doubles, $e_i$ sees exactly the colours adjacent to $i$ in $B$;
normalization is exactly $\delta(B)\ge2$.

\paragraph{All-double cores}
Suppose that every core fibre is a double, so that there are $k\in\{3,4,5\}$ of them. By \cref{lem:double} the core is determined by its active graph $B$ on $k$ indices: an inactive pair contributes $o_io_j$ and an active pair contributes $o_ie_j,e_io_j,e_ie_j$. In the notation of Kriesell and Mohr, $Z(B)$ has vertices $(i,1),(i,2)$ and edges $(i,a)(j,b)$ for $ij\in E(B)$ and $(a,b)\neq(1,1)$; identifying $(i,1)$ with $o_i$ and $(i,2)$ with $e_i$, the core is $Z(B)$ together with the edges $o_io_j$ for the nonedges $ij$ of $B$. Their routing theorem \cite[Theorem~6, Section~5, p.~12]{KriesellMohr2022}, applied with the canonical colouring and transversal $\{(i,1)\}$, gives disjoint connected bags $Q_i\ni o_i$ that are adjacent whenever $ij\in E(B)$. For a nonedge $ij$ the edge $o_io_j$ joins $Q_i$ and $Q_j$. Hence the bags form a $K_k$ model whose every branch contains an odd vertex. This uses neither normalization nor any property of $B$; direct constructions are recorded in Appendix~G of the technical companion.

\paragraph{One larger fibre and two doubles}

Let $T$ be the larger fibre and let the doubles be $\{b,u\}$ and
$\{c,v\}$, with $b,c$ odd. Even vertices $u,v$ cannot see singletons and
must see both other core colours; hence the double interface is active,
so $uv,uc,vb$ are edges. Every even $y\in T$ must likewise see both
double colours, and consequently sees all of $b,u,c,v$. Each odd $x\in T$
sees exactly one of $b,u$ and exactly one of $c,v$.

The number of odd $x\in T$ having unequal bits $[xu],[xv]$ is even: modulo
two it equals $d_T(u)+d_T(v)$, since each even $y\in T$ contributes two.
But $T$ has an odd number of odd vertices. Therefore some odd $x$ has
equal bits, and sees either both $u,v$ or both $b,c$. The model
\[
 \{b,v\},\quad\{c,u\},\quad\{x\}
\]
has internal edges $bv,cu$, the crossing $uv$ between the first two
branches, and two crossings from $x$ in either equal-bit case. The three
roots are $b,c,x$.
No all-odd or one-odd restriction on $T$ has slipped into this argument.

\paragraph{Two triples and a double}

Write the double $\{o,e\}$. The even $e$ sees each triple in precisely
two vertices, because it must see both other core colours and has even
degree into a triple. The unique nonneighbour in either triple is odd.
Indeed, in a charge-three triple this is automatic; in a charge-one triple,
both even vertices must see both other core colours and therefore both
vertices of the double, making the unique odd vertex the nonneighbour.
Let the two omitted odd vertices be $a,b$. Both meet $o$ by odd parity
toward the double.

If $ab$ is an edge, $\{a\},\{b\},\{o\}$ is a rooted triangle. If not,
odd degree into the opposite triple gives an edge $az$ with $z\ne b$ and
an edge $bx$ with $x\ne a$. Both $x,z$ meet $e$. The branches
\[
 \{o\},\qquad\{a,z,e\},\qquad\{b,x\}
\]
are disjoint: $a,x$ are distinct vertices of one triple and $b,z$ distinct
vertices of the other. They are connected by $az,ze,bx$, and their three
crossings are $oa,ob,ex$. The roots are $o,a,b$.

\paragraph{One triple and three doubles: charge one}

Let $T=\{t,x,y\}$ with $t$ odd and $x,y$ even. Define
$U=\{i:xe_i\in E\}$ and $V=\{i:ye_i\in E\}$. An even triple vertex
sees either both members or neither member of each double, and
normalization gives $|U|,|V|\ge2$. Even parity of $e_i$ toward $T$ gives
\[
 [te_i]=\mathbf1_{i\in U\triangle V},\qquad
 [to_i]=1-\mathbf1_{i\in U\triangle V}.
\]
Also $S_i=N(e_i)\cap T$ is nonempty exactly for $i\in U\cup V$.
An even double vertex cannot have active degree zero: its only possible
other colour would then be $T$. Thus the active graph is a three-vertex
path or triangle, although most of the following models do not need a
specific active edge.

There are exactly four possibilities up to exchanging $x,y$ and relabelling
the indices. For each, the four branches below are disjoint; every branch
contains one of the distinct odd vertices $t,o_1,o_2,o_3$.
\begin{enumerate}
\item $U=V=\{1,2,3\}$. Use
 $\{t\},\{o_1\},\{o_2,x\},\{o_3,y\}$.
 The nonsingleton branches have internal edges $o_2x,o_3y$.
 Three crossings from $t$ are $to_1,to_2,to_3$; the other three are
 $o_1x,o_1y,o_2y$.
\item $U=V=\{l,m\}$, with omitted $k$. Since $S_k$ is empty, $e_k$
 must see both other double colours; hence $kl,km$ are active. Use
 $\{t\},\{o_k\},\{o_l,x,e_l\},\{o_m,y,e_m\}$.
 Internal paths are $o_l,x,e_l$ and $o_m,y,e_m$.
 Three $t$-crossings use the three odd double vertices. The $o_k$
 crossings use $o_ke_l,o_ke_m$. The remaining crossing is $e_le_m$
 when $lm$ is active, and $o_lo_m$ when it is inactive. This is the
 only normalization-dependent active-edge step in these four models.
\item $U=\{l,m\}$ and $V=\{k,l,m\}$. Use
 $\{t\},\{o_k,y,e_k\},\{o_l\},\{o_m,x\}$.
 Internal edges give path $o_k,y,e_k$ and edge $o_mx$.
 Three crossings from $t$ are $te_k,to_l,to_m$; the three others are
 $yo_l,yo_m,xo_l$.
\item $U=\{i,k\}$ and $V=\{j,k\}$ with distinct $i,j,k$. Use
 $\{t\},\{o_i,x,e_i\},\{o_j,y,e_j\},\{o_k\}$.
 The internal paths go through $x$ and $y$. The $t$ crossings are
 $te_i,te_j,to_k$; the $o_k$ crossings to the two larger branches are
 $o_kx,o_ky$; their remaining crossing is $e_ie_j$ if $ij$ is active
 and $o_io_j$ otherwise.
\end{enumerate}
Equal subsets of a three-set have common size two or three; unequal
subsets of size at least two have sizes two and three or are distinct
two-subsets. Thus this is a proof of exhaustion, not an enumeration
assumption.

\paragraph{One triple and three doubles: charge three, active triangle}

Let $T=\{a,b,c\}$ be all odd. Every $S_i=N(e_i)\cap T$ is empty or a
two-set. If nonempty, write $q_i$ for its omitted vertex: $o_i$ meets
exactly $q_i$ in $T$. If empty, $o_i$ meets all of $T$.

If some $t\in T$ meets all three $o_i$, or meets all three $e_i$, it is
an apex for the rooted triangle
$\{o_1,e_2\},\{o_2,e_3\},\{o_3,e_1\}$; all six crossings for the
resulting $K_4$ are explicit. To prove completeness of this criterion,
let $h$ be the number of nonempty supports. For $h\le1$, an omitted
vertex meets all $o_i$ (with arbitrary choice if $h=0$). For $h=2$, the
same is true if both omissions agree, and different omissions are the
first exception. For $h=3$, equal omissions give the odd-apex case;
exactly two distinct omissions leave a third vertex contained in all
three supports, giving the even-apex case; three distinct omissions are
the second exception.

In the first exception, label $q_1=a,q_2=b,S_3=\varnothing$ and use
\[
 \{c\},\quad\{o_1,a,e_2\},\quad\{o_2,b,e_1\},\quad\{o_3\}.
\]
The nonsingleton paths are $o_1,a,e_2$ and $o_2,b,e_1$. Their mutual
crossing is $e_2e_1$. Three $c$ crossings are $ce_2,ce_1,co_3$;
the remaining two $o_3$ crossings are $o_3a,o_3b$.
In the second exception, label $q_1=a,q_2=b,q_3=c$ and use
\[
 \{a\},\quad\{o_1\},\quad\{b,e_3\},\quad\{c,e_2\}.
\]
Internal edges are $be_3,ce_2$; the six crossings in branch order are
$ao_1,ae_3,ae_2,o_1e_3,o_1e_2,e_3e_2$.
All roots are original current odd vertices, including the triple
vertices that replace some double roots.

\paragraph{One triple and three doubles: charge three, active path}

Label the path $1,2,3$. Normalization forces $S_1,S_3$ to be nonempty
because the two endpoint even vertices already see just one other double
colour. Their omissions $q_1,q_3$ are therefore defined; $S_2$ may be
empty. If $q_1=q_3=a$, use
\[
 \{o_1\},\quad\{o_3\},\quad\{a\},\quad\{o_2,e_1,e_2\}.
\]
The last branch is connected by $o_2e_1,e_1e_2$. The first three are
pairwise adjacent by $o_1o_3,o_1a,o_3a$. The last branch meets $o_1,o_3$
through $e_2$, and meets $a$ through exactly one of $o_2,e_2$, because
$a$ is odd. This works for every centre support.

Otherwise set $q_1=a,q_3=b$, with third triple vertex $c$.
The empty centre and its three possible omissions give all remaining
cases:
\begin{enumerate}
\item If $S_2=\varnothing$, use
 $\{o_1\},\{o_3\},\{a,o_2,b\},\{c,e_1,e_2\}$.
 The larger branches have paths $a,o_2,b$ and $c,e_1,e_2$.
 Their crossing is $o_2e_1$; the endpoint singletons meet each other
 through $o_1o_3$, meet the first larger branch through $a,b$, and
 meet the second through $e_2$.
\item If $q_2=a$, use
 $\{o_1\},\{o_3\},\{a,o_2,e_1,b\},\{c,e_2,e_3\}$.
 The first larger branch has path $a,o_2,e_1,b$, and the second has
 edges $ce_2,ce_3$. The endpoint crossings are exactly as in the
 previous case. The two larger branches meet through $e_1e_2$.
 The case $q_2=b$ is literally the same construction after exchanging
 endpoints $1,3$ and vertices $a,b$; no support assumption changes.
\item If $q_2=c$, use
 $\{c\},\{b,e_1\},\{a,e_2,e_3\},\{o_2\}$.
 Internal edges are $be_1,ae_2,ae_3$. The six crossings in branch
 order are $ce_1,ce_3,co_2,e_1e_2,e_1o_2,e_3o_2$.
\end{enumerate}
Every branch contains a current odd vertex, and the displayed branch
sets in each model are disjoint.

The cases exhaust the normalized core. Attach the singleton fibres and lift the models through their connected supports. This proves the theorem.
\end{proof}

\section{The excess-six theorem}
\label{sec:excess-six}

We prove the positive statement with respect to the odd set of a specified
witness. In particular, the root condition survives passage from a weak
oddomorphism witness to its host graph.

\begin{theorem}[Excess six]
\label{thm:excess-six}
Let $n\geq2$, and let $W\subseteq F$ be a specified oddomorphism witness
for $K_n$, with specified odd set $P\subseteq V(W)$. If
$|V(W)|\leq n+6$, then $W$ (and hence $F$) has a $K_n$ minor whose
every branch set meets $P$. The branch sets are not required to respect the fibres.
\end{theorem}

We use the support-tracked normalization lemma and the original-rooted
excess-five theorem established earlier. The proof below supplies every
additional fibre-pattern case at excess six. The only further general
minor theorem used is the classical fact that a finite simple graph of
minimum degree at least three contains a subdivision of $K_4$.

\subsection{Interfaces and the reduction to seven patterns}
\label{subsec:six-interfaces}
Throughout the local arguments, a normalized witness has independent
fibres $V_i$ and an indicator $p:V(W)\to\{0,1\}$ satisfying
\[
 |N(v)\cap V_j|\equiv p(v)\pmod2\quad(j\ne i, v\in V_i),
 \qquad \sum_{v\in V_i}p(v)\equiv1\pmod2.
\]
Normalized means that every even vertex has neighbours in at least two
other fibres. Write a double fibre as $D_i=\{o_i,e_i\}$, with $o_i$
odd. An outside odd vertex has exactly one neighbour in $D_i$; an
outside even vertex sees both members or neither. For two doubles,
solving the four parity equations gives precisely
\begin{equation}
\label{eq:six-double-interface}
 \begin{array}{c|cc}
  &o_j&e_j\\\hline
 o_i&1-b_{ij}&b_{ij}\\
 e_i&b_{ij}&b_{ij}
 \end{array},\qquad b_{ij}\in\{0,1\}.
\end{equation}
The \emph{active graph} $B$ has edge $ij$ when $b_{ij}=1$.
An active pair therefore induces the path $o_i-e_j-e_i-o_j$.

We shall repeatedly use the following absorption observation. If a
connected branch contains a whole fibre $V_i$, then it is adjacent to
any disjoint branch containing an odd vertex: that odd vertex lies
outside $V_i$ and has positive odd degree into $V_i$. Thus, after
connectivity and disjointness are checked, a whole-fibre branch needs
no further adjacency check against odd-rooted branches.

We also use same-fibre twin deletion (cf.~Neuen and Seppelt~\cite[Lemma~5.28 in arXiv version~2]{NeuenSeppelt2026}). Two vertices $u,u'$ of one fibre with equal
open neighbourhoods have the same indicator, since $n\geq2$ and their
degrees into another fibre agree. Deleting both changes every relevant
degree by zero or two and changes their fibre charge by $2p(u)$.
All fibre charges remain odd, so no fibre disappears. The resulting
witness has its original odd set restricted to the surviving vertices.
In particular, if two vertices can be deleted from an excess-six
witness this way, the excess-five theorem applies without requiring
that the resulting witness still be normalized.

After normalization, discard singleton fibres temporarily. A singleton
vertex is odd and is adjacent exactly to the odd vertices outside its
fibre. Singleton vertices therefore extend any odd-rooted core model.
They have no even neighbours, so normalization is preserved in the
core. If the remaining excess is at most five, the preceding theorem
already applies.

Suppose that the remaining excess is six. If every vertex is odd,
the nonsingleton sizes are $(7)$, $(5,3)$ or $(3,3,3)$. One vertex,
an edge, or a cycle partitioned into three consecutive nonempty pieces,
respectively, supplies the required core model. In the last case every
vertex has a neighbour in each of the other two fibres, so a cycle exists.
If an even vertex remains, its own fibre and two other fibres are
nonsingletons. The partitions of six into at least three positive parts
are exactly
\[
 1^6,\quad2+1^4,\quad2+2+1+1,\quad3+1+1+1,\quad
 2+2+2,\quad3+2+1,\quad4+1+1.
\]
Consequently the remaining sizes are
\begin{equation}
\label{eq:six-patterns}
 (2^6),\quad(3,2^4),\quad(3,3,2,2),\quad(4,2,2,2),\quad
 (3,3,3),\quad(4,3,2),\quad(5,2,2).
\end{equation}
The following lemmas cover all seven patterns.

\subsection{Six doubles}
\label{subsec:six-doubles}

\begin{lemma}
\label[lemma]{lem:six-doubles}
A normalized witness with six double fibres has an odd-rooted $K_6$ minor.
\end{lemma}
\begin{proof}
This is the case of six doubles of Kriesell and Mohr's routing theorem \cite[Theorem~6, Section~5, p.~12]{KriesellMohr2022}, exactly as in the paragraph on all-double cores in the proof of \cref{thm:excess-five}: the witness is $Z(B)$ for its active graph $B$ on six indices together with the odd--odd edges of the inactive pairs, their theorem gives disjoint connected bags $Q_i\ni o_i$ adjacent for every active pair, and the odd--odd edges join the remaining pairs. Normalization is not needed. A direct proof for the normalized case, which does not use \cite{KriesellMohr2022}, is given in Appendix~G of the technical companion.
\end{proof}

\subsection{Three nonsingleton fibres}
\label{subsec:six-three-fibres}
In a normalized three-fibre core, every even vertex sees both other
fibres. In particular it sees both members of a double and exactly
two members of a triple.

\begin{lemma}
\label[lemma]{lem:six-three-fibres}
The patterns $(5,2,2)$, $(4,3,2)$ and $(3,3,3)$ have odd-rooted core
clique minors.
\end{lemma}
\begin{proof}
\emph{Pattern $(5,2,2)$.} In fact the argument works for any odd-charge
fibre $T$ and two doubles $\{b,u\},\{c,v\}$, with $b,c$ odd.
Normalization forces their interface active. Every even vertex of
$T$ sees $u,v$. The sum of the two even degrees of $u,v$ into $T$
shows that an even number of odd vertices of $T$ have unequal
adjacency bits to $u,v$. Since the number of odd vertices is odd,
some odd $t$ has equal bits. Thus $t$ sees either $u,v$ or $b,c$,
and
\[
 \{b,v\},\quad\{c,u\},\quad\{t\}
\]
is the required model, with the first two branches joined by $uv$.

\emph{Pattern $(4,3,2)$.} Write the fibres as $A,T,D=\{o,e\}$.
Every even vertex of $A\cup T$ connects the whole double $D$.
If there is an odd--odd edge $at$ between $A,T$, choose any even
$g\in A$ and use $\{a\},\{t\},D\cup\{g\}$.

Otherwise $T=\{t,x,y\}$ with unique odd vertex $t$; each odd
vertex of $A$ sees exactly one of $x,y$. If $A$ has three odd
vertices and one even $g$, then $tg$ is forced, and $g$ sees one
of $x,y$, say $x$. Even parity of $x$ into $A$ forces an odd
neighbour $a$. Use
\[
 D\cup\{y\},\quad\{t,g\},\quad\{a,x\},
\]
whose last two branches meet by $gx$. If instead
$A=\{a,g_1,g_2,g_3\}$ with $a$ odd, let $x$ be the even
neighbour of $a$ in $T$. Count the $g_i$-neighbourhoods
$\{t,x\},\{t,y\},\{x,y\}$ by $h,k,\ell$. We have
$h+k+\ell=3$, $h+k$ odd and $1+h+\ell$ even. If $h=0$,
both $k,\ell$ are odd, a contradiction. Choose $g_i$ seeing
$t,x$, and use $D\cup\{y\},\{a,x\},\{t,g_i\}$.

\emph{Pattern $(3,3,3)$ with at least one all-odd triple.}
The all-odd case was handled above. If $T_0$ is all odd and
$T_j=\{o_j,x_j,y_j\}$ has one odd vertex, the two two-element
neighbourhoods of $x_j,y_j$ in $T_0$ intersect in a vertex $v_j$.
Odd parity forces $v_jo_j$ as well, so $T_j\cup\{v_j\}$ is
connected. With two all-odd triples, choose an odd edge between
one unused vertex of $T_0$ and the other all-odd triple; these two
singletons and $T_j\cup\{v_j\}$ give the model.

With just one all-odd triple $T_0$, choose such connectors $v_1,v_2$.
If some $u\in T_0\setminus\{v_1\}$ sees $o_2$, use
$\{u\},\{o_2\},T_1\cup\{v_1\}$; also allow the symmetric
choice. If neither exists, positive odd degree gives
$N(o_2)\cap T_0=\{v_1\}$ and $N(o_1)\cap T_0=\{v_2\}$.
Since $v_1o_1$ is an edge, $v_1=v_2=:v$. An edge $o_1o_2$
now gives a triangle. Otherwise the two two-element sets
$N(x_2)\cap T_1,N(y_2)\cap T_1$ must differ: if equal, their
omitted vertex cannot be even, since an even vertex needs two
neighbours in $T_2$; if it is $o_1$, odd parity forces the excluded
edge $o_1o_2$. The distinct sets intersect and cover $T_1$,
so $T_1\cup\{x_2,y_2\},\{v\},\{o_2\}$ is a model.

\emph{Pattern $(3,3,3)$ with one odd vertex per triple.}
Write $T_i=\{o_i,x_i,y_i\}$. Every even row and column of an
interface matrix has two ones. If its odd--odd entry is one, either
it is the only odd-incident edge and the even--even block is complete,
or the odd row and column are full and the even--even block is a
perfect matching. If the odd--odd entry is zero, each odd vertex has
a unique even neighbour in the opposite triple. Denote the neighbour
of $o_i$ in $T_j$ by $c_{ij}$. All four even--even edges occur
except $c_{ji}c_{ij}$. These are all possibilities, directly from
row and column parity.

If $o_io_j$ is an edge, and two even vertices of $T_i\cup T_j$
see different two-subsets of the third triple $T_k$, then these two
distinct two-subsets intersect and together cover $T_k$. The two stars
at these even vertices therefore form a connected set containing the
whole fibre $T_k$. Use this set as one branch, with $\{o_i\}$ and
$\{o_j\}$ as the other two branches; $o_io_j$ joins the latter, and
absorption gives the two remaining adjacencies. If all four see the same two-subset, the
omitted vertex must be $o_k$ (an omitted even vertex would have
at most one neighbour in $T_i$). Odd parity then forces
$o_ko_i,o_ko_j$, giving a triangle.

Hence suppose all odd--odd edges are absent. In each triple the
two incoming pointers either agree, denoted E, or differ, denoted D.
Swapping the two even vertices independently and permuting triples
leaves four cases, according to the number of D's. A bar denotes
the other even vertex of the same triple.
\begin{itemize}
\item With no D's, let the common pointer in $T_i$ be $p_i$.
Use $\{o_0\},\{o_1,p_2\},\{o_2,p_0,p_1\}$.
The crossings are $o_0p_2,o_0p_1,o_1p_0$.
\item With three D's, use
$\{o_0,c_{01}\},\{o_1,c_{12}\},\{o_2,c_{20}\}$.
The crossings $c_{01}c_{12},c_{12}c_{20},c_{20}c_{01}$ avoid
the unique forbidden pair in each interface because its incoming
pointers differ.
\item With only $T_0$ of type D, put
$\alpha=c_{10}\ne\beta=c_{20}$ and let $p_1,p_2$ be the common
pointers elsewhere. Use
\[
 \{o_0\},\quad\{o_1,\alpha,\bar p_2,p_1\},\quad
 \{o_2,\beta,\bar p_1,p_2\}.
\]
The latter two contain the paths
$o_1-\alpha-\bar p_2-p_1$ and
$o_2-\beta-\bar p_1-p_2$; their crossings with each other and
the singleton are $\alpha\bar p_1,o_0p_1,o_0p_2$.
\item With only $T_2$ of type E, put
$c_{02}=c_{12}=p_2$, $\beta=c_{20}\ne c_{10}$ and
$\gamma=c_{01}\ne c_{21}$. Use
$\{o_0\},\{o_1,p_2\},\{o_2,\beta,\gamma\}$.
The third branch contains $o_2-\beta-\gamma$, and the crossings
are $o_0p_2,o_0\gamma,p_2\gamma$.
\end{itemize}
Each listed even--even edge differs from the sole forbidden pointer
pair of its interface. The branches are disjoint and contain the
three distinct odd vertices, completing the proof.
\end{proof}

\subsection{One four-vertex fibre and three doubles}
\label{subsec:six-four-triple-double}
\begin{lemma}
\label[lemma]{lem:six-four-triple-double}
The pattern $(4,2,2,2)$ has an odd-rooted $K_4$ minor.
\end{lemma}
\begin{proof}
If the four-fibre is $A=\{a,g_1,g_2,g_3\}$ with unique odd
vertex $a$, each $g_j$ sees at least two doubles. If every double
is seen, the incidence graph matches the three doubles to distinct
$g_j$: a pair of doubles together sees all three $g_j$, so Hall's
condition holds. The three resulting whole-double branches and
$\{a\}$ give the model. If a double $D_3$ is unseen, all $g_j$
see $D_1,D_2$. Even parity at $e_3$ forces $ae_3$ absent and
hence $ao_3$ present. Use
$\{a\},\{o_3\},D_1\cup\{g_1\},D_2\cup\{g_2\}$.

Now write $A=\{a_1,a_2,a_3,g\}$ with $g$ even, and let $S$
be the double indices seen by $g$, so $|S|\geq2$. Put
$n(i)=|N(o_i)\cap\{a_1,a_2,a_3\}|$. Odd parity makes $n(i)$
even on $S$ and odd outside it. For $i\in S$, with remaining
indices $j,k$, if $jk$ is active and some odd $a$ sees $o_k$,
use
\[
 D_i\cup\{g\},\quad D_j\cup\{e_k\},\quad\{a\},\quad\{o_k\}.
\]
Absorption handles all but the singleton crossing $ao_k$.
Allowing exchange of $j,k$, failure means $jk$ is inactive or
$n(j)=n(k)=0$.

Suppose this fails for every $i\in S$. If $S=\{1,2\}$, the
positive odd number $n(3)$ forces $13,23$ inactive. Then $e_3$
sees at most the four-fibre, contradicting normalization. Thus $S$
contains all three indices. Every $e_i$ sees $g$, so normalization
requires an active double neighbour. Every active $jk$ forces
$n(j)=n(k)=0$, and therefore all three numbers vanish. Each $a_h$
sees exactly $e_1,e_2,e_3$ in the core, and all singleton vertices
outside the core. The three $a_h$ are consequently twins in the
full normalized graph. Delete two and apply the original-rooted
excess-five theorem, since the excess drops from six to four.
\end{proof}

\subsection{Two triples and two doubles}
\label{subsec:six-two-triples}
\begin{lemma}
\label[lemma]{lem:six-two-triples}
A normalized $(3,3,2,2)$ witness has an odd-rooted $K_4$ minor.
\end{lemma}
\begin{proof}
Write the triples as $A,B$ and the doubles as
$D=\{d_0,d_1\},E=\{e_0,e_1\}$, with subscript zero odd.
We distinguish the odd charges of the two triples.

\paragraph{Both triples have charge one.}
Let $a,b$ be their odd vertices. Suppose first that $ab$ is an edge.
Let $C_D,C_E$ be the sets of even triple vertices seeing the respective
doubles. Every even triple vertex sees at least one double, so their
union has size four. Each set is nonempty: normalization forces the
even vertex of its double to meet a triple, in two vertices, at least
one of which is even. Thus choose distinct $x\in C_D,y\in C_E$.
The branches $\{a\},\{b\},D\cup\{x\},E\cup\{y\}$ work by
absorption.

Suppose $ab$ is absent. Let $\alpha\in A$, $\beta\in B$ be
the unique neighbours of $b,a$ respectively, and let
$\alpha',\beta'$ be the remaining vertices. The interface is
\[
 \begin{array}{c|ccc}
 &b&\beta&\beta'\\\hline
 a&0&1&0\\
 \alpha&1&t&1-t\\
 \alpha'&0&1-t&1-t
 \end{array},\qquad t\in\{0,1\}.
\]
For $t=1$ it is a four-vertex path and two isolated vertices.
The latter vertices $\alpha',\beta'$ must each see both doubles;
use $\{a,\beta\},\{b,\alpha\},D\cup\{\alpha'\},
E\cup\{\beta'\}$.

For $t=0$, name the six-vertex path
\[
 a-y-z-w-x-b,
\]
with $A=\{a,z,x\}$ and $B=\{b,y,w\}$. For an internal even
vertex $h$, let $S_h\subseteq\{D,E\}$ record the doubles it sees.
Every $S_h$ is nonempty. If a triple has odd vertex $r$ and even
vertices $u,v$, parity at $f_1$ and $f_0$ of a double $F$ gives
\begin{equation}
\label{eq:six-endpoint-bits}
 [rf_1]=[F\in S_u]\mathbin\oplus[F\in S_v],\qquad
 [rf_0]=1\mathbin\oplus[F\in S_u]\mathbin\oplus[F\in S_v].
\end{equation}
If the double interface is active, use
\[
 \{a,y,z\},\quad\{b,x,w\},\quad\{d_0,e_1\},\quad\{e_0,d_1\}.
\]
The first pair meet by $zw$, the last pair by $d_1e_1$.
Each internal path vertex sees both vertices of some double, which
are split between the last two branches; $y$ and $x$ therefore
supply all four other crossings.

Suppose the double interface is inactive. Normalization of $d_1,e_1$
forces both to see both triples. Equivalently,
\[
 S_z\cup S_x=S_y\cup S_w=\{D,E\}.
\]
Number the path positions $0,\ldots,5$. For each odd double vertex
$f_0$, let $f_F,\ell_F$ be its first and last neighbour positions.
It has supported internal neighbours in both triples, so
$f_F<\ell_F$. If
\[
 \max(f_D,f_E)<k\leq\min(\ell_D,\ell_E)
\]
for an integer $k$, cut the path just before position $k$ and use
its prefix and suffix with $\{d_0\},\{e_0\}$. Each singleton
sees both path pieces, and $d_0e_0$ supplies their crossing.

If there is no such cut, exchange $D,E$ so
$\ell_E\leq f_D$. The mandatory support in both triples gives
$2\leq\ell_E\leq f_D\leq3$. Equality at two would force both
$z,x$ to support $D$, giving $ad_0$ by
\eqref{eq:six-endpoint-bits}, contrary to $f_D=2$. Equality at three similarly gives $be_0$, contrary to $\ell_E=3$.
Consequently $\ell_E=2$ and $f_D=3$: the odd vertex $e_0$ has no
neighbour among $w,x$, and $d_0$ has none among $y,z$. An internal even
vertex $h$ is adjacent to $f_0$ exactly when $F\in S_h$, and every
$S_h$ is nonempty, so the internal supports are exactly
\[
 S_y=S_z=\{E\},\qquad S_w=S_x=\{D\}.
\]
Use
\[
 U=\{a,x,y,d_1\},\quad V=\{b,z,w,e_1\},\quad
 \{d_0\},\quad\{e_0\}.
\]
The first two are connected by $ay,ad_1,d_1x$ and
$be_1,e_1z,zw$. Their six crossings, in the displayed branch order,
are $xb,xd_0,ye_0,wd_0,ze_0,d_0e_0$.

\paragraph{Exactly one triple has charge three.}
Let $A$ be all odd and $B=\{b,x,y\}$ have unique odd vertex $b$.
Let $T_D,T_E\subseteq\{x,y\}$ be the double-connector sets.
Their union is $\{x,y\}$ by normalization. If both are nonempty,
choose distinct representatives $s\in T_D,t\in T_E$ and any
$a\in N_A(b)$. The branches $D\cup\{s\},E\cup\{t\},
\{a\},\{b\}$ work.

Otherwise exchange doubles so $T_D=\varnothing$ and
$T_E=\{x,y\}$. The exact triple--double blocks are
\[
 E(B,D)=\{bd_0\},\qquad
 E(B,E)=\{be_0,xe_0,xe_1,ye_0,ye_1\}.
\]
Both $x,y$ have exactly two neighbours in $A$, and the double
interface must be active, since otherwise $d_1$ sees at most $A$.
If $b,d_0$ have a common neighbour $a\in A$, the triangle
$a,b,d_0$ together with $E\cup\{x\}$ gives the model.

Otherwise their nonempty odd neighbourhoods in $A$ are disjoint,
so both are singletons, say $N_A(b)=\{a\}$ and
$N_A(d_0)=\{f\}$; let $c$ be the third vertex of $A$.
The edges $ad_1,cd_1,fd_0$ are forced. Counting the four edges
from $x,y$ into $A$, odd parity gives, after exchanging $x,y$,
\[
 N_A(x)=\{a,c\},\qquad N_A(y)=\{a,f\}.
\]
Now use
\[
 \{a,b\},\quad\{c,x,e_0\},\quad\{f,y\},\quad
 \{d_0,d_1,e_1\}.
\]
They are connected by $ab$, $c-x-e_0$, $fy$, $d_0-e_1-d_1$.
Their six crossings are
\[ ax,\quad ay,\quad ad_1,\quad e_0y,\quad cd_1,\quad fd_0. \]
Every branch contains an odd vertex. Exchanging $A,B$ also covers
charge pair $(1,3)$.

\paragraph{Both triples have charge three.}
The $A$--$B$ matrix has all row and column sums odd. It is therefore
one of a perfect matching, a five-edge cross, or $K_{3,3}$: if a
row is full, the other two rows agree, with common support of size
one or three; without a full row it is a permutation matrix.
Put
\[
 S=N(d_1)\cap(A\cup B),\qquad T=N(e_1)\cap(A\cup B).
\]
Each meets each triple in zero or two vertices. If the double block
is inactive, normalization forces two in each triple; if active,
it forces only $S,T\ne\varnothing$.

We use Dirac's classical theorem that every finite simple graph of
minimum degree at least three contains a $K_4$ subdivision; see,
for example, the statement in Diwan~\cite{Diwan2020}.
In each use below, we contract disjoint connected bags, every bag
containing an original odd vertex, and verify minimum degree in
the resulting \emph{simple} quotient. A $K_4$ subdivision in that
quotient gives four connected branch sets: assign all internal
vertices of each subdivision path to one of its ends. Expanding
the bags gives an odd-rooted model in the original graph.

For the complete interface, the direct model
$\{a_0,b_0\},\{a_1,b_1\},\{a_2\},\{b_2\}$ suffices.

\emph{The matching interface.} Write its edges as $a_i b_i$.
Every triple vertex has exactly three neighbours, its mate and one
vertex of each double. In the inactive case choose
$x\in S\cap A$ and $y\in T\cap B$ as the two ends of one matching edge;
this is possible because the two corresponding two-element index sets
intersect. Contract
$d_1x,e_1y$. An unabsorbed triple vertex loses a distinct neighbour
only if its mate and one double neighbour were merged, but the only
vertices whose mates are anchors are $x,y$, themselves absorbed.
The odd double vertices keep three distinct neighbours. Each new
bag sees the three remaining vertices of its four-element support
in distinct bags. Thus the simple quotient has minimum degree three.

In the active case, if $d_0$ sees both ends $x,y$ of a matching
edge, use
\[ \{x\},\quad\{y\},\quad\{d_0\},\quad\{d_1,e_0,e_1\}. \]
The last branch is connected, sees both endpoints through their
$E$-neighbours, and sees $d_0$ through $e_1$. Treat $e_0$
symmetrically. If neither such triangle exists, neither $S$ nor
$T$ is confined to one triple: otherwise the corresponding odd
double vertex sees all of the other triple and the one omitted
vertex in the first, giving a matching pair. Consequently
\[
 S=(A\cup B)\setminus\{a_i,b_j\},\quad
 T=(A\cup B)\setminus\{a_k,b_\ell\},\quad i\ne j, k\ne\ell.
\]
Choose distinct $x,y\in S\cap T$ as follows:
\[
\begin{array}{c|cc}
\text{condition}&x&y\\\hline
j\ne k&a_j&b_k\\
j=k, i\ne\ell&b_i&a_\ell\\
j=k, i=\ell&a_t&b_t\quad(t\notin\{i,j\}).
\end{array}
\]
In the first two rows the mate of $x$ lies outside $S$ and the
mate of $y$ outside $T$; in the last row $x,y$ are mates.
Contract $d_1x,e_1y$. These conditions prevent any unabsorbed
triple vertex from having its mate merged with its double neighbour.
Also $y\in S$ ensures that the three neighbours $a_i,b_j,e_1$
of $d_0$ have distinct images, and $x\in T$ similarly protects
$e_0$. Each merged bag sees the three other support vertices in
distinct bags. Again the simple quotient has minimum degree three.

\emph{The cross interface.} Write its centres as $a_0,b_0$.
In the inactive case choose
$x\in S\cap\{a_1,a_2\}$ and $y\in T\cap\{b_1,b_2\}$,
and contract $d_1x,e_1y$. An unabsorbed leaf keeps its centre
neighbour and two distinct double-neighbour images. Each centre
started with degree five and loses at most one neighbour per
contraction, leaving degree at least three. The odd double vertices
retain three distinct neighbours, and each merged bag sees the three
other support vertices in distinct bags. Apply the same theorem.

In the active case put $U=\{d_0,e_1\}$ and $V=\{e_0,d_1\}$.
They are connected, odd-rooted and adjacent. A triple vertex sees
both precisely when its membership bits in $S,T$ agree; call such
a vertex balanced. Each triple has an even number of unequal bits,
so has a balanced vertex. If a leaf $z$ is balanced, use
\[
 \{z\},\quad(A\cup B)\setminus\{z\},\quad U,\quad V.
\]
Deleting a leaf leaves a connected tree, containing a balanced
vertex of the opposite triple, so both first branches see $U,V$.
If no leaf is balanced, both centres are balanced and
$\{a_0\},\{b_0\},U,V$ works. This completes all charge pairs.
\end{proof}

\subsection{One triple and four doubles}
\label{subsec:six-one-triple}
\begin{lemma}
\label[lemma]{lem:six-one-triple}
A normalized $(3,2,2,2,2)$ witness has an odd-rooted $K_5$ minor.
\end{lemma}
\begin{proof}
Write the triple as $T$, and let $B$ be the active graph on the
four double indices. Put
\[
 S_i=N(e_i)\cap T,\qquad s_i(t)=[t\in S_i].
\]
Even parity gives $|S_i|\in\{0,2\}$, and the whole triple--double
interface is determined by
\begin{equation}
\label{eq:six-triple-interface}
 [te_i]=s_i(t),\qquad [to_i]=p(t)\mathbin\oplus s_i(t).
\end{equation}
Conversely these rules satisfy every interface parity, since the
charge of $T$ is odd. Normalization is exactly
\begin{equation}
\label{eq:six-one-triple-normalization}
 d_B(i)+[S_i\ne\varnothing]\geq2,\qquad
 |\{i:t\in S_i\}|\geq2\quad\text{for even }t\in T.
\end{equation}
In particular $B$ has no isolated vertex. Its seven possible
isomorphism types are
\[
 2K_2,\quad P_4,\quad K_{1,3},\quad C_4,\quad
 \text{paw},\quad K_4-e,\quad K_4.
\]
Indeed two edges force $2K_2$; three give one of the two spanning
trees; four are the complements of two disjoint or incident edges;
five and six give the last two types. We handle all seven, allowing
empty supports whenever \eqref{eq:six-one-triple-normalization} permits.

For an active pair $ij$, put $Q_{ij}=o_i-e_j-e_i-o_j$.
An odd $t\in T$ has one neighbour in each of $D_i,D_j$.
Splitting $Q_{ij}$ between these two neighbours produces two
connected branches, with odd roots $o_i,o_j$, both adjacent to $t$.
Numbering path positions $0,1,2,3$, the permitted cuts are
\[
\begin{array}{c|c|c}
(s_i(t),s_j(t))&\text{neighbour positions}&\text{cut after position}\\\hline
00&0,3&0,1,2\\
01&0,1&0\\
10&2,3&2\\
11&2,1&1.
\end{array}
\]
We refer to this as a path split. When we give five branches
$A_1,A_2,A_3,A_4,R$, a crossing list $W$ has pair order
$12,13,14,23,24,34$, and $V$ has order $1R,2R,3R,4R$.
Thus the two lists explicitly certify all ten adjacencies.

\paragraph{Matching and complete active graphs.}
For $B=2K_2$ with edges $12,34$, choose any odd $t$, split
$Q_{12},Q_{34}$, and add $R=\{t\}$. The four crossings between
paths are the inactive edges $o_io_j$ for $i\in\{1,2\}$,
$j\in\{3,4\}$.

For $B=K_4$, choose any odd $t$. Among its four bits there is an
equal pair; pair those indices, and pair the remaining two.
Split the two corresponding paths, using the middle cut for an
equal pair. Each equal pair gives two branches containing even
vertices. An unequal pair gives a singleton and a branch containing
both even vertices. Thus there is at most one singleton. Branches
from different paths see one another through even--even edges,
or through an odd--even edge if one is that singleton.

\paragraph{The path $1-2-3-4$.}
If an odd $t$ is omitted by an endpoint support, reflect so
$t\notin S_1$. Use $A_1=\{o_1\}$,
$A_2=\{o_2,e_1,e_2\}$, split $Q_{34}$ into $A_3,A_4$,
and take $R=\{t\}$. All apex and cut-edge crossings hold.
The other crossings are the inactive edges $o_1o_3,o_1o_4,o_2o_4$
and the active edge $e_2o_3$. This always covers an all-odd triple,
since $|S_1|\leq2$.

Otherwise $T=\{r,x,y\}$ has unique odd vertex $r$, and both
endpoint supports contain $r$. Relabel so $S_1=\{r,x\}$.
If $S_4=\{r,x\}$, use
\[
 (A_1,A_2,A_3,A_4,R)=
 (\{o_1,x\},\{o_2,e_1\},\{o_3,e_2\},\{o_4,e_3\},\{r,e_4\}).
\]
All five are edges, and their crossings are
\[
\begin{split}
 W&=(xe_1,o_1o_3,o_1o_4,e_1e_2,o_2o_4,e_2e_3),\\
 V&=(xe_4,e_1r,o_3e_4,e_3e_4).
\end{split}
\]
If $S_4=\{r,y\}$, use
\[
 (\{o_1,e_1,x\},\{o_2,e_3\},\{o_3,e_2\},\{o_4,y\},\{r,e_4\}).
\]
The first branch is the path $o_1-x-e_1$ and the others are edges.
The crossings are
\[
\begin{split}
 W&=(e_1o_2,o_1o_3,o_1o_4,e_3e_2,o_2o_4,*),\\
 V&=(e_1r,e_3e_4,o_3e_4,ye_4).
\end{split}
\]
Since even $y$ lies in $S_4$ but not $S_1$, normalization forces
$y\in S_2\cup S_3$. Accordingly $*$ is $ye_2$ or $yo_3$.

\paragraph{The cycle $1-2-3-4-1$.}
For an odd $t$ write $s_i=s_i(t)$. Two templates, both with
$R=\{t\}$, are
\[
\begin{array}{c|l|l}
 &\text{condition}&(A_1,A_2,A_3,A_4)\\\hline
L&s_2=s_4=0&
(\{o_1,e_1,e_2\},\{o_2\},\{o_3,e_3,e_4\},\{o_4\})\\
M&s_1=s_2, s_3=s_4&
(\{o_1,e_2\},\{o_2,e_1\},\{o_3,e_4\},\{o_4,e_3\}).
\end{array}
\]
The nonsingletons in L are whole-double paths and those in M are
active edges. Their conditions give all apex crossings. The
remaining crossings are respectively
\[
\begin{split}
 W_L&=(o_2e_1,e_2e_3,o_4e_1,o_2e_3,o_2o_4,o_4e_3),\\
 W_M&=(e_2e_1,o_1o_3,e_2e_3,e_1e_4,o_2o_4,e_4e_3).
\end{split}
\]
Up to cycle symmetry these handle zero, one, two opposite, two
adjacent, and four occurrences of $t$, using L,L,L,M,M respectively.
If $T$ is all odd, at least one of its vertices has a different
occurrence count from three, because the total incidence count is
at most eight. Thus that charge is covered.

If the unique odd vertex $r$ occurs three times, rotate so its
occurrences are $1,2,3$. Write $T=\{r,x,y\}$. The support $S_4$
is empty or $\{x,y\}$; it cannot be empty, since then the even
vertices have only three incidences in total, whereas normalization
requires at least four. Relabel $x,y$ so $x\in S_2$. Use
\[
 (\{o_1,e_1,e_4\},\{o_2,e_3\},\{o_3,e_2,x\},\{o_4\},\{r\}).
\]
The nonsingletons contain the paths $o_1-e_4-e_1$,
$o_2-e_3$, $o_3-e_2-x$. Their crossings are
\[
\begin{split}
 W&=(e_1o_2,e_1e_2,e_1o_4,e_3e_2,e_3o_4,xo_4),\\
 V&=(re_1,re_3,re_2,ro_4).
\end{split}
\]

\paragraph{The star with centre $1$.}
Every leaf support $S_2,S_3,S_4$ has size two by normalization.
The leaf odd vertices form a triangle. If an odd $r\in T$ is
omitted by all three supports, take
\[
 \{o_1,e_1,e_2\},\quad\{o_2\},\quad\{o_3\},\quad\{o_4\},\quad\{r\}.
\]
The first branch absorbs $D_1$ and sees every leaf root via $e_1$;
the last sees every leaf root by the omitted-support condition.

For charge one, if this does not apply, write $T=\{r,x,y\}$
and arrange $r\in S_2$. Use
\[
 A=\{o_1,e_3,e_4,x,y\},\quad\{o_2\},\quad\{o_3\},\quad
 \{o_4\},\quad R=\{r,e_1,e_2\}.
\]
The latter is the path $r-e_2-e_1$. The first contains the star
with centre $o_1$ and leaves $e_3,e_4$; each of $x,y$ attaches
to it because its at least two supporting indices cannot be confined
to $2$. Every leaf support contains $x$ or $y$, so $A$ meets
every leaf singleton. The edges $e_1o_i$ give the corresponding
$R$ crossings, and $o_1e_2$ joins $A,R$.

For charge three, if all leaf supports agree their common omitted
odd vertex gives the first construction. Otherwise arrange
$S_3\ne S_4$ and take
\[
 \{o_1,e_1,e_2\},\quad\{o_2\},\quad\{o_3\},\quad\{o_4\},\quad
 R=T\cup\{e_3,e_4\}.
\]
Distinct two-subsets of $T$ intersect and cover it, so $R$ is
connected. Each leaf root sees its omitted triple vertex in $R$;
$e_1e_3$ joins $R$ to the first branch, which sees all leaves.

\paragraph{The paw with edges $12,13,23,14$.}
For odd $t$ write $s_i=s_i(t)$. If $s_4=0$, take
$A_1=\{o_1,e_4,e_1\}$, $A_4=\{o_4\}$, $R=\{t\}$ and
split $Q_{23}$ into $A_2,A_3$. The six nonapex crossings are
\[
 (e_1o_2,e_1o_3,e_1o_4,\text{cut edge},o_2o_4,o_3o_4).
\]
If $s_2=s_3$, take $A_2=\{o_2,e_3\}$,
$A_3=\{o_3,e_2\}$, split $Q_{14}$ into $A_1,A_4$ and again
use $R=\{t\}$. The equal bits give the two apex crossings at
$A_2,A_3$; the four crossings between paths are
$o_1e_3,o_1e_2,o_4o_2,o_4o_3$.

An all-odd triple always has a vertex outside $S_4$, so the first
case applies. For charge one, failure of both conditions permits
labelling $T=\{r,x,y\}$ with $S_4=\{r,x\}$,
$r\in S_2$, $r\notin S_3$. Use
\[
 (\{o_1,e_2,e_3\},\{o_2,e_1\},\{o_3\},\{o_4,x\},\{r,e_4\}).
\]
The first is an $o_1$-centred star, and the other nonsingletons
are edges. Their crossings are
\[
\begin{split}
 W&=(e_2e_1,e_2o_3,*,e_1o_3,o_2o_4,o_3o_4),\\
 V&=(o_1e_4,e_1e_4,o_3r,xe_4).
\end{split}
\]
Since even $x\in S_4$, normalization puts it in some other
$S_i$. These alternatives give $*=xo_1,xe_2,xe_3$ for
$i=1,2,3$ respectively.

\paragraph{The diamond $K_4-34$.}
For odd $t$ let $q_i=s_i(t)$. The following templates all have
$R=\{t\}$:
\[
\begin{array}{c|l|l}
 &\text{condition}&(A_1,A_2,A_3,A_4)\\\hline
 A&q_3=q_4=0&
 (\{o_1,e_1,e_3\},\{o_2,e_2,e_4\},\{o_3\},\{o_4\})\\
 B&q_1=q_3, q_2=q_4&
 (\{o_1,e_3\},\{o_2,e_4\},\{o_3,e_1\},\{o_4,e_2\})\\
 C&q_1=0, q_2=q_3&
 (\{o_1\},\{o_2,e_3\},\{o_3,e_2\},\{o_4,e_1,e_4\})\\
 D&q_2=0, q_1=q_4&
 (\{o_1,e_4\},\{o_2\},\{o_3,e_2,e_3\},\{o_4,e_1\}).
\end{array}
\]
Every nonsingleton is an active edge or a whole-double path.
Whole-double branches see $t$ by absorption; zero-bit singletons
see $t$ by \eqref{eq:six-triple-interface}; swapped pairs with
equal bits both see $t$. The other six crossings are
\[
\begin{array}{c|llllll}
 A&e_1e_2&e_1o_3&e_1o_4&e_2o_3&e_2o_4&o_3o_4\\
 B&e_3o_2&e_3e_1&o_1e_2&o_2e_1&e_4e_2&o_3o_4\\
 C&o_1e_3&o_1e_2&o_1e_4&e_3e_2&o_2e_1&o_3o_4\\
 D&e_4o_2&o_1e_2&e_4e_1&o_2e_3&o_2e_1&o_3o_4.
\end{array}
\]
Their conditions cover all bit vectors except
\[
 u=0011,\qquad v=1101,\qquad w=1110.
\]
For completeness, the following table exhibits the coverage:
\[
\begin{array}{c|cccc}
q_1q_2\backslash q_3q_4&00&01&10&11\\\hline
00&A&C&D&u\\
01&A&B&C&C\\
10&A&D&B&D\\
11&A&v&w&B.
\end{array}
\]

A further construction handles either charge whenever $S_3,S_4$
are distinct and nonempty:
\[
 \{o_1,e_2\},\quad\{o_2,e_1\},\quad\{o_3\},\quad\{o_4\},\quad
 R=T\cup\{e_3,e_4\}.
\]
The supports are distinct two-subsets of $T$, so $R$ is connected.
It contains an odd vertex. Every $o_i$ has an odd, hence positive,
number of neighbours in $T$, giving all four $R$ crossings.
The other crossings are
$(e_2e_1,e_2o_3,e_2o_4,e_1o_3,e_1o_4,o_3o_4)$.

For charge one, write $T=\{r,x,y\}$. Swapping $3,4$ exchanges
$v,w$. For $w=1110$, $S_1,S_2,S_3$ each contain $r$ and
one even vertex. Normalization forces $S_4=\{x,y\}$,
since otherwise the two even vertices have only three incidences.
Thus $S_3,S_4$ are distinct and nonempty. For $u=0011$, the
same construction applies unless $S_3=S_4$, in which case
relabel so both equal $\{r,x\}$. Normalization of $y$ forces
$S_1=S_2=\{x,y\}$. Use
\[
 (A_1,A_2,A_3,A_4,R)=
 (\{y,o_1,e_1\},\{o_2,e_2,e_4\},\{o_3\},\{o_4\},\{r,x,e_3\}).
\]
The three nonsingletons contain paths $o_1-y-e_1$,
$o_2-e_4-e_2$, $r-e_3-x$. Their crossings are
\[
\begin{split}
 W&=(e_1e_2,e_1o_3,e_1o_4,e_2o_3,e_2o_4,o_3o_4),\\
 V&=(xe_1,e_3e_2,xo_3,xo_4).
\end{split}
\]

For charge three, suppose none of the three odd vertices has a
covered vector. Their three vectors lie in $\{u,v,w\}$ and have
coordinatewise XOR zero, since every $|S_i|$ is even. Repetition
would leave a nonzero vector, so they must be $u,v,w$ once each.
Consequently $S_3=\{t_u,t_w\}$ and $S_4=\{t_u,t_v\}$,
which are distinct and nonempty. The connected-triple construction
applies. This exhausts both charges in all seven active graphs.

In every displayed construction the branches are disjoint, and
$A_i$ contains $o_i$ while the fifth branch contains an odd vertex
of $T$ (or, in the star notation, these same five roots are distributed
between $A,R$ and the leaf singletons). Connectivity was established
by the displayed paths, stars or overlapping supports, and the
crossing lists verify every remaining adjacency.
\end{proof}

\subsection{Assembly and preservation of the original roots}
\begin{proof}[Proof of Theorem~\ref{thm:excess-six}]
Apply support-tracked normalization to the specified witness $W$.
The resulting simple minor $H$ has the same $n$ nonempty odd-charge
fibres and order at most $n+6$. Every current vertex $x$ has a
connected original support $C_x\subseteq V(W)$, these supports
are pairwise disjoint, and
\[
 p_H(x)\equiv|C_x\cap P|\pmod2.
\]
Every current edge is witnessed by an original edge between supports.
If the excess is at most five, apply the original-rooted excess-five
theorem to the current witness. Otherwise remove singleton fibres
conceptually. An all-odd core has one of the three elementary models
already described. A core containing an even vertex has one of
\eqref{eq:six-patterns}, and Lemmas~\ref{lem:six-doubles},
\ref{lem:six-three-fibres}, \ref{lem:six-four-triple-double},
\ref{lem:six-two-triples} and~\ref{lem:six-one-triple} supply an
odd-rooted core clique model. The only reduction within these lemmas
uses twin deletion, lowers the excess to four and invokes the
previous excess-five theorem; thus there is no circular dependency.

Add each singleton fibre as its own branch. These vertices are pairwise
adjacent and meet every core branch at its odd vertex. Finally replace
each branch $Q$ by $\bigcup_{x\in Q}C_x$. These unions are disjoint
and connected, all crossing edges lift, and a current odd vertex
in each branch has an original support meeting $P$. The lifted model
lies in $W\subseteq F$ and has the asserted original-root property.

An empty core is already handled by the singleton clique. One- and
two-fibre cores are covered by the all-odd cases, since a normalized
even vertex would require at least three nonsingleton fibres. Thus
no lower bound on $n$ beyond $n\geq2$ is hidden in the case analysis.
\end{proof}

\begin{remark}
The proof concerns the specified witness and its specified odd set;
extra host edges or vertices are irrelevant. Its rooted conclusion
means that every branch meets that set, not that a prescribed vertex
from each fibre occupies a prescribed branch. The positive theorem
is entirely structural. Its sharpness as a uniform excess bound is
established separately by the order-fifteen, $K_8$-minor-free example;
that negative example is not a premise of the proof above.
\end{remark}

\section{The fifteen vertex obstruction}\label{sec:sharpness}
Let $\Theta$ be the theta graph on $\{0,\ldots,6\}$ with paths
\[
 5,2,6;\qquad 5,0,3,6;\qquad 5,1,4,6.
\]
Thus $E(\Theta)=\{03,05,14,15,25,26,36,46\}$. Replace each vertex $i$ by an independent double $\{o_i,e_i\}$ and use the active interface of \cref{lem:double} precisely for $ij\in E(\Theta)$. Adjoin a vertex $s$ adjacent to all $o_i$ and to no $e_i$. Call the resulting graph $H$.

\begin{figure}[ht]
\centering
\begin{tikzpicture}[scale=.9,every node/.style={circle,draw,fill=white,inner sep=2pt}]
\node (5) at (0,0) {$5$}; \node (6) at (6,0) {$6$};
\node (2) at (3,0) {$2$};
\node (0) at (2,1.6) {$0$}; \node (3) at (4,1.6) {$3$};
\node (1) at (2,-1.6) {$1$}; \node (4) at (4,-1.6) {$4$};
\draw (5)--(2)--(6) (5)--(0)--(3)--(6) (5)--(1)--(4)--(6);
\end{tikzpicture}
\caption{The active graph $\Theta(2,3,3)$. Its doubled rooted construction is due to Kriesell and Mohr. The graph $H$ adds a singleton adjacent exactly to the seven odd double vertices.}\label{fig:theta}
\end{figure}
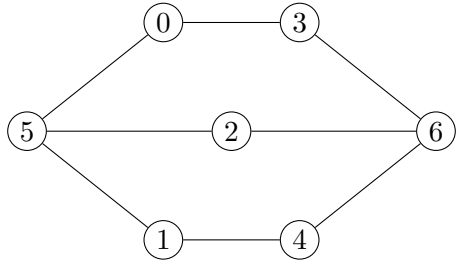

\begin{lemma}\label[lemma]{lem:h-witness}
The graph $H$ has fifteen vertices and forty-four edges and admits an oddomorphism to $K_8$. The odd vertices are $s,o_0,\ldots,o_6$ and have degree seven. The other vertices have degree four, except $e_5,e_6$, which have degree six.
\end{lemma}
\begin{proof}
Use the seven double fibres and $\{s\}$. Each double interface satisfies the required parity equations by \cref{lem:double}. Every $o_i$ has one neighbour in the singleton, every $e_i$ has none, and $s$ has one neighbour in each double. Each fibre has exactly one odd vertex. There are $3\cdot8+(21-8)+7=44$ edges. Every odd double vertex has one neighbour in each of the six other doubles and the neighbour $s$, so has degree seven. The even degree is twice its degree in $\Theta$, and $s$ has degree seven.
\end{proof}

In the technical companion, $H$ is the root-apex graph $H(42232)$ of a seven-index active graph encoded by the integer $42232$. The rooted obstruction of \cref{lem:known-theta} handles a model with singleton branch $s$. An ordinary $K_8$ model could have a different singleton. We first record the counting observation that enables a complete structural exclusion of all possibilities.

\begin{lemma}[Singleton pivot]\label[lemma]{lem:pivot}
Let $G$ have at most $2t-1$ vertices. Every $K_t$ model contains a singleton branch $\{x\}$, where $\deg_G(x)\ge t-1$. If $\deg_G(x)=t-1$, then the other $t-1$ branches contain exactly one vertex of $N_G(x)$ each.
\end{lemma}
\begin{proof}
Otherwise the $t$ disjoint branches use at least $2t$ vertices. Each of the other $t-1$ branches must contain a neighbour of $x$, and their disjointness forces these neighbours to be distinct. When $x$ has exactly $t-1$ neighbours, each is used once, with exactly one in each branch.
\end{proof}

\begin{lemma}[Kriesell--Mohr rooted theta obstruction]\label[lemma]{lem:known-theta}
The graph $H-s$ has no $K_7$ model whose seven branches each meet $\{o_0,\ldots,o_6\}$.
\end{lemma}
\begin{proof}
The graph is the root completion of the construction of Kriesell and Mohr \cite[Section 3, pp.~8--9, immediately before Theorem 2]{KriesellMohr2022}: their $Z(\Theta)$ has vertices $(i,1)=o_i$ and $(i,2)=e_i$ and the three active edges for each $ij\in E(\Theta)$, and $H-s$ adds $o_io_j$ for each nonedge $ij$ of $\Theta$. We give their counting argument, written for the root-completed graph, to keep the proof self-contained.

Suppose such a model exists. Each branch contains exactly one odd vertex; let $C_i$ be the branch containing $o_i$. The graph $H-s$ is connected: its odd vertices induce the complement of $\Theta$, which has minimum degree three on seven vertices and is therefore connected, and every even vertex has an odd neighbour. All odd vertices are already used, so adding unused even vertices one at a time to an adjacent branch gives such a model whose branches partition $V(H-s)$. The three cycles of $\Theta$ have lengths $5$, $5$ and $6$; thus $\Theta$ is triangle-free and has no $4$-cycle.

(a) Let $A=\{i:C_i=\{o_i\}\}$. Singleton branches are pairwise adjacent, so $A$ is independent in $\Theta$. If $i\in A$ and $j\in N_\Theta(i)$, then $|C_j|\ge3$. Indeed $o_j$ is the only odd vertex of $C_j$ and is not adjacent to $o_i$, so $C_j$ contains a neighbour $e_k$ of $o_i$, where $k\in N_\Theta(i)$. If $C_j$ were $\{o_j,e_k\}$, then $o_je_k$ would be an edge, so $jk\in E(\Theta)$ with $k\ne j$, and $i,j,k$ would form a triangle.

(b) Counting the fourteen vertices,
\[
 14=\sum_j|C_j|\ \ge\ |A|+3|N_\Theta(A)|+2\bigl(7-|A|-|N_\Theta(A)|\bigr)=14+|N_\Theta(A)|-|A|,
\]
so $|N_\Theta(A)|\le|A|$.

(c) Every nonempty independent set $A$ of $\Theta$ satisfies $|N_\Theta(A)|\ge|A|+1$. The vertices $5,0,3,6,2$ form a $5$-cycle $Q$, and the remaining vertices $1,4$ are adjacent, with $N_\Theta(1)=\{4,5\}$ and $N_\Theta(4)=\{1,6\}$. Let $J\subseteq V(Q)$ be nonempty and independent. For a cyclic orientation of $Q$, the successor map and the predecessor map both inject $J$ into $N_Q(J)$. If $|N_Q(J)|=|J|$, both are bijections onto $N_Q(J)$, so $J$ is invariant under the rotation by two steps. That rotation generates all rotations of a $5$-cycle, so $J=V(Q)$, which is not independent. Hence $|N_Q(J)|\ge|J|+1$. An independent set $A$ contains at most one of $1$ and $4$. If it contains neither, $|N_\Theta(A)|\ge|N_Q(A)|\ge|A|+1$. If $1\in A$, then $4\in N_\Theta(A)\setminus V(Q)$; so $|N_\Theta(A)|\ge|N_Q(A\setminus\{1\})|+1\ge|A|+1$ when $A\ne\{1\}$, and $|N_\Theta(\{1\})|=2$. The case $4\in A$ is the same with $1$ and $6$ in place of $4$ and $5$.

By (b) and (c), $A=\varnothing$. The count in (b) then gives $|C_j|=2$ for every $j$, so $C_j=\{o_j,e_{f(j)}\}$ with $f(j)\in N_\Theta(j)$, and $f$ is a permutation of $\{0,\ldots,6\}$ without fixed points.

(d) If $ij\in E(\Theta)$, then $C_i$ and $C_j$ are adjacent only if $f(i)=j$ or $f(j)=i$. Indeed $o_io_j$ is not an edge. An edge $o_ie_{f(j)}$ would require $if(j)\in E(\Theta)$; since $f(j)\in N_\Theta(j)$, this gives a triangle unless $f(j)=i$, and $o_ie_i$ is not an edge. The same holds for $e_{f(i)}o_j$. Finally, if $e_{f(i)}e_{f(j)}$ is an edge while $f(i)\ne j$ and $f(j)\ne i$, then $i,j,f(j),f(i)$ are four distinct vertices forming a $4$-cycle.

(e) Vertex $5$ has the three neighbours $0,1,2$. At most one of them equals $f(5)$, and since $f$ is injective at most one satisfies $f(j)=5$. By (d), some neighbour $j$ of $5$ has $C_j$ not adjacent to $C_5$, a contradiction.
\end{proof}
This lemma and its proof are due to Kriesell and Mohr; no priority is claimed.

\begin{proposition}[Structural exclusion]\label[proposition]{prop:noK8}
The graph $H$ has no $K_8$ minor.
\end{proposition}
\begin{proof}
Suppose such a model exists. The graph $H$ is connected: $s$ sees every $o_i$, and every even vertex has an odd neighbour. Extend the model to a partition of $V(H)$ into eight branches by repeatedly adding an unused vertex adjacent to a used vertex to that vertex's branch. Connectedness of $H$ guarantees such an edge until every vertex is used; connectedness, disjointness and pairwise adjacency of the branches are preserved.

Let $S$ be the \emph{exact} set of singleton vertices of this final spanning model, and put $k=|S|$. We do not require an original singleton to survive the extension. By \cref{lem:h-witness,lem:pivot}, $S$ is a nonempty subset of $\{s,o_0,\ldots,o_6\}$. Each $x\in S$ has degree seven, so every other branch contains exactly one neighbour of $x$. Call this the \emph{singleton saturation condition}. The nonsingleton branches satisfy
\begin{equation}\label{eq:singleton-budget}
 \sum_{B\text{ nonsingleton}}(|B|-2)=(15-k)-2(8-k)=k-1.
\end{equation}
If $s\in S$, the remaining seven branches each contain one odd vertex and contradict \cref{lem:known-theta}. Henceforth $s\notin S$.

\paragraph{The eight possible singleton types.}
Write $I=\{i:o_i\in S\}$, and let $B_s$ be the branch containing $s$. This branch is nontrivial and connected, so it contains some $o_j$ adjacent to $s$. For every $i\in I$, the vertex $s$ already supplies the unique neighbour of $o_i$ in $B_s$. Saturation therefore forbids $o_io_j$, giving
\begin{equation}\label{eq:singleton-neighbourhood}
 I\subseteq N_\Theta(j).
\end{equation}
This reasoning allows several odd vertices in $B_s$; each obeys the same restriction. Since $\Delta(\Theta)=3$, we have $1\le k\le3$.

The automorphisms
\[
 a=(0\ 1)(3\ 4),\qquad b=(0\ 3)(1\ 4)(5\ 6)
\]
of $\Theta$ generate a group of order four. They act on both odd and even indices and fix $s$, so preserve $H$. Up to this group, the possibilities in \eqref{eq:singleton-neighbourhood} are
\[
\begin{array}{c|c|c}
k&\text{representatives }I&\text{respective orbit sizes}\\\hline
1&\{0\},\{2\},\{5\}&4,1,2\\
2&\{0,1\},\{0,2\},\{0,6\},\{5,6\}&2,4,4,1\\
3&\{0,1,2\}&2.
\end{array}
\]
To check completeness directly, the pairs contained in a $\Theta$-neighbourhood are precisely
\[
01,02,12,23,24,34,35,45,56,06,16;
\]
the only triples are $012$ and $234$. The displayed permutations give the stated orbits. Thus the following eight cases exhaust the possibilities without a model search.

\paragraph{Necessary pair restrictions.}
Every pair branch is an edge and has at least seven distinct external neighbours, one in each other branch. Neither $\{e_0,e_3\}$ nor $\{e_1,e_4\}$ can be a pair branch: each endpoint has degree four, so the pair has at most six external neighbours. For this proof a pair is \emph{eligible for $S$} if it is a two-element subset of $V(H)\setminus S$, is an edge other than these two, and satisfies singleton saturation for every $x\in S$. This is only a necessary condition; eligible pairs need not form a model.

All partner assertions below follow by checking this rule against
\[
\begin{array}{c|ccccccc}
i&0&1&2&3&4&5&6\\\hline
N_\Theta(i)&\{3,5\}&\{4,5\}&\{5,6\}&\{0,6\}&\{1,6\}&\{0,1,2\}&\{2,3,4\}.
\end{array}
\]
We give every partner list used in a forcing step.

\paragraph{One singleton.}
By \eqref{eq:singleton-budget}, all seven other branches are pairs.

\emph{$I=\{0\}$.} The only eligible partner of $e_0$ is $e_5$, and that of $e_3$ is $e_6$, forcing these pairs. The partners of $e_1$ are $o_4,e_5$, so $e_1$ pairs with $o_4$. Those of $e_2$ are $e_5,o_6$, forcing $e_2o_6$. Those of $e_4$ are $o_1,o_6$, forcing $e_4o_1$. Those of $o_2$ are $o_3,e_6$, forcing $o_2o_3$. The final pair is $so_5$. But $\{e_3,e_6\}$ and $\{s,o_5\}$ are anticomplete: $s$ has no even neighbours and $35,65\notin E(\Theta)$.

\emph{$I=\{2\}$.} The eligible partners of $s$ are $o_5,o_6$. The automorphism $b$ fixes $2$ and swaps $5,6$, so assume the pair is $so_6$. The partners of $o_0$ are $e_3,o_6$, and those of $o_1$ are $e_4,o_6$, forcing $o_0e_3$ and $o_1e_4$. The partners of $e_6$ are $e_2,e_3,e_4$, forcing $e_2e_6$. Those of $e_5$ are $e_0,e_1,e_2$, so it pairs with $e_0$ or $e_1$. Either even pair is anticomplete to $\{s,o_6\}$, since $06,16,56\notin E(\Theta)$.

\emph{$I=\{5\}$.} Both $e_0$ and $e_1$ have $e_5$ as their only eligible partner, contradicting disjointness.

\paragraph{Two singletons.}
Now there is exactly one triple branch and five pair branches. A vertex with no eligible partner must belong to the unique triple.

\emph{$I=\{0,1\}$.} Neither $e_3$ nor $e_4$ has an eligible partner, so both lie in the triple. They are nonadjacent, and their common neighbours outside $S$ are $o_6,e_6$. The vertex $o_6$ violates saturation, because $e_3$ already sees $o_0$ and $e_4$ already sees $o_1$, while $o_6$ sees both. Thus the triple is $\{e_3,e_4,e_6\}$. Both $e_0$ and $e_1$ must now lie in pairs, and each has only the eligible partner $e_5$, a contradiction.

\emph{$I=\{0,2\}$.} Neither $o_3$ nor $o_6$ has an eligible partner, so both lie in the triple. They are nonadjacent; their common neighbours outside $S$ are $o_1,o_5,s$. Saturation excludes $o_1,s$, so the triple is $\{o_3,o_5,o_6\}$. But $N_\Theta(0)\cap N_\Theta(2)=\{5\}$, and the argument for \eqref{eq:singleton-neighbourhood} requires $B_s$ to contain $o_5$. That vertex is already in the triple not containing $s$, a contradiction.

\emph{$I=\{0,6\}$.} None of $e_0,e_1,o_2,e_4$ has an eligible partner. Four vertices cannot fit into the unique triple.

\emph{$I=\{5,6\}$.} None of $e_0,e_1,e_3,e_4$ has an eligible partner, giving the same contradiction.

For a direct check of the preceding pair exclusions, the complete eligible-pair lists for these four cases are
\begin{align*}
01:&\quad e_0e_5,e_1e_5,o_2e_6,e_2e_5,e_2o_6,o_3o_4,o_5o_6,o_5s;\\
02:&\quad e_0e_5,o_1e_4,e_1o_4,e_1e_5,e_2e_5,e_3e_6,o_4o_5,o_5s;\\
06:&\quad o_1o_3,e_2e_5,o_3s,e_3e_6,o_4o_5;\\
56:&\quad o_0o_4,o_1o_3,o_2s,e_2e_5,e_2e_6.
\end{align*}

\paragraph{Three singletons.}
Take $I=\{0,1,2\}$. A vertex's \emph{signature} is its set of neighbours in $S$. The signatures are
\[
\begin{array}{c|c}
\text{vertices}&\text{signature}\\\hline
e_0,e_1,e_2,o_5&\varnothing\\
e_3;\ e_4;\ e_6&\{o_0\};\ \{o_1\};\ \{o_2\}\\
o_3;\ o_4;\ o_6&\{o_1,o_2\};\ \{o_0,o_2\};\ \{o_0,o_1\}\\
e_5,s&S.
\end{array}
\]
The branch containing $o_3$ must contain $e_3$, the only vertex whose nonempty signature supplies the missing $o_0$ without repeating a neighbour of another singleton. Every further vertex in that branch has empty signature. Within the permitted set $\{o_3,e_3,e_0,e_1,e_2,o_5\}$, the only neighbour of $e_3$ is $e_0$. Connectivity therefore forces the path $o_3e_0e_3$ into the branch. In the same way, a branch contains $o_4e_1e_4$ and another contains $o_6e_2e_6$.

These three branches are distinct: any two of $o_3,o_4,o_6$ have overlapping nonempty signatures and cannot lie together under saturation. Three branches of size at least three contribute at least three to the left side of \eqref{eq:singleton-budget}, whereas $k-1=2$. This is the final contradiction.
\end{proof}

\begin{proof}[Proof of \cref{thm:sharp}]
For $n=8$, \cref{lem:h-witness,prop:noK8} give the fifteen-vertex example. By \cref{thm:main}, any witness to $K_8$ of order at most fourteen forces a $K_8$ minor. In particular this holds for every host of that order with a standard weak oddomorphism, proving minimality.

For $n>8$, adjoin $n-8$ vertices to $H$, pairwise adjacent and adjacent exactly to $s,o_0,\ldots,o_6$ among the old vertices. Give each a new singleton colour and odd indicator. The old even vertices see none of the new singletons; every odd vertex sees one vertex in each new singleton fibre, and every new vertex sees an odd number in each old fibre. Thus the colouring is an oddomorphism to $K_n$ on $15+n-8=n+7$ vertices.

Deleting $k$ vertices from a $K_t$ model destroys at most $k$ branch sets, so $h(G)\le h(G-X)+|X|$, where $h$ is the Hadwiger number. Delete the $n-8$ new vertices. Since $h(H)\le7$, the extension has Hadwiger number at most $7+n-8=n-1$. It has no $K_n$ minor. For every $n\ge8$, \cref{thm:main} rules out smaller hosts of order at most $n+6$ with a weak oddomorphism, so $m(n)=n+7$.
\end{proof}

\paragraph{A check that the obstruction is tight in minor order.}
The fourteen-vertex core itself contains a $K_7$ model:
\[
\{o_3\},\quad\{e_0,e_3\},\quad\{o_0,o_4\},\quad
\{e_1,o_5\},\quad\{e_2,e_6\},\quad\{o_2,e_5\},\quad\{o_1,o_6\}.
\]
Each two-vertex branch is an edge. In the listed branch order, witnesses for the twenty-one pairwise adjacencies are
\begin{align*}
12&:o_3e_0,&13&:o_3o_4,&14&:o_3o_5,&15&:o_3e_6,&16&:o_3o_2,&17&:o_3o_1,\\
23&:e_3o_0,&24&:e_0o_5,&25&:e_3e_6,&26&:e_0e_5,&27&:e_3o_6,\\
34&:o_4e_1,&35&:o_4e_6,&36&:o_0o_2,&37&:o_0o_1,\\
45&:o_5e_2,&46&:e_1e_5,&47&:o_5o_6,&56&:e_2e_5,&57&:e_2o_6,&67&:o_2o_1.
\end{align*}
Thus $h(H)=7$. Two of these branches contain only even vertices; this model cannot be used as if it were rooted in the specified odd set.

\begingroup
\renewcommand{\theequation}{6.\arabic{equation}}
\renewcommand{\theHequation}{rootedfour.\arabic{equation}}
\setcounter{equation}{0}
\section{Rooted four-colour witnesses}\label{sec:rooted-four}
We prove \cref{thm:rooted-four} without an order bound. The odd set $P$ is the one specified by the original witness, and $P_i=P\cap V_i$. A branch need only meet $P$; the four roots are not prescribed and need not have different colours. We use the witness notation of \cref{sec:normalization}, writing $q$ for the number of fibres when a reduction applies to several targets.

\subsection{Root preservation at separators of order at most two}\label{sec:kfour-separator}
The following reduction keeps the original roots when a small separator is replaced. It does not use the excess-six theorem.

For vertex sets $U,Z$ that are disjoint, let $e_W(U,Z)$ count edges between them. All equalities in the parity calculations are in $\Ftwo$. For a colour $i$, write $U_i=U\cap V_i$.

\begin{lemma}[Charges outside a separator]\label{kfour:lem:charges}
Let $S\subseteq V(W)$ miss at least one colour, and put $Y=\{i:S_i\ne\varnothing\}$. Every component $C$ of $W-S$ has a common charge
\[
 \pi_C=|P\cap C_i|\pmod2 \qquad(i\notin Y).
\]
Moreover $\sum_C\pi_C=1$.
\end{lemma}
\begin{proof}
For different $i,j\notin Y$, summing \eqref{eq:parity} in $C_i$ gives $|P\cap C_i|=e_W(C_i,C_j)$, since there are no neighbours of colour $j$ in $S$ or another component. Interchanging $i,j$ gives the common charge. When only one colour is missed there is nothing to compare. For a fixed missed colour, the components partition its entire fibre, so their charges sum to one by \eqref{eq:charge}.
\end{proof}

Suppose $W-S$ has at least two components. Choose a component $C_0$ with $\pi_{C_0}=1$, set $K=V(C_0)$, and let $D=V(W)\setminus(K\cup S)$; thus $D\ne\varnothing$. For $y\in Y$ set $T_y=S_y$ and $\alpha_y=|P\cap K_y|\pmod2$.

\begin{lemma}[Separator quotient]\label{kfour:lem:quotient}
Create a vertex $t_y$ of colour $y$ for each $y\in Y$, and retain the vertices of $K$. Retain all edges of $W[K]$. Join $v\in K$ to $t_y$ if and only if $v$ has an odd number of neighbours in $T_y$. For distinct $y,z\in Y$, use the edge bit
\begin{equation}\label{kfour:eq:epsilon}
 \varepsilon(y,z)=1+\alpha_y+e_W(K_z,T_y).
\end{equation}
This bit is symmetric. Give $v\in K$ its old indicator and set $p_F(t_y)=1+\alpha_y$. The resulting graph $F$ is a specified witness to $K_q$ of smaller order. It is a minor of $W[K\cup S]$ after completing $S$ to a clique.
\end{lemma}
\begin{proof}
For a missed colour $x$ and a separator colour $y$, the witness equations yield
\[
 e_W(K_x,K_y)=\alpha_y,\qquad
 e_W(K_x,K_y)+e_W(K_x,T_y)=1.
\]
Thus $e_W(K_x,T_y)=1+\alpha_y$. For two separator colours,
\[
 \alpha_y=e_W(K_y,K_z)+e_W(K_y,T_z),\quad
 \alpha_z=e_W(K_y,K_z)+e_W(K_z,T_y).
\]
Their sum proves symmetry of \eqref{kfour:eq:epsilon}. Every retained vertex keeps its degree parity into each other colour: it has no neighbour in $D$, and incidences to $T_y$ are replaced by their parity. The degree of $t_y$ into a missed colour is $1+\alpha_y$ by the first identity. Into colour $z\in Y\setminus\{y\}$ it is
$e_W(K_z,T_y)+\varepsilon(y,z)=1+\alpha_y$. The colour charges are one: a missed colour has charge one in $K$; colour $y\in Y$ has charge $\alpha_y+(1+\alpha_y)=1$.

Contract each $T_y$ inside the completed separator clique. Every edge selected to $K$ has an original representative, and all selected separator edges are available in that clique. Delete surplus edges. The order is $|K|+|Y|<|V(W)|$.
\end{proof}

Call $t_y$ \emph{artificially odd} if $p_F(t_y)=1$ but $P\cap T_y=\varnothing$. In that case
\begin{equation}\label{kfour:eq:root-supply}
 |P\cap D_y|=1 \pmod2,
\end{equation}
since the original fibre has charge one and $\alpha_y=0$. To lift rooted models it suffices to realise each $t_y$ by a connected support $X_y\subseteq T_y\cup D$ containing $T_y$, with the supports pairwise disjoint and adjacent wherever \eqref{kfour:eq:epsilon} requires, and with $X_y\cap P\ne\varnothing$ for every artificially odd $t_y$. Original roots in $T_y$ already root its support. Expanding these supports in a model of $F$ then preserves connectivity, disjointness, all interbranch edges, and roots.

\begin{lemma}[Minimal rooted counterexamples have no small separator]\label{kfour:lem:minimal}
If a specified witness to $K_q$, $q\ge3$, without a model rooted in its own odd set of order $t$ exists, choose one of minimum order. It has no separation of order at most two. The same statement holds with witnesses restricted to any minor-closed class.
\end{lemma}
\begin{proof}
Suppose otherwise, and choose a separator $S$ of minimum size $k\le2$. Every component of $W-S$ is adjacent to every member of $S$: otherwise its neighbourhood is a separator of size less than $k$. Since $q>k$, Lemmas~\ref{kfour:lem:charges} and~\ref{kfour:lem:quotient} apply. We construct the supports just described. The resulting $F$ is an actual minor of $W$, and is smaller, so minimality gives a rooted model of order $t$ in $F$. Its lifted model contradicts our choice of $W$.

If $k=0$, use the charge-one component alone. If $k=1$, the sole support is the separator vertex, enlarged by a discarded component containing an original root when it is artificially odd; \eqref{kfour:eq:root-supply} supplies that component. If $k=2$ and $S$ is monochromatic, take $X_y=S\cup D_j$ for a discarded component $D_j$, choosing one containing a root if needed. It is connected because $D_j$ attaches to both vertices.

Now suppose $S=\{a,b\}$ has different colours. If neither quotient vertex is artificially odd, take $\{a\}\cup D_j$ and $\{b\}$. If precisely one is artificially odd, put a discarded component containing a root into that vertex's support. These two supports are connected and adjacent. Original odd separator vertices remain in their own support.

If both quotient vertices are artificially odd, then $a,b\notin P$, and \eqref{kfour:eq:root-supply} supplies at least two distinct roots in $D$ (in the two different separator colours). In $G_D=W[D\cup S]$ there are two vertex-disjoint paths from $\{a,b\}$ to distinct vertices of $P\cap D$. To see this, if not, the set-to-set form of Menger's theorem gives a vertex $z$ meeting every such path. It is not $a$ or $b$, since a component containing a discarded root is adjacent to both separator vertices. If $z\in D$, choose a root in $(P\cap D)\setminus\{z\}$. Its component in $G_D-z$ avoids $a,b$. This component has no neighbours outside $G_D$, except through $a,b$ which it does not reach, so $z$ would be a cut vertex of $W$. That contradicts the minimality of $k=2$.

Let these paths be $Q_a,Q_b$. Their interiors lie in $D$, and they give two disjoint connected rooted supports containing respectively $a,b$. If the supports are not adjacent, extend one along the interior of a shortest connecting path in its component of $D$, or to the other separator vertex through that component. More explicitly, if both paths enter the same component of $D$, join their vertex sets inside it; otherwise join $Q_a-a$ to $b$ inside the component containing $Q_a-a$ together with $b$. Add only the interior of a shortest such path to the first support. The supports stay disjoint and become adjacent. This supplies any required quotient edge and finishes the lifting.
\end{proof}

\subsection{Odd endpoint sets in a cyclic order}\label{sec:kfour-parity}
A \emph{matching of a finite even set} here means a partition into unordered pairs. It need not consist of edges of the graph. For two disjoint pairs of points on a circle, say they \emph{alternate} when their four endpoints occur in the order $a,b,a',b'$ with one pair $\{a,a'\}$ and the other $\{b,b'\}$. Given matchings $M_A,M_B$ of disjoint even sets $A,B$ in a cyclic order, let $I(M_A,M_B)\in\Ftwo$ be the parity of the number of alternating pairs of matching edges.

\begin{lemma}[Endpoint formula]\label{kfour:lem:endpoint}
Cut the cyclic order at a point that is not an endpoint, thereby obtaining a linear order $<$. Then
\begin{equation}\label{kfour:eq:endpoint}
 I(M_A,M_B)=|\{(a,b)\in A\times B:a<b\}|\pmod2.
\end{equation}
In particular the parity is independent of both matchings.
\end{lemma}
\begin{proof}
Fix a pair $\{b,b'\}\in M_B$, with $b<b'$. A pair of $M_A$ alternates with it exactly when one of its endpoints lies in the open linear interval $(b,b')$. Consequently the number of alternating pairs of $M_A$, modulo two, is $|A\cap(b,b')|$. Put $f(b)=|\{a\in A:a<b\}|\pmod2$. As $A\cap B=\varnothing$, this number is $f(b)+f(b')$. Summing over $M_B$ counts $f(b)$ once for each $b\in B$, giving \eqref{kfour:eq:endpoint}. Changing the cut has no effect since $|A|$ and $|B|$ are even, or directly since the left side was defined cyclically.
\end{proof}

\begin{lemma}[Four odd classes force an alternating pair]\label{kfour:lem:fourclasses}
Let $P_1,P_2,P_3,P_4$ be disjoint odd-sized sets in a cyclic order. For each $i<j$, choose any matching $M_{ij}$ of $P_i\cup P_j$. Then
\begin{equation}\label{kfour:eq:mainidentity}
 I(M_{12},M_{34})+I(M_{13},M_{24})+I(M_{14},M_{23})=1.
\end{equation}
Thus at least one complementary pair of matchings contains an alternating pair of matching edges.
\end{lemma}
\begin{proof}
Use a linear cut as above and write $t_{ij}=|\{(u,v)\in P_i\times P_j:u<v\}|\pmod2$. Lemma~\ref{kfour:lem:endpoint} gives
\begin{align*}
 I(M_{12},M_{34})&=t_{13}+t_{14}+t_{23}+t_{24},\\
 I(M_{13},M_{24})&=t_{12}+t_{14}+t_{32}+t_{34},\\
 I(M_{14},M_{23})&=t_{12}+t_{13}+t_{42}+t_{43}.
\end{align*}
Adding cancels the repeated terms and leaves
\[
 (t_{23}+t_{32})+(t_{24}+t_{42})+(t_{34}+t_{43})
 =|P_2||P_3|+|P_2||P_4|+|P_3||P_4|=1.
\]
The last equality uses the odd cardinalities. No choice or enumeration of orders or matchings enters the argument.
\end{proof}

\begin{lemma}[Odd-degree endpoint pairing]\label{kfour:lem:pairing}
In any finite graph $H$, its set $T$ of odd-degree vertices can be paired so that each pair is joined by a path in $H$.
\end{lemma}
\begin{proof}
Each connected component has an even number of vertices in $T$ by the handshaking identity. Pair those vertices within their component and choose a path for each pair. The paths within one component need not be disjoint, and we impose no such requirement.
\end{proof}
Only paths belonging to \emph{complementary colour pairs} will need to be disjoint; their disjoint vertex sets ensure this. No simultaneous disjointness is required within one two-colour graph.

\begin{theorem}[Cyclic witness criterion]\label{kfour:thm:cyclic}
Let $W$ be a specified witness to $K_4$. If a cycle $C$ of $W$ contains $P$, then $W$ has a $P$-rooted $K_4$ model.
\end{theorem}
\begin{proof}
For each $i<j$, set $H_{ij}=W[V_i\cup V_j]$. Since the fibres are independent, \eqref{eq:parity} says that the odd-degree vertices of $H_{ij}$ are exactly $P_i\cup P_j$. Apply Lemma~\ref{kfour:lem:pairing} to obtain a matching $M_{ij}$ and a path in $H_{ij}$ joining each matched pair. Each $P_i$ is odd-sized by \eqref{eq:charge}. Read all endpoints in their order on $C$ and apply Lemma~\ref{kfour:lem:fourclasses}.

We obtain two matching pairs with alternating endpoints, one from $M_{ij}$ and one from $M_{k\ell}$, where $\{i,j\}\cap\{k,\ell\}=\varnothing$. Their chosen paths are vertex-disjoint: one lies in $V_i\cup V_j$, the other in $V_k\cup V_\ell$. Denote their endpoints, in cyclic order, by $a,b,c,d$, so that the paths join $a$ to $c$ and $b$ to $d$. All four are distinct members of $P$. By \cite[Lemma~7]{FabilaMonroyWood2013}, the cycle and those two paths give a $K_4$ model rooted at $a,b,c,d$. This is a $P$-rooted model.
\end{proof}
\begin{remark}[Essential distinctions]
The joining paths are allowed to meet $C$ at additional vertices. Fabila-Monroy and Wood's Lemma~7 permits this; it does not require two chords internally disjoint from the cycle. Paths chosen for the same colour pair may intersect, since only one of them is used. A matched pair may have both endpoints in the same fibre. None of these facts weakens the required conclusion, which is set-rooted, not an assigned transversal.
\end{remark}

\subsection{A cycle through the roots}\label{sec:kfour-mainproof}
\begin{lemma}[Maximal-root cycle principle]\label{kfour:lem:cycle}
Let $G$ be 3-connected and let $S\subseteq V(G)$. Either $G$ has an $S$-rooted $K_4$ model, or some cycle contains every vertex of $S$.
\end{lemma}
\begin{proof}
Choose a cycle $C$ containing as many members of $S$ as possible; 3-connectivity guarantees the existence of a cycle. If $r\in S\setminus V(C)$, the fan lemma supplies three paths from $r$ to three distinct vertices $a,b,c$ of $C$, pairwise disjoint except at $r$ and with interiors outside $C$.

Orient $C$ so that $a,b,c$ appear in this order. If any of the three open arcs $(a,b)$, $(b,c)$, $(c,a)$ has no member of $S$, replace that arc by the two fan paths through $r$ between its endpoints. The result is a cycle containing all roots previously on $C$ and also $r$, contradicting maximality.

Therefore all three open arcs contain roots. Partition $V(C)$ into the three half-open arcs $[a,b)$, $[b,c)$, $[c,a)$. Each is nonempty and connected, contains a root, and consecutive arcs are adjacent through an edge of $C$; there are three arcs, so they are pairwise adjacent. The union of the three fan paths with $a,b,c$ removed is a fourth connected set, contains $r$, is disjoint from the three arcs, and is adjacent to each arc at its corresponding fan endpoint. These four sets form the required rooted clique model.
\end{proof}
For $|S|\ge3$, this is the unsubdivided case of \cite[Lemma~4]{GroenlandNederlofKoana2024}. The proof above records the elementary fan argument in the form needed here.

\begin{proof}[Proof of Theorem~\ref{thm:rooted-four}]
If a counterexample exists, choose a specified $K_4$ witness without an original-rooted $K_4$ of minimum order among all specified $K_4$ witnesses. Lemma~\ref{kfour:lem:minimal}, with $q=t=4$, shows that it has no separation of order at most two. It has at least four vertices because every fibre is nonempty. Hence it is 3-connected.

It has no $P$-rooted $K_4$ by hypothesis, so Lemma~\ref{kfour:lem:cycle} gives a cycle containing all of $P$. Theorem~\ref{kfour:thm:cyclic} then gives exactly the excluded rooted model, a contradiction. Thus every specified $K_4$ witness has that model.

The same model lies in every graph containing the witness, so the assertion also applies to a specified witness of a weak oddomorphism.
\end{proof}

\endgroup
\begingroup
\renewcommand{\theequation}{7.\arabic{equation}}
\renewcommand{\theHequation}{smallfibre.\arabic{equation}}
\setcounter{equation}{0}
\section{Six-colour witnesses with a small fibre}\label{sec:small-fibre}
We prove \cref{thm:small-fibre}, extending the all-order singleton criterion to a two-vertex fibre. The proof reduces a hypothetical double-bearing counterexample to a four-connected root-preserving minor witness, then applies a clean-root consequence of the classical rural-society theorem. This section uses \cref{thm:rooted-four} and the separator machinery of \cref{sec:kfour-separator}, but not the excess-six theorem or any computation.

A \emph{double} is a fibre $\{o,e\}$ with $p(o)=1$ and $p(e)=0$; every two-vertex fibre has this form by \eqref{eq:charge}. A root $r$ is \emph{clean} if $N_W(r)\subseteq P$. A \emph{root-preserving minor witness} $F$ of $W$ to $K_q$ is a specified $K_q$ witness with pairwise disjoint connected supports $X_x\subseteq V(W)$, one for each $x\in V(F)$, such that every edge $xy$ of $F$ is represented by an edge between $X_x$ and $X_y$, and $X_x$ meets the original $P$ whenever $x$ is odd in $F$. The supports need not stay within a colour, cover $W$, or have an assigned number of roots. Even supports may contain roots. Expanding supports lifts any model rooted in the odd set of $F$ to a $P$-rooted model in $W$. These witnesses compose: replace every vertex in a support for a second reduction by its first support and take their union. Disjointness, connectivity, all required edges and the root condition are preserved.

\subsection{Two supports and colour deletion}
Use the notation $S,Y,K,D,T_y,\alpha_y,F$ of \cref{sec:kfour-separator}: $K$ is a charge-one component of $W-S$, $D$ is the union of the other components, and $t_y$ represents the separator class $T_y$. All identities involving marks, charges or edge bits in this section are in $\Ftwo$. Besides \eqref{kfour:eq:epsilon}, for any missed colour $m\notin Y$ we have
\begin{equation}\label{kfive:eq:discard}
\varepsilon(y,z)=e_W(T_y,T_z)+e_W(T_y,D_z)+e_W(T_y,D_m)
\qquad(y,z\in Y,\ y\ne z).
\end{equation}
Indeed, sum the degree-parity equations of $T_y$ toward colours $z$ and $m$ and cancel their common right side. Substitute
$e_W(K_m,T_y)=1+\alpha_y$ from the proof of \cref{kfour:lem:quotient} to obtain the identity. If $t_y$ is artificially odd, \eqref{kfour:eq:root-supply} supplies an original root in $D_y$.

The support rule following \cref{kfour:lem:quotient} also applies after deleting whole colours of $F$: supports are then required only for retained separator vertices. Retained vertices of $K$ have singleton supports; every selected edge to $t_y$ has a representative at $T_y$, and the prescribed supports realise all retained separator edges and required original roots. Vertices of deleted colours in the original graph may occur inside supports. The quotient is a witness to the retained complete target by \cref{lem:attachment}.

The proof of \cref{kfour:lem:minimal}, before using minimality, proves the following useful reduction: if a specified $K_q$ witness, $q\ge3$, has a separator of order at most two, the quotient at a separator of minimum size has root-preserving supports and smaller order. We use that explicit construction, rather than assuming that witnesses having a double form a minor-closed class.

\begin{lemma}[Two rooted supports across a three-separator]
\label{kfive:lem:two}
Let $G$ be 3-connected, let $S=\{a,b,c\}$ separate $G$, and let
$D_0$ be a component of $G-S$. For every $R\subseteq D_0$ with
$|R|\ge2$, there are disjoint connected adjacent sets
$X_a\subseteq D_0\cup\{a\}$ and
$X_b\subseteq D_0\cup\{b\}$, containing the respective terminals
and each meeting $R$.
\end{lemma}
\begin{proof}
Every component of $G-S$ attaches to all three terminals. In
$H=G[D_0\cup\{a,b\}]$, if two disjoint paths from $\{a,b\}$
to distinct members of $R$ did not exist, Menger's theorem would
give a blocking vertex $z$. It is not $a$ or $b$, since the other
terminal reaches $R$ through connected $D_0$. For $z\in D_0$,
choose $r\in R\setminus\{z\}$ and let $L$ be its component in
$H-z$. Then $L$ avoids $a,b$ and $N_G(L)\subseteq\{c,z\}$.
Another component of $G-S$ is outside this separation, contrary
to 3-connectivity. The two paths give the initial supports, both
meeting $D_0$. If necessary join their intersections with $D_0$
by a shortest path in $G[D_0]$ and add its internal vertices to
one support. These interiors avoid both old supports, so this
preserves disjointness and connectivity and supplies adjacency.
\end{proof}

\begin{lemma}[Restoring any two retained separator vertices]
\label{kfive:lem:tworestore}
In the setting of Lemma~\ref{kfour:lem:quotient}, suppose $W$ is
3-connected and $S=\{a,b,c\}$ has three different colours.
After deleting the colour of $c$ (and possibly other whole colours)
from $F$, the remaining separator vertices $t_a,t_b$ can be
realised by disjoint connected adjacent supports in
$\{a,b\}\cup D$, rooted whenever they are odd.
\end{lemma}
\begin{proof}
Here the subscripts name the corresponding colours. If neither
vertex is artificially odd, add any discarded component to one
terminal and leave the other terminal alone. Any odd quotient
vertex then already has its own original root. With one artificial
odd vertex, \eqref{kfour:eq:root-supply} supplies a discarded component
containing a root; add it to that terminal's support. It attaches
to the other terminal, so the supports are adjacent.

With two artificial odd vertices, \eqref{kfour:eq:root-supply} gives two
distinct roots of their two distinct colours in $D$. If some
discarded component has at least two roots, apply
Lemma~\ref{kfive:lem:two} there with $R$ its root set. Otherwise choose
two root-containing discarded components and allocate one to each
terminal. The supports are connected and rooted, and the component
allocated to $a$ attaches to $b$. Thus they are adjacent in every
case, supplying any required quotient edge. These choices impose
no identity or colour on the root assigned to a support.
\end{proof}

\begin{theorem}[An odd side omitting a colour]\label{kfive:thm:missing}
Let $q\ge6$, let $W$ be a 3-connected specified $K_q$ witness,
and let $S$ be a three-vertex separator. If a charge-one component
$K$ of $W-S$ omits a colour, then $W$ has a $P$-rooted $K_5$.
Equivalently, in a 3-connected witness without such a model,
every charge-one component at every three-separator meets every
fibre. No minimality, cleanness or edge-minimality is assumed.
\end{theorem}
\begin{proof}
Let $Y$ be the separator's colours. A colour $z$ omitted by $K$
belongs to $Y$, since each colour outside $Y$ has odd root count
in $K$. Write $D=V(W)\setminus(K\cup S)$.

If $|Y|\le2$, choose four colours outside $Y$. Their vertices in
$K$ induce a specified $K_4$ witness, because $K$ is a component
outside $S$ and each chosen colour has odd charge. The connected
set $B=D\cup T_z$ contains the whole fibre $V_z$: the omitted
colour has no vertices in $K$. Connectivity follows because every
discarded component attaches to every member of $T_z\ne\varnothing$.
Thus $B$ contains an original root and is adjacent to each rooted
branch of the $K_4$ model. It is disjoint from those branches and
supplies the fifth branch.

If $|Y|=3$, name $S=\{a,b,c\}$ so that $a$ has colour $z$.
Construct $F$ by Lemma~\ref{kfour:lem:quotient}. Its colour-$z$ fibre
is the singleton $t_a$, with mark one. Keep the two colours of
$a,b$ and any three colours outside $Y$, deleting all others,
including the whole colour of $c$. This is a specified five-colour
witness $F'$. Deleting its singleton $t_a$ leaves a specified
$K_4$ witness, whose rooted $K_4$ extends by $\{t_a\}$ to a
rooted $K_5$ of $F'$. Lemma~\ref{kfive:lem:tworestore} supplies the
two retained separator supports. Retained vertices in $K$ have
singleton supports. Every selected quotient edge has an original
representative, so the five branches lift to the original set $P$.
\end{proof}

\begin{corollary}[Opposite sides of a double]\label{kfive:cor:opposite}
Let $W$ be a 3-connected specified $K_6$ witness without a
$P$-rooted $K_5$, and let $V_6=\{o,e\}$ be a double. If a
three-separator $S$ contains exactly one member of the double,
every charge-one component of $W-S$ contains the other member.
There is consequently exactly one charge-one component.
\end{corollary}
\begin{proof}
A component not containing the other double vertex has no vertex
of colour 6 and is excluded by Theorem~\ref{kfive:thm:missing} if it
has charge one. Charges sum to one, so an odd component exists;
all such components must contain the same other double vertex.
\end{proof}

\begin{theorem}[Singleton fibre of a missed colour]\label{kfive:thm:missed}
Let $q\ge6$, let $W$ be a $3$-connected specified $K_q$ witness, let $S$ be a separator of $W$ with $|S|=3$ and colour set $Y$, and let $K$ be a component of $W-S$ with $\pi_K=1$. If some colour $z\notin Y$ satisfies $|K\cap V_z|=1$, then $W$ has a $P$-rooted $K_5$.
\end{theorem}
\begin{proof}
Let $K\cap V_z=\{w\}$. The colour $z$ is missed by $S$, so $|P\cap K\cap V_z|\equiv\pi_K=1$ and $w\in P$. Put $D=V(W)\setminus(K\cup S)$, which is nonempty, and form the corrected quotient $F$ keeping $K$ as in \cref{kfour:lem:quotient}. As $z\notin Y$, $F$ has no vertex $t_z$, so its colour-$z$ fibre is $\{w\}$. Every component of $D$ has neighbours in all three vertices of $S$: otherwise its neighbourhood is a separator of $W$ of order at most two.

Choose a set $Z$ of five colours with $z\in Z$ as follows.
\begin{enumerate}[label=(\alph*)]
\item If $|Y|=3$, write $S=\{a,b,c\}$ with colours $c_a,c_b,c_c$. There are $q-4\ge2$ colours outside $Y\cup\{z\}$; pick two, $m_1,m_2$, and put $Z=\{z,m_1,m_2,c_a,c_b\}$.
\item If $|Y|=2$, then $S$ has a class $T_y$ of two vertices and a class $T_{y'}$ of one vertex. There are $q-3\ge3$ colours outside $Y\cup\{z\}$; pick three and put $Z=\{z,m_1,m_2,m_3,y\}$.
\item If $|Y|=1$, there are $q-2\ge4$ colours outside $Y\cup\{z\}$; let $Z$ consist of $z$ and four of them.
\end{enumerate}
Let $F'$ be $F$ restricted to the fibres of the colours in $Z$; every other colour is deleted as a whole fibre. By~\cref{lem:attachment}, $F'$ is a specified $K_5$ witness, its colour-$z$ fibre is the singleton $\{w\}$, and $w$ is adjacent in $F'$ to every odd vertex of $F'$ outside $\{w\}$. Deleting the fibre $\{w\}$ leaves a specified $K_4$ witness, which by~\cref{thm:rooted-four} has a $K_4$ model $M_1,\dots,M_4$ every branch of which contains an odd vertex of $F'$ other than $w$. Each $M_i$ is therefore adjacent to $w$, and $\{w\},M_1,\dots,M_4$ is a $K_5$ model of $F'$ rooted in its odd set.

We now give supports. Retained vertices of $K$ receive themselves. In case~(a), \cref{kfive:lem:tworestore} supplies disjoint connected adjacent supports $X_a\ni a$ and $X_b\ni b$ in $\{a,b\}\cup D$, each meeting $P$ whenever $t_a$, respectively $t_b$, is odd. In case~(b), let $X_y=T_y\cup D_1$ for a component $D_1$ of $D$; it is connected because $D_1$ is connected and adjacent to both vertices of $T_y$. If $t_y$ is odd and $T_y$ contains a root, $X_y$ meets $P$. If $t_y$ is artificially odd, the supply equation gives a root of colour $y$ in $D$; choose $D_1$ to be its component. In case~(c) no separator vertex is retained. All supports are disjoint from each other and from $K$. An edge of $F'$ between two vertices of $K$ is an edge of $W$. An edge $vt_y$ with $v\in K$ means that $v$ has an odd, hence positive, number of neighbours in $T_y\subseteq X_y$. In case~(a) the possible edge $t_at_b$ is represented because $X_a$ and $X_b$ are adjacent. Every odd vertex of $F'$ has a support meeting $P$. Thus $F'$ is a root-preserving minor witness of $W$ by~\cref{kfour:lem:quotient} and its support rule, and expanding the supports turns the model into a $P$-rooted $K_5$ of $W$.
\end{proof}

\begin{corollary}[Separators avoiding a double]\label{kfive:cor:together}
Let $W$ be a $3$-connected specified $K_6$ witness with a double $V_6=\{o,e\}$ and no $P$-rooted $K_5$. If $T$ is a separator of $W$ with $|T|=3$ and $T\cap V_6=\varnothing$, then $o$ and $e$ lie in the same component of $W-T$.
\end{corollary}
\begin{proof}
Colour $6$ is missed by $T$, so the component $K_o$ of $W-T$ containing $o$ has charge $|P\cap K_o\cap V_6|=|\{o\}|=1$. If $e\notin K_o$, then $K_o\cap V_6=\{o\}$, and Theorem~\ref{kfive:thm:missed} with $z=6$ gives a $P$-rooted $K_5$.
\end{proof}

\subsection{Separators meeting a double}
Throughout this subsection $W$ is a specified $K_6$ witness without a $P$-rooted $K_5$, and $V_6=\{o,e\}$ is a double with $p(o)=1$.

\begin{lemma}[Clean double and whole-double separators]
\label{kfive:lem:clean}
Every neighbour of the double is an original root, and every
original root outside $V_6$ is adjacent to exactly one of $o,e$.
If $W$ is 3-connected, no three-separator contains both $o,e$.
\end{lemma}
\begin{proof}
A common neighbour $x$ would make $\{o,x,e\}$ connected and
rooted and containing the whole double. A rooted $K_4$ in the
four colours other than 6 and the colour of $x$ would be disjoint
from it and adjacent to it in every branch, a contradiction.
An external even vertex adjacent to either double vertex is
adjacent to both by its parity into the two-vertex fibre, so none
exists. An external odd vertex has exactly one neighbour there.

For $S=\{o,e,v\}$, the separator misses four colours. Some
component $K$ of $W-S$ has charge one, giving a rooted $K_4$ on
those four colours inside $K$. Another component $D_0$ attaches
to both $o,e$ by 3-connectivity. The connected rooted set
$D_0\cup\{o,e\}$ contains the whole double and is adjacent to
every branch of that $K_4$, giving the excluded $K_5$.
\end{proof}

\begin{proposition}[Reducing a three-separator at a double]
\label{kfive:prop:reduce}
Suppose $W$ is 3-connected and a three-separator meets $V_6$.
Then $W$ has a smaller root-preserving minor witness $F$ to
$K_6$, still with a double of colour 6, and still without a
model rooted in its own odd set of order five.
\end{proposition}
\begin{proof}
By Lemma~\ref{kfive:lem:clean}, the separator contains exactly one
double vertex, say $S=\{s,u,v\}$, and the other is $d$.
Corollary~\ref{kfive:cor:opposite} gives a unique charge-one component
$K$, containing $d$. Every other component has charge zero. Let
$D$ be their union and form the full six-colour quotient $F$ of
Lemma~\ref{kfour:lem:quotient}. Its colour-6 fibre is $\{d,t_6\}$;
\[
 p_F(t_6)=1+|P\cap K_6|=1+p(d)=p(s).
\]
Thus it remains a double with the same odd/even assignment.

Every root of $D$ is adjacent to $s$. Indeed, it is adjacent to
exactly one double vertex, and cannot be adjacent to $d\in K$.
Every component of $D$ contains a root, since it attaches to $s$
and all neighbours of $s$ are roots. We use $X_6=\{s\}$; its
mark is unchanged, so it needs no new root.

If $u,v$ have the same colour, use
$X_{c_u}=\{u,v\}\cup D$. It is connected, rooted, disjoint
from $\{s\}$ and adjacent to it. This realises the full quotient.

Let $c_u\ne c_v$. If $D$ has at least two components, choose $D_1,D_2$ and use
$X_{c_u}=\{u\}\cup D_1$ and
$X_{c_v}=\{v\}\cup D_2$. Both contain roots, each is adjacent
to $s$, and they are adjacent because $D_1$ attaches to $v$.
Unused components can be deleted.

If $D$ is one component with at least two roots, apply
Lemma~\ref{kfive:lem:two} with terminals $u,v$ and the omitted terminal
$s$. Its two supports are rooted and adjacent. Each is adjacent
to $s$ through any root it contains in $D$. Together with
$\{s\}$ they realise every possible quotient edge and every
required original root.

The remaining case is $P\cap D=\{x\}$. The charge of $D$ in
every colour missed by $S$ is zero, and $D_6=\varnothing$.
Thus $x$ has colour $c_u$ or $c_v$; rename so it has colour
$c_u$. Take $X_{c_u}=\{u\}\cup D$ and $X_{c_v}=\{v\}$.
The former is connected, rooted, adjacent to $s$ through $x$,
and adjacent to $v$ because $D$ attaches to $v$. The latter
needs no restored root: if $v$ is not already a root and
$t_{c_v}$ were odd, \eqref{kfour:eq:root-supply} would require a root of
colour $c_v$ in $D$, a contradiction. For any missed colour $m$,
\eqref{kfive:eq:discard} gives
\[
 \varepsilon(6,c_v)=[sv\in E(W)]
   +e_W(\{s\},D_{c_v})+e_W(\{s\},D_m)
   =[sv\in E(W)].
\]
The last two terms vanish: every neighbour of $s$ is a root,
and the only root in $D$ has colour $c_u$. Hence the possible
remaining edge between $\{s\}$ and $\{v\}$ is already present.

All supports contain their full separator class, lie in that
class plus $D$, are disjoint and connected, realise all selected
edges, and are rooted whenever their quotient vertex is odd.
Lemma~\ref{kfour:lem:quotient} gives a smaller root-preserving minor
witness $F$ retaining the double. If $F$ had a model rooted in
its own odd set of order five, it would lift to the prohibited
model in $W$. Therefore $F$ is again a counterexample.
\end{proof}

\subsection{Separators avoiding a double}
We use the following classical rooted-triangle criterion. Its statement is the connected-case contrapositive of Fabila-Monroy and Wood's Lemma~1, attributed there to Wood and Linusson \cite[Lemma~1]{FabilaMonroyWood2013}. We include the short tree argument needed for the support construction.

\begin{lemma}[Three terminals]\label{kfive:lem:three}
Let $H$ be a connected graph and $a,b,c$ distinct vertices of $H$. Suppose that for every vertex $v$ of $H$ some component of $H-v$ contains at least two of $a,b,c$. Then $H$ has disjoint connected sets $X_a\ni a$, $X_b\ni b$, $X_c\ni c$, every two joined by an edge, with $X_t\cap\{a,b,c\}=\{t\}$ for each $t$.
\end{lemma}
\begin{proof}
Let $\tau$ be an inclusion-minimal subtree of $H$ containing $a,b,c$. Its leaves are terminals, so $\tau$ is either a path whose ends are two terminals and which contains the third in its interior, or a subdivided claw with leaves $a,b,c$ and a centre of degree three outside $\{a,b,c\}$. In the first case name the interior terminal $b$ and put $z=b$; in the second let $z$ be the centre. For each terminal $t\ne z$ let $L_t$ be the vertex set of the component of $\tau-z$ containing $t$. It is a path from $t$ to a neighbour of $z$. Put $L_b=\varnothing$ if $z=b$. For $u\in L_t$ let $\tau[t,u]$ be the subpath of $L_t$ from $t$ to $u$.

By hypothesis some component of $H-z$ contains two terminals other than $z$. If $z=b$ these are $a,c$. If $z$ is a claw centre, rename the leaves so that they are $a,c$; the leaves play symmetric roles. Let $\Pi$ be an $a$--$c$ path in $H-z$. Let $x$ be the last vertex of $\Pi$ in $L_a$, and let $y$ be the first vertex of $\Pi$ after $x$ lying in $V(\tau)$. It exists, since $c\in V(\tau)$. By the choice of $x$ and since $\Pi$ avoids $z$, $y\in L_b\cup L_c$, and the interior $\Pi(x,y)$ is disjoint from $\tau$.

If $y\in L_c$, put $X_a=V(\tau[a,x])$, $X_c=V(\tau[c,y])\cup V(\Pi(x,y))$, and let $X_b$ consist of the remaining vertices of $\tau$. If $y\in L_b$ (possible only in the claw case), put $X_a=V(\tau[a,x])$, $X_b=V(\tau[b,y])\cup V(\Pi(x,y))$, and let $X_c$ consist of the remaining vertices of $\tau$. The two subpaths removed from $\tau$ are initial segments of legs, so the remainder is a subtree. It contains $z$ and the terminal not yet placed, which is $b$ itself if $z=b$. The second set is connected because the last vertex of $\Pi(x,y)$ is adjacent to $y$; if $\Pi(x,y)$ is empty, $xy$ is an edge. All three sets are disjoint and each contains exactly one terminal. The neighbour of $x$ on $\tau$ towards $z$, and the neighbour of $y$ on $\tau$ towards $z$, lie in the remainder. The successor of $x$ on $\Pi$ lies in the second set. Hence every two of the three sets are adjacent.
\end{proof}

\begin{lemma}[Root-free discarded side]\label{kfive:lem:rootfree}
Let $W$ be a $3$-connected specified $K_6$ witness with a double $V_6=\{o,e\}$. Let $T$ be a separator with $|T|=3$ and $T\cap V_6=\varnothing$ such that $o$ and $e$ lie in one component $K$ of $W-T$. Then the corrected quotient $F$ keeping $K$ is a root-preserving minor witness of $W$ to $K_6$ with fewer vertices than $W$, and $\{o,e\}$ is a double of $F$.
\end{lemma}
\begin{proof}
Colour $6$ is missed by $T$, so $\pi_K=|P\cap K\cap V_6|=1$. Thus $K$ is a valid selection for~\cref{kfour:lem:quotient} and its support rule, and $D=V(W)\setminus(K\cup T)$ is nonempty. Let $D_1,\dots,D_s$ be the components of $W-T$ other than $K$; each is adjacent to all three vertices of $T$, by $3$-connectivity.

\emph{$D$ contains no root.} A root $x\in D$ lies outside $V_6\subseteq K$, so by~the parity condition into the double it is adjacent to $o$ or to $e$. But $x$ lies in a component of $W-T$ different from $K$, and $o,e\notin T$. This is impossible.

\emph{No quotient vertex is artificially odd.} If $t_y$ were odd with $T_y\cap P=\varnothing$, the supply equation would give a root of colour $y$ in $D$. Hence every support containing $T_y$ meets $P$ whenever $t_y$ is odd.

\emph{No $D_j$ is a single vertex.} If $D_j=\{v\}$, then $p(v)=0$ and $N(v)=T$. The colour classes of $T$ have sizes $3$, or $2$ and $1$, or $1,1,1$, so some class $T_y$ has odd size. Proper colouring implies that $y$ differs from the colour of $v$, since $v$ is adjacent to every vertex of $T$. Hence $|N(v)\cap V_y|=|T_y|$ is odd, contrary to $p(v)=0$.

\emph{Supports.} If $T$ is monochromatic of colour $y$, put $X_y=T\cup D_1$. If $T$ has two colours, with classes $T_y$ of size two and $T_{y'}$ of size one, put $X_y=T_y\cup D_1$ and $X_{y'}=T_{y'}$. Now let $T=\{a,b,c\}$ have three colours. If $s\ge2$, put $X_a=\{a\}\cup D_1$, $X_b=\{b\}\cup D_2$ and $X_c=\{c\}$. If $s=1$, apply Lemma~\ref{kfive:lem:three} to $H=W[D_1\cup T]$ with terminals $a,b,c$. The graph $H$ is connected. For $v\in T$, the graph $H-v$ contains the connected set $D_1$ and the two other terminals, which are adjacent to $D_1$. For $v\in D_1$, suppose every component of $H-v$ contains at most one terminal. A component $L$ with no terminal satisfies $N_W(L)\subseteq\{v\}$, since all $W$-neighbours of $L$ lie in $D_1\cup T$. A component $L$ with one terminal $t$ satisfies $N_W(L\setminus\{t\})\subseteq\{t,v\}$. By $3$-connectivity the first kind does not exist and the second kind is $\{t\}$. Then $D_1=\{v\}$, which was excluded. Lemma~\ref{kfive:lem:three} yields $X_a,X_b,X_c$ with $X_t\subseteq D_1\cup\{t\}$.

In every case the supports are disjoint and connected, satisfy $T_y\subseteq X_y\subseteq T_y\cup D$, and are pairwise adjacent. Every possible edge between separator vertices of $F$ is therefore represented, and every odd $t_y$ has a support meeting $P$. By~\cref{kfour:lem:quotient} and its support rule, $F$ is a root-preserving minor witness of $W$ to $K_6$ on $|K|+|Y|\le|V(W)|-|D|$ vertices. Since $6\notin Y$, the colour-$6$ fibre of $F$ is $K\cap V_6=\{o,e\}$ with marks $1,0$: a double.
\end{proof}

\begin{lemma}[Four-connected double reduction]\label{kfive:lem:fourconn}
Suppose a specified $K_6$ witness $W$ has a double and no $P$-rooted $K_5$. Among its root-preserving minor witnesses \emph{to $K_6$} having a double, choose $W^*$ of minimum order. Then $W^*$ is four-connected, has no model of order five rooted in its own odd set $P^*$, and the odd vertex of each of its doubles is clean.
\end{lemma}
\begin{proof}
The family is nonempty because it contains $W$. A $P^*$-rooted $K_5$ would expand to the prohibited $P$-rooted model in $W$. Fix a double of $W^*$ and call its colour 6 and its vertices $o^*,e^*$.

If $W^*$ has a separator of order at most two, apply the explicit smaller-quotient construction in the proof of \cref{kfour:lem:minimal}. In the quotient, colour 6 has at most two vertices: retained double vertices and at most one vertex representing members of the double in the separator. Its fibre is nonempty by odd charge. If it has one vertex, restrict to that singleton and four other colours; \cref{thm:rooted-four,lem:attachment} give a rooted $K_5$, which lifts to a contradiction. If it has two vertices, odd charge makes it a double. Composition of supports gives a smaller member of the same minimum family, again a contradiction. Thus $W^*$ is three-connected.

By \cref{kfive:lem:clean}, no three-separator contains both double members. A three-separator meeting just one gives the smaller double-preserving quotient of \cref{kfive:prop:reduce}, contrary to minimality after composing supports. A three-separator avoiding the double cannot split its members by \cref{kfive:cor:together}. If it leaves them together, \cref{kfive:lem:rootfree} gives another smaller root-preserving $K_6$ witness with a double, also a contradiction. Hence no separator of order at most three exists. There are six nonempty fibres and a double, so $|V(W^*)|\ge7$, and $W^*$ is four-connected. Its odd double vertex is clean by \cref{kfive:lem:clean}.
\end{proof}
The minimization is over a fixed target and explicitly verified reductions. It does not assert that the original $W$ is four-connected or that either $W^*-o^*$ or $W^*-e^*$ is four-connected.

\subsection{A clean root in a four-connected witness}
A \emph{separation} of $G$ is a pair $(A,B)$ with $A\cup B=V(G)$ and no edge between $A\setminus B$ and $B\setminus A$; its order is $|A\cap B|$. Following \cite[Section~2]{KNTWBounded2012}, a \emph{society} is a graph $G$ with a cyclic order $\Omega$ on a set $V(\Omega)\subseteq V(G)$. A \emph{cross} consists of two vertex-disjoint nontrivial paths, with ends $a,c$ and $b,d$ occurring in this cyclic order in $\Omega$, and with no other vertex in $V(\Omega)$. A society is \emph{four-connected} if there is no separation $(A,B)$ of order at most three with $V(\Omega)\subseteq A$ and $B\setminus A\ne\varnothing$. It is \emph{rural} if $G$ embeds in a closed disc with $V(\Omega)$ on its boundary in the order $\Omega$.

\begin{theorem}[Robertson and Seymour]\label{kfive:thm:rural}
Every four-connected society with no cross is rural.
\end{theorem}
This is the formulation in \cite[Theorem~2.5]{KNTWBounded2012}, attributed there to \cite[(2.3)--(2.4)]{RobertsonSeymour1990}. It is an external theorem, not a new assertion here. We also use Roberson's planarity-preservation theorem: a weak oddomorphic image of a planar graph is planar \cite[Lemma~8.2 in arXiv version~1]{Roberson2022}. In particular, a specified $K_q$ witness with $q\ge5$ is not planar.

\begin{theorem}[Clean hub in a four-connected witness]\label{kfive:thm:hub}
Let $q\ge5$, let $W$ be a $4$-connected specified $K_q$ witness, and let $r$ be a clean root of $W$. Then $W$ has a $P$-rooted $K_5$ model with $\{r\}$ as a branch.
\end{theorem}
\begin{proof}
Put $G=W-r$ and $N=N(r)\subseteq P$. Since $W$ is $4$-connected, $G$ is $3$-connected. If $G$ has a $K_4$ model every branch of which meets $N$, its branches meet $P$ and are adjacent to $r$, so adding $\{r\}$ gives the required model. Suppose not. By~\cref{kfour:lem:cycle}, some cycle $C$ of $G$ contains $N$. Let $\Omega$ be the cyclic order of $N$ along $C$.

\emph{$(G,\Omega)$ has no cross.} If $P_1,P_2$ were a cross with ends $u_1,v_1$ and $u_2,v_2$ alternating in $\Omega$, then $C$ contains $u_1,u_2,v_1,v_2$ in this order and $P_1,P_2$ form a $(u_1v_1,u_2v_2)$-linkage. By~\cite[Lemma~7]{FabilaMonroyWood2013}, $G$ has a $K_4$ model whose branches contain $u_1,u_2,v_1,v_2\in N$, contrary to our assumption.

\emph{$(G,\Omega)$ is $4$-connected.} Let $(A,B)$ be a separation of $G$ of order at most three with $N\subseteq A$ and $B\setminus A\ne\varnothing$. All neighbours of $r$ lie in $A$, so $(A\cup\{r\},B)$ is a separation of $W$ of the same order. Both $r\in(A\cup\{r\})\setminus B$ and $B\setminus A$ are nonempty, contrary to $4$-connectivity.

By~\cref{kfive:thm:rural}, $(G,\Omega)$ is rural: $G$ is drawn in a closed disc $\Delta$ with every vertex of $N$ on $\partial\Delta$. Place $r$ in the complementary open disc of the sphere and join it to the vertices of $N$ by pairwise internally disjoint arcs in that region. This is a planar drawing of $W$, since $N(r)=N$. By~Roberson's planarity-preservation theorem, $W$ is not a specified $K_q$ witness, a contradiction.
\end{proof}

The unrooted part of this argument is a consequence of the same classical theorems for every four-connected nonplanar graph: every vertex can be a singleton branch of a $K_5$ model. Here cleanliness is what ensures that all four other branches meet the specified odd set. No novelty is claimed for the unrooted consequence.

\subsection{The small-fibre theorem and its consequences}
\begin{proof}[Proof of Theorem~\ref{thm:small-fibre}]
If a fibre is a singleton $\{s\}$, use \cref{thm:rooted-four} on four other colours and add $\{s\}$ by \cref{lem:attachment}. Otherwise the small fibre is a double. Suppose there is no $P$-rooted $K_5$, and choose $W^*$ as in \cref{kfive:lem:fourconn}. It is four-connected and the odd vertex $o^*$ of its chosen double is clean. By \cref{kfive:thm:hub} it has a $P^*$-rooted $K_5$. Expanding the composed supports gives a $P$-rooted $K_5$ in the original $W$, a contradiction. If $|V(W)|\le17$, one of the six nonempty fibres has size at most two, so the first assertion gives the numerical conclusion. The model also lies in every graph containing $W$.
\end{proof}

\begin{corollary}\label{kfive:cor:order}
\begin{enumerate}[label=(\alph*)]
\item Every specified $K_6$ witness on at most seventeen vertices has a $P$-rooted $K_5$. A specified $K_6$ witness without a $P$-rooted $K_5$, if one exists, has every fibre of size at least three and at least eighteen vertices.
\item For every $q\ge6$, every specified $K_q$ witness with a fibre of at most two vertices has a $P$-rooted $K_5$.
\item Let $q\ge6$. If a specified $K_q$ witness has at least $q-6$ singleton fibres and at least $q-5$ fibres of at most two vertices, it has a $P$-rooted $K_{q-1}$. In particular, a specified $K_7$ witness with a singleton fibre and at most eighteen vertices has a $P$-rooted $K_6$. A specified $K_8$ witness with two singleton fibres and at most nineteen vertices has a $P$-rooted $K_7$. A specified $K_9$ witness with three singleton fibres and at most twenty vertices has a $P$-rooted $K_8$.
\end{enumerate}
\end{corollary}
\begin{proof}
(a) Six fibres of size at least three need eighteen vertices. (b) Restrict to six colours including the small fibre, using \cref{lem:attachment}; the restricted odd set is contained in $P$. (c) Delete $q-6$ singleton fibres. The remaining specified $K_6$ witness still has a fibre of at most two vertices and, by Theorem~\ref{thm:small-fibre}, a $K_5$ model rooted in its odd set, which is contained in $P$. By~\cref{lem:attachment}, each deleted singleton $s$ is a root adjacent to every root outside $\{s\}$, hence to the other singletons and to a root of each of the five branches. The five branches and the $q-6$ singletons form a $P$-rooted $K_{q-1}$. For the particular cases, the six-colour core has at most seventeen vertices, so (a) applies.
\end{proof}

\begin{remark}[Scope and comparison]
The conclusion is set-rooted, not a prescribed-transversal assertion. It neither proves that every specified $K_6$ witness has a $P$-rooted $K_5$ nor settles unrestricted weak $K_7\to K_6$ or weak $K_9\to K_8$. No counterexample at order eighteen is constructed. If one exists at that order, all six fibres have size three by \cref{kfive:cor:order}(a); sharpness of seventeen is not established.

The earlier excess-six theorem already gives additional rooted-five conclusions by colour deletion. If a six-colour witness has order $n$ and a fibre of size $M\ge n-11$, deleting that fibre leaves a five-colour witness of order at most eleven, so \cref{thm:main} gives a $P$-rooted $K_5$. Together with its direct application for $n\le12$, this covers every witness of order at most fourteen. The content of \cref{thm:small-fibre} is the all-order double-fibre criterion with the other five fibre sizes arbitrary; \cref{kfive:cor:order}(a) gives the universal bound seventeen without using excess six.
\end{remark}

\endgroup
\section*{Status of the proofs}
\addcontentsline{toc}{section}{Status of the proofs}
\Cref{thm:main,thm:sharp} have written proofs in this paper. Besides elementary graph theory they use three external results, each cited at the point of use: the classical fact that a finite simple graph of minimum degree at least three contains a subdivision of $K_4$; Kriesell and Mohr's routing theorem for at most six doubled vertices \cite[Theorem~6, Section~5, p.~12]{KriesellMohr2022}, used only for cores consisting entirely of double fibres (direct proofs are given in Appendix~G of the technical companion); and their doubled theta construction \cite[Section~3, pp.~8--9, immediately before Theorem~2]{KriesellMohr2022}, whose rooted obstruction we prove in \cref{lem:known-theta} by their counting argument. No computation is a premise. Exhaustive enumerations of the fifteen-vertex graph and of all excess-six cores were run as independent corroboration; they are documented in the companion.

\Cref{thm:rooted-four} has a separate complete proof in \cref{sec:rooted-four}. It uses elementary parity calculations, Menger's theorem and the fan lemma, together with Fabila-Monroy and Wood's crossing-linkage lemma \cite[Lemma~7]{FabilaMonroyWood2013}. The needed case of the cycle principle of Groenland, Nederlof and Koana \cite[Lemma~4]{GroenlandNederlofKoana2024} is proved in full. This argument uses neither \cref{thm:main} nor the Weak Roberson theorem, and no computation is a premise.

\Cref{thm:small-fibre} has a complete proof in \cref{sec:small-fibre}. Besides the earlier rooted-four proof and elementary parity and linkage arguments, it uses the classical three-terminal criterion, the rural-society theorem in the form of \cite[Theorem~2.5]{KNTWBounded2012}, attributed to \cite[(2.3)--(2.4)]{RobertsonSeymour1990}, and Roberson's planarity-preservation theorem \cite[Lemma~8.2 in arXiv version~1]{Roberson2022}. It has no computational premise and does not use \cref{thm:main} or the main theorem of \cite{WeakRoberson2026}. The rural-society theorem's precise statement and hypotheses were checked in the cited formulation; its long original proof was not reverified for this revision.

AI tools assisted with mathematical exploration, proof development, computational checks and manuscript preparation. The author takes responsibility for the final content.

\paragraph{Code and certificate archive.}
The supporting code and certificate archive is identified by the Zenodo DOI \url{https://doi.org/10.5281/zenodo.23603860}~\cite{CertikSupport2026}. No computation is a premise of the core results; the archive documents corroboration and the finite inputs of the technical companion.

\clearpage
\phantomsection\addcontentsline{toc}{part}{Technical companion}

\restoremaketitle
\title{Minimum orders of clique oddomorphism counterexamples:\\ technical companion}
\author{Michal \v{C}ert\'{\i}k\\ World Bank Group\\ Washington, DC, USA\\ \href{mailto:michal.certik@gmail.com}{\texttt{michal.certik@gmail.com}}}
\date{17 October 2026}
\maketitle
\begin{abstract}
This companion collects results and evidence that accompany, but are not used in, the core paper \emph{Minimum orders of clique oddomorphism counterexamples}. Within a specified thirty-one-member family of root-apex graphs we determine exactly which have no $K_8$ minor, giving three extremal families. We determine exactly when an even vertex attached to eight double fibres forces a $K_8$ minor, and we classify, for the three hard root-apex cores, when one extra apex--even edge admits a $K_8$ model with a singleton apex branch; the negative part of that classification has a structural proof. We also record a bounded computer-assisted nine-colour theorem, the direct proofs of the all-double cases that the core paper cites from Kriesell and Mohr, and the computational provenance of every finite statement.
\end{abstract}
\begingroup\small
\noindent\textbf{How to read this document.} Section and statement numbers continue those of the core paper: Sections~1--7, and every statement numbered there (for instance \cref{thm:main}, \cref{lem:double} or \cref{prop:noK8}), refer to the core paper. Sections~8--11 and Appendices~A--G are in this companion. Relative to the original combined article, insertion of main-article Sections~6 and~7 shifts the former companion Sections~6--9 by two; appendix letters and mathematical content are unchanged. Each result states its evidence: \emph{structural} (a written proof), \emph{certificate-backed} (a written reduction to explicitly checked finite data), or \emph{inherited} (finite data and checks from earlier archived legs, not re-executed here).
\par\endgroup
\clearpage
\setcounter{section}{7}
\setcounter{equation}{9}
\section{Root-apex obstructions: three extremal families}\label{sec:root-apex}
The graph $H$ of \cref{sec:sharpness} of the core paper is one member of a natural family. This section determines exactly which members of a specified thirty-one-member family exclude an ordinary $K_8$ minor, and gives a structural proof for each excluded case.

\subsection{Root-apex graphs and the specified family}
For a graph $A$ on $\{0,\ldots,6\}$ let $K(A)$ be the graph on $o_0,\ldots,o_6,e_0,\ldots,e_6$ in which, for $i<j$, the pair $\{i,j\}$ contributes the single edge $o_io_j$ if $ij\notin E(A)$ and the three edges $o_ie_j,e_io_j,e_ie_j$ if $ij\in E(A)$; there are no other edges. Let $H(A)$ be obtained from $K(A)$ by adding a vertex $s$ adjacent exactly to $o_0,\ldots,o_6$. Put $P=\{o_0,\ldots,o_6\}$. By \cref{lem:double,lem:attachment}, the fibres $\{o_i,e_i\}$ and $\{s\}$ with odd set $P\cup\{s\}$ form an oddomorphism witness of $H(A)$ to $K_8$; every $o_i$ and $s$ has degree seven and $e_i$ has degree $2\deg_A(i)$.

A graph $A$ is encoded by the integer $\sum_{ij\in E(A)}2^{\beta(i,j)}$, where
\[
 \beta(i,j)=\tfrac12 i(13-i)+j-i-1\qquad(0\le i<j\le6)
\]
numbers the pairs lexicographically. We write $H(m)$ for $H(A)$ when $A$ has code $m$. The \emph{specified family} consists of the thirty-one codes
\begin{equation}\label{eq:thirtyone}
\begin{gathered}
42232,\ 42234,\ 42238,\ 42486,\ 44280,\ 44281,\ 44283,\ 44287,\ 44340,\ 44341,\ 44348,\\
44403,\ 44404,\ 44405,\ 44412,\ 44472,\ 44476,\ 44536,\ 44856,\ 44858,\ 44922,\ 44924,\\
48564,\ 48628,\ 48756,\ 48820,\ 48824,\ 48828,\ 61112,\ 111980,\ 112110.
\end{gathered}
\end{equation}
Each code is the least code in its isomorphism class. By the inherited seven-double classification of the historical archive, these are exactly the isomorphism types of seven-vertex graphs $A$ for which $K(A)$ has no $P$-rooted $K_7$; all other $1{,}013$ types have explicit rooted models. That classification is used only to interpret \cref{thm:root-apex} as a statement about witnesses of profile $(1,2^7)$, and its positive half is the input \ref{input:P7} of \cref{sec:attachment}; the proofs of \cref{prop:noK8-44340,prop:noK8-44348} and \cref{thm:root-apex} do not use it.

\begin{lemma}\label[lemma]{lem:rooted-ordinary}
$H(A)$ has a $K_8$ model whose branches all meet $P\cup\{s\}$ if and only if $K(A)$ has a $K_7$ model whose branches all meet $P$. If $K(A)$ has no such model, every $K_8$ model of $H(A)$ has a branch containing no vertex of $P\cup\{s\}$.
\end{lemma}
\begin{proof}
Eight disjoint branches meeting the eight vertices of $P\cup\{s\}$ contain exactly one of them each. Every neighbour of $s$ lies in $P$, so the connected branch containing $s$ is $\{s\}$, and deleting it leaves the required model of $K(A)$. Conversely, $\{s\}$ is adjacent to every branch containing a vertex of $P$. For the last statement, a $K_8$ model in which every branch meets the eight vertices of $P\cup\{s\}$ would be of the first kind.
\end{proof}

\begin{lemma}[Spanning normal form]\label[lemma]{lem:normal-form}
Let $G$ be a connected graph on fifteen vertices with a $K_8$ minor. Then $G$ has a $K_8$ model whose branches partition $V(G)$. In any such model let $S$ be the set of vertices forming singleton branches and $q=|S|$. Then:
\begin{enumerate}[label=(\alph*)]
\item $S$ is a nonempty clique of vertices of degree at least seven, and
\begin{equation}\label{eq:budget15}
 \sum_{B\text{ nonsingleton}}(|B|-2)=q-1,
\end{equation}
so every nonsingleton branch has at most $q+1$ vertices;
\item if $x\in S$ has degree seven, every branch other than $\{x\}$ contains exactly one neighbour of $x$; hence every nonsingleton branch contains exactly one vertex of $N_x:=N(x)\setminus S$;
\item if $x\in S$ has degree eight, every branch other than $\{x\}$ contains one or two neighbours of $x$, and exactly one of them contains two;
\item a branch with two vertices is an edge $uv$ with $|N(\{u,v\})|\ge7$, where $N(X)$ denotes the set of vertices outside $X$ adjacent to $X$.
\end{enumerate}
\end{lemma}
\begin{proof}
As in the proof of \cref{prop:noK8}, repeatedly add an unused vertex adjacent to a branch to that branch; connectivity of $G$ makes this possible until the branches partition $V(G)$, and it preserves the model. If $q=0$, eight branches of size at least two would need sixteen vertices. Singleton branches are pairwise adjacent, and each is adjacent to the seven other branches, so $S$ is a clique of vertices of degree at least seven. The $8-q$ nonsingleton branches contain $15-q$ vertices, which gives \eqref{eq:budget15}. For (b), the seven other branches are disjoint and each contains a neighbour of $x$; with exactly seven neighbours each contains exactly one. The vertices of $S\setminus\{x\}$ are neighbours of $x$ and form their own branches, which gives the second sentence. For (c), the eight neighbours of $x$ lie in the seven other branches, each containing at least one. For (d), a branch is adjacent to seven disjoint branches.
\end{proof}

For $x\in S$ of degree seven we call $N_x$ the \emph{index set} of $x$. A nonsingleton branch containing a vertex of $N_x$ contains no other vertex of $N_x$. A vertex outside $S$ and outside every index set is \emph{free}. We write $B(v)$ for the branch containing $v$.

\subsection{The classification}
\begin{theorem}[Root-apex classification]\label{thm:root-apex}
Let $m$ be one of the thirty-one codes \eqref{eq:thirtyone}. Then $H(m)$ has no $K_8$ minor if and only if $m\in\{42232,44340,44348\}$.
\end{theorem}
\begin{proof}
For each of the other twenty-eight codes, \cref{tab:positive-models} in Appendix~\ref{app:root-apex-data} lists eight vertex sets; each is connected and every two are adjacent, so each row is a $K_8$ model. The graph of \cref{sec:sharpness} is $H(42232)$: the bijection $5\mapsto0$, $6\mapsto1$, $2\mapsto6$, $0\mapsto5$, $3\mapsto2$, $1\mapsto4$, $4\mapsto3$ maps $E(\Theta)$ onto the active edges $04,05,06,12,13,16,25,34$ of code $42232$, and acting on both odd and even indices while fixing $s$ it is an isomorphism. Hence \cref{prop:noK8} excludes $H(42232)$. \Cref{prop:noK8-44340,prop:noK8-44348} below exclude the other two.
\end{proof}

The three active graphs are pairwise nonisomorphic: code $42232$ is $\Theta(2,3,3)$ (internally disjoint paths of lengths $2,3,3$ between $0$ and $1$), code $44340$ is $\Theta(1,3,4)$ (paths of lengths $1,3,4$ between $0$ and $3$), and code $44348$ has nine edges. The hosts $H(42232)$ and $H(44340)$ have $44$ edges and $H(44348)$ has $46$. In the first two hosts exactly two vertices have degree six, namely $e_0,e_1$ and $e_0,e_3$ respectively; they are nonadjacent in $H(42232)$ and adjacent in $H(44340)$.

\begin{corollary}[Three extremal families]\label[corollary]{cor:three-families}
For every $n\ge8$ there are three pairwise nonisomorphic graphs of order $n+7$, each with an oddomorphism to $K_n$ and no $K_n$ minor.
\end{corollary}
\begin{proof}
Apply the singleton extension in the proof of \cref{thm:sharp} to each of $H(42232)$, $H(44340)$ and $H(44348)$: add $n-8$ new vertices, adjacent to one another and, among the old vertices, exactly to $s,o_0,\ldots,o_6$, each forming a new singleton fibre. The proof of \cref{thm:sharp} shows that this gives an oddomorphism to $K_n$ and, since each host has Hadwiger number at most seven by \cref{thm:root-apex}, no $K_n$ minor. Each extension adds the same number of edges, so the third extension has two more edges than the first two. In the extensions, the new vertices and the old odd vertices have degree $n-1\ge7$ and the even degrees are unchanged; hence the two vertices of degree six are unchanged, nonadjacent in the first extension and adjacent in the second.
\end{proof}

\Cref{thm:root-apex} concerns the displayed thirty-one codes. Together with the inherited seven-double classification it determines which fifteen-vertex witnesses to $K_8$ of profile $(1,2^7)$ exclude a $K_8$ minor: such a witness is $H(A)$ for its active graph $A$ (the singleton is adjacent exactly to the odd double vertices), and if $K(A)$ has a rooted $K_7$ then \cref{lem:rooted-ordinary} gives a $K_8$ model. It is not a classification of all fifteen-vertex counterexamples, whose fibre profiles may differ.

\subsection{The \texorpdfstring{$\Theta(1,3,4)$}{Theta(1,3,4)} host \texorpdfstring{$H(44340)$}{H(44340)}}
Here $E(A)=\{03,05,06,14,16,23,25,34\}$. Write the active graph as three internally disjoint $a$--$b$ paths
\[
 ab,\qquad a-x-y-b,\qquad a-u-v-w-b,\qquad (a,b,x,y,u,v,w)=(0,3,5,2,6,1,4).
\]
The active neighbourhoods are
\[
\begin{array}{c|ccccccc}
i&0&1&2&3&4&5&6\\\hline
N_A(i)&\{3,5,6\}&\{4,6\}&\{3,5\}&\{0,2,4\}&\{1,3\}&\{0,2\}&\{0,1\}
\end{array}
\]
The vertices $o_i$ and $s$ have degree seven, $e_0,e_3$ have degree six, and the other even vertices have degree four.

\begin{lemma}\label[lemma]{lem:rooted-44340}
$K(A)$ has no $K_7$ model whose branches all meet $P$.
\end{lemma}
\begin{proof}
Suppose such a model exists. Each branch contains exactly one $o_i$; call it $B_i$. Inactive edges join two odd vertices, so they are never internal to a branch, and for $ij\in E(A)$ the edge $o_io_j$ is absent. Let $I$ index the singleton branches $\{o_i\}$. They are pairwise adjacent, so $I$ is independent in $A$. If $i\in I$ and $j\in N_A(i)$, then $B_j$ has at least three vertices: a branch $\{o_j,e_k\}$ adjacent to $o_i$ would need $jk,ik\in E(A)$, and with $ij$ this is a triangle, whereas $A$ is triangle-free. Counting the at most fourteen vertices,
\[
 14\ge|I|+3|N_A(I)|+2(7-|I|-|N_A(I)|)=14+|N_A(I)|-|I|.
\]
The cycle $C=a\,x\,y\,b\,w\,v\,u\,a$ spans $A$. For a nonempty independent $I$, the $2|I|$ edges of $C$ leaving $I$ end in $N_C(I)$, at most two at each vertex, so $|N_C(I)|\ge|I|$; equality would make $I\cup N_C(I)$ closed under $C$-adjacency, hence equal to $V(C)$, and would make the odd cycle $C$ bipartite. Thus $|N_A(I)|\ge|N_C(I)|>|I|$ for nonempty $I$, and so $I=\varnothing$.

All branches are therefore pairs and use all fourteen vertices: $B_i=\{o_i,e_{f(i)}\}$ for a permutation $f$ with $if(i)\in E(A)$. It has no fixed point, and as $A$ is triangle-free it has no $3$-cycle, so its cycle type is $(7)$ or $(5)(2)$. For type $(7)$, $f$ orients the unique Hamiltonian cycle $C$; in either orientation none of $af(b)$, $f(a)b$, $f(a)f(b)$ is an edge of $A$, so $B_a,B_b$ are not adjacent although $ab\in E(A)$. For type $(5)(2)$ the $5$-cycle is $a\,b\,w\,v\,u\,a$ and $f$ transposes $x,y$. A crossing between $B_a$ and $B_x=\{o_x,e_y\}$ forces $f(a)=b$, and a crossing between $B_b$ and $B_y=\{o_y,e_x\}$ forces $f(b)=a$; both cannot hold on a $5$-cycle.
\end{proof}

\begin{proposition}\label[proposition]{prop:noK8-44340}
$H(44340)$ has no $K_8$ minor.
\end{proposition}
\begin{proof}
Suppose otherwise and take a spanning model as in \cref{lem:normal-form}. The vertices of degree at least seven are $s,o_0,\ldots,o_6$, so $S\subseteq P\cup\{s\}$ and \cref{lem:normal-form}(b) applies to every singleton. If $s\in S$, the seven other branches contain exactly one odd double vertex each and contradict \cref{lem:rooted-44340}. Otherwise write $S=\{o_i:i\in I\}$ with $k=|I|$. The branch $B(s)$ contains a neighbour $o_j$ of $s$; as $s$ is the only neighbour of each $o_i$ ($i\in I$) in $B(s)$, $o_j$ is not adjacent to $o_i$, so $I\subseteq N_A(j)$. Also $I$ is independent. The automorphism $(0\,3)(2\,5)(4\,6)$ of $A$, acting on both odd and even indices and fixing $s$, leaves ten orbits of possible $I$, with representatives
\[
 \{0\},\{1\},\{2\},\{4\};\quad\{0,1\},\{0,2\},\{0,4\},\{2,4\},\{4,6\};\quad\{0,2,4\},
\]
covering all eighteen labelled possibilities. Call $v$ a \emph{partner} of $u$ if $\{u,v\}$ satisfies \cref{lem:normal-form}(b),(d): it is an edge of $H-S$, has at least seven external neighbours, and contains exactly one neighbour of each vertex of $S$. By \eqref{eq:budget15} all nonsingleton branches are pairs when $k=1$, and exactly one is a triple when $k=2$.
\begin{itemize}
\item $I=\{0\}$: both $e_5$ and $e_6$ have $e_0$ as their only partner.
\item $I=\{1\}$ and $I=\{4\}$: $e_1$ has no partner.
\item $I=\{2\}$: the only partner of $e_5$ is $e_0$ and the only partner of $e_2$ is $e_3$, forcing $e_0e_5$ and $e_2e_3$. The partners of $e_4$ are $o_1,e_3$, forcing $o_1e_4$. No vertex of $\{o_1,e_4\}$ is adjacent to a vertex of $\{e_0,e_5\}$.
\item $I=\{0,1\}$: $e_1,e_2,o_3,e_5$ have no partner and cannot all lie in the unique triple.
\item $I=\{0,4\}$: $e_1,o_2,e_5,o_6,e_6$ have no partner. $I=\{4,6\}$: $o_0,e_1,o_3,e_4,e_6$ have no partner.
\item $I=\{2,4\}$: $e_1,e_5$ have no partner, so both lie in the unique triple; but they are nonadjacent with no common neighbour.
\item $I=\{0,2\}$: $o_6,e_6$ have no partner, so they lie in the triple, whose third vertex is a common neighbour outside $S$, namely $e_0$ or $e_1$. If it is $e_0$, then $e_5$, whose only partner is $e_0$, cannot be placed. So the triple is $\{e_1,o_6,e_6\}$, and the pairs $e_0e_5$, $e_2e_3$ (the only partner of $e_2$) and $o_1e_4$ (the partners of $e_4$ are $o_1,e_3$) are forced; again $\{o_1,e_4\}$ and $\{e_0,e_5\}$ are not adjacent.
\item $I=\{0,2,4\}$: the neighbourhoods in $S$ of the other vertices are
\[
\begin{array}{c|cccccc}
\text{in }S&\varnothing&\{o_0,o_2\}&\{o_4\}&\{o_2,o_4\}&\{o_0\}&S\\\hline
\text{vertices}&e_0,e_2,o_3,e_4&o_1,e_5&e_1,o_5&o_6&e_6&e_3,s
\end{array}
\]
By \cref{lem:normal-form}(b) each nonsingleton branch contains exactly one neighbour of each of $o_0,o_2,o_4$; equivalently, the neighbourhoods in $S$ of its vertices are pairwise disjoint and their union is $S$. So $B(o_6)$ contains a vertex with neighbourhood $\{o_0\}$, namely $e_6$; as $o_6e_6$ is not an edge and the only neighbour of $e_6$ with empty neighbourhood in $S$ is $e_0$, the branch contains $o_6,e_0,e_6$ and has at least three vertices. Neither $e_1$ nor $e_5$ has a partner, and neither can join $B(o_6)$. If they lie in different branches, three branches have at least three vertices, exceeding the budget $k-1=2$. If they lie together, being nonadjacent with no common neighbour they need a branch of at least four vertices, again exceeding the budget together with $B(o_6)$.
\end{itemize}
All cases are impossible.
\end{proof}

\subsection{The host \texorpdfstring{$H(44348)$}{H(44348)} with an even vertex of degree eight}\label{subsec:44348}
Throughout this subsection $A$ is the active graph of code $44348$,
\[
 E(A)=\{03,04,05,06,14,16,23,25,34\},
\]
and $H=H(A)$. Equivalently, $A$ is the $7$-cycle $C=0\,5\,2\,3\,4\,1\,6\,0$ together with the chords $03,04$; its only triangle is $\{0,3,4\}$. The permutation $\sigma=(1\,2)(3\,4)(5\,6)$ is an automorphism of $A$; acting on odd and even indices and fixing $s$, it is an automorphism of $H$ fixing $o_0,e_0,s$. The neighbourhoods in $H$ are
\[
\begin{array}{l|l@{\qquad}l|l}
o_0&o_1,o_2,e_3,e_4,e_5,e_6,s & e_0&o_3,o_4,o_5,o_6,e_3,e_4,e_5,e_6\\
o_1&o_0,o_2,o_3,o_5,e_4,e_6,s & e_1&o_4,o_6,e_4,e_6\\
o_2&o_0,o_1,o_4,o_6,e_3,e_5,s & e_2&o_3,o_5,e_3,e_5\\
o_3&o_1,o_5,o_6,e_0,e_2,e_4,s & e_3&o_0,o_2,o_4,e_0,e_2,e_4\\
o_4&o_2,o_5,o_6,e_0,e_1,e_3,s & e_4&o_0,o_1,o_3,e_0,e_1,e_3\\
o_5&o_1,o_3,o_4,o_6,e_0,e_2,s & e_5&o_0,o_2,e_0,e_2\\
o_6&o_2,o_3,o_4,o_5,e_0,e_1,s & e_6&o_0,o_1,e_0,e_1\\
s&o_0,o_1,o_2,o_3,o_4,o_5,o_6
\end{array}
\]
Thus $o_0,\ldots,o_6,s$ have degree seven, $e_0$ has degree eight, $e_3,e_4$ have degree six, and $e_1,e_2,e_5,e_6$ have degree four. A direct check of the table shows:
\begin{enumerate}[label=(W)]
\item the only edges $uv$ of $H$ with $|N(\{u,v\})|<7$ are $e_1e_6$ and $e_2e_5$, each with six external neighbours.
\end{enumerate}
By \cref{lem:normal-form}(d), neither $\{e_1,e_6\}$ nor $\{e_2,e_5\}$ is ever a branch. The proof repeatedly uses one more fact from the table: no vertex of $\{o_1,e_6\}$ is adjacent to a vertex of $\{e_2,e_3\}$.

\begin{lemma}[Rooted obstruction for $44348$]\label[lemma]{lem:rooted-44348}
$K(A)$ has no $K_7$ model whose branches all meet $P$.
\end{lemma}
\begin{proof}
Suppose there is one. Each branch contains exactly one odd vertex; let $B_i$ be the branch of $o_i$. Each $e_k$ is adjacent to $o_j$ for every $j\in N_A(k)\ne\varnothing$, and $o_j$ lies in a branch; so unused vertices, which are all even, can be added one at a time to branches. Hence we may assume $B_i=\{o_i\}\cup E_i$, where $E_0,\ldots,E_6$ partition $\{e_0,\ldots,e_6\}$. For $ij\in E(A)$ the edge $o_io_j$ is absent, so $B_i$ and $B_j$ must be joined by an edge $o_ie_k$ with $e_k\in E_j$, $e_ko_j$ with $e_k\in E_i$, or $e_ke_l$; each such edge comes from an active pair. Let $I=\{i:E_i=\varnothing\}$. Counting even vertices,
\begin{equation}\label{eq:rooted-budget}
 \sum_{j\notin I}(|E_j|-1)=|I|.
\end{equation}

\emph{Case $I=\varnothing$.} Then $E_i=\{e_{f(i)}\}$ for a permutation $f$, and $if(i)\in E(A)$ because $B_i$ is connected. The cycles of $f$ have length at least two; a $2$-cycle is an edge of $A$ and a longer cycle is a cycle of $A$. Since $7$ is odd, some cycle is an odd cycle of $A$: the triangle $034$, one of the $5$-cycles $0\,5\,2\,3\,4$ and $0\,3\,4\,1\,6$, or $C$. These are all the odd cycles of $A$. Indeed, $A$ is connected with nine edges and seven vertices, so its cycle space has dimension three. It has the basis formed by the triangle $034$ and the $4$-cycles $0\,5\,2\,3$ and $0\,6\,1\,4$. Of its seven nonzero elements, the sum of the two $4$-cycles has degree four at $0$ and is not a cycle. The other six are the cycles $034$, $0\,5\,2\,3$, $0\,6\,1\,4$, $0\,5\,2\,3\,4$, $0\,3\,4\,1\,6$ and $C$, and the edge set of every cycle of $A$ is one of these elements. For the triangle the complementary pairs are $16$ and $25$; for either $5$-cycle the remaining pair is the corresponding one of those two edges. They are the only edges of $A$ inside $\{1,2,5,6\}$. This leaves exactly eight permutations, and each violates an active adjacency:
\[
\begin{array}{l|c|c}
f&\text{active pair }ij&\text{nonadjacent branches}\\\hline
(0\,3\,4)(1\,6)(2\,5)&06&\{o_0,e_3\},\{o_6,e_1\}\\
(0\,4\,3)(1\,6)(2\,5)&05&\{o_0,e_4\},\{o_5,e_2\}\\
(0\,5\,2\,3\,4\,1\,6)&04&\{o_0,e_5\},\{o_4,e_1\}\\
(0\,6\,1\,4\,3\,2\,5)&03&\{o_0,e_6\},\{o_3,e_2\}\\
(0\,5\,2\,3\,4)(1\,6)&06&\{o_0,e_5\},\{o_6,e_1\}\\
(0\,4\,3\,2\,5)(1\,6)&14&\{o_1,e_6\},\{o_4,e_3\}\\
(0\,6\,1\,4\,3)(2\,5)&05&\{o_0,e_6\},\{o_5,e_2\}\\
(0\,3\,4\,1\,6)(2\,5)&23&\{o_2,e_5\},\{o_3,e_4\}
\end{array}
\]
Here $(i_1\,i_2\,\ldots)$ means $f(i_1)=i_2,\ldots$. In each row the listed sets have no edge between them, by the table above.

\emph{Case $I\ne\varnothing$.} Singleton branches are pairwise adjacent, so $I$ is independent in $A$. Let $i\in I$ and $j\in N_A(i)$. The branch $B_j$ contains a neighbour of $o_i$, and $o_j$ is not one, so $E_j$ contains some $e_k$ with $k\in N_A(i)$. If moreover $|E_j|=1$, then $jk\in E(A)$ as $B_j$ is connected, so $\{i,j,k\}$ is a triangle, that is $\{i,j,k\}=\{0,3,4\}$. Let $X=\{j\in N_A(I):|E_j|=1\}$, where $N_A(I)=\bigcup_{i\in I}N_A(i)$. As $I$ is independent, $N_A(I)\cap I=\varnothing$, so $|E_j|\ge2$ for every $j\in N_A(I)\setminus X$. Thus $X\subseteq\{0,3,4\}\setminus I$, and \eqref{eq:rooted-budget} gives
\begin{equation}\label{eq:rooted-X}
 |X|\ge|N_A(I)|-|I|.
\end{equation}
For nonempty independent $I$ we have $|N_A(I)|\ge|N_C(I)|>|I|$ by the argument in the proof of \cref{lem:rooted-44340}, since $C$ is an odd cycle spanning $A$. Hence $X\ne\varnothing$, and $I$ contains exactly one vertex of the triangle $\{0,3,4\}$.

Suppose $0\in I$. Then $I\in\{\{0\},\{0,1\},\{0,2\},\{0,1,2\}\}$, $N_A(I)=\{3,4,5,6\}$, $X\subseteq\{3,4\}$, and a pair branch at $3$ or $4$ must be $\{o_3,e_4\}$ or $\{o_4,e_3\}$, respectively. If $1\in I$ then $4\notin X$, because $4\in N_A(1)$ and $\{o_4,e_3\}$ contains no neighbour of $o_1$; similarly $2\in I$ gives $3\notin X$. So $|X|\le2,1,1,0$ respectively, while \eqref{eq:rooted-X} requires $|X|\ge3,2,2,1$.

By the symmetry $\sigma$, it remains to suppose $3\in I$. The non-neighbours of $3$ are $1,5,6$, and $16\in E(A)$, so $I$ is one of $\{3\},\{1,3\},\{3,5\},\{3,6\},\{1,3,5\},\{3,5,6\}$, with $|N_A(I)|-|I|$ equal to $2,2,1,2,1,1$. Now $X\subseteq\{0,4\}$ and a pair branch at $0$ or $4$ is $\{o_0,e_4\}$ or $\{o_4,e_0\}$. The first contains no neighbour of $o_5$ or $o_6$, although $0\in N_A(5)\cap N_A(6)$; so $0\notin X$ when $I$ meets $\{5,6\}$. The second contains no neighbour of $o_1$, although $4\in N_A(1)$; so $4\notin X$ when $1\in I$. By \eqref{eq:rooted-X} only two situations survive: $I=\{3\}$ with $X=\{0,4\}$, and $I\in\{\{3,5\},\{3,5,6\}\}$ with $X=\{4\}$.
\begin{itemize}
\item $I=\{3\}$ and $X=\{0,4\}$: $E_0=\{e_4\}$, $E_4=\{e_0\}$, and by \eqref{eq:rooted-budget} $|E_2|=2$ and $|E_j|=1$ for $j\ne2,3$. Then $E_5=\{e_k\}$ with $k\in N_A(5)=\{0,2\}$, so $E_5=\{e_2\}$. But $B_2$ needs a neighbour of $o_3$; $o_2$ is not one and the even neighbours of $o_3$ are $e_0,e_2,e_4$, so $e_2\in E_2$, a contradiction.
\item $I\in\{\{3,5\},\{3,5,6\}\}$ and $X=\{4\}$, so $5\in I$ and $E_4=\{e_0\}$: since $0,2\in N_A(5)$, both $B_0$ and $B_2$ need a neighbour of $o_5$. Neither $o_0$ nor $o_2$ is one, and the even neighbours of $o_5$ are $e_0,e_2$, of which only $e_2$ is available. It cannot lie in both.\qedhere
\end{itemize}
\end{proof}

\begin{proposition}\label[proposition]{prop:noK8-44348}
$H(44348)$ has no $K_8$ minor.
\end{proposition}
\begin{proof}
Suppose otherwise and take a spanning model as in \cref{lem:normal-form}. The vertices of degree at least seven are $o_0,\ldots,o_6,s,e_0$, so $S$ is a nonempty clique among them. Part (b) of \cref{lem:normal-form} applies to every vertex of $S$ other than $e_0$, and part (c) applies to $e_0$ if $e_0\in S$.

\emph{Step 1: $s\notin S$.} If $s\in S$ then $e_0\notin S$, as $s,e_0$ are not adjacent. The seven other branches contain exactly one neighbour of $s$ each, that is, exactly one odd double vertex, and they contradict \cref{lem:rooted-44348}.

\emph{Step 2: the case $e_0\notin S$.} Then $S\subseteq P$ by Step~1; write $S=\{o_i:i\in I\}$, so $I$ is nonempty and independent in $A$. The branch $B(s)$ has at least two vertices, so it contains a neighbour $o_j$ of $s$. Since $s$ is the only neighbour of $o_i$ in $B(s)$ for each $i\in I$, the vertex $o_j$ is not adjacent to $o_i$, so $ij\in E(A)$. Hence every odd double vertex in $B(s)$ has its index in
\[
 C(I)=\bigcap_{i\in I}N_A(i),
\]
and in particular $C(I)\ne\varnothing$. The nonempty independent sets of $A$ are the seven singletons, the twelve pairs $01,02,12,13,15,24,26,35,36,45,46,56$, and the five triples $012,135,246,356,456$. For $12,15,26,012,135,246$ the set $C(I)$ is empty. Up to $\sigma$ the remaining cases are represented by
\[
 \{0\},\{1\},\{3\},\{5\};\quad\{0,1\},\{1,3\},\{3,5\},\{3,6\},\{5,6\};\quad\{3,5,6\}.
\]
When $q=1$ all seven other branches are pairs, each consisting of one vertex of the index set $N(o_i)$ and one vertex outside $N[o_i]$.

\emph{$I=\{0\}$.} The neighbours of $e_2$ in $N(o_0)$ are $e_3,e_5$; by (W) the pair is $\{e_2,e_3\}$. Likewise $\{e_1,e_4\}$ is a pair. The neighbours in $N(o_0)$ of $o_3,o_4,o_5,o_6$ are respectively $\{o_1,e_4,s\}$, $\{o_2,e_3,s\}$, $\{o_1,s\}$, $\{o_2,s\}$. With $e_3,e_4$ used, these four vertices need four distinct partners in $\{o_1,o_2,s\}$, which is impossible.

\emph{$I=\{1\}$.} Here $N(o_1)=\{o_0,o_2,o_3,o_5,e_4,e_6,s\}$. The neighbours of $e_1$ in $N(o_1)$ are $e_4,e_6$, so by (W) $\{e_1,e_4\}$ is a pair. The vertex $e_6\in N(o_1)$ pairs with a neighbour outside $N[o_1]$, namely $e_0$ or $e_1$; so $\{e_0,e_6\}$ is a pair. The neighbours of $e_5$ in $N(o_1)$ are $o_0,o_2$, and $\{e_5,o_2\}$ has no neighbour in $\{e_1,e_4\}$; so $\{e_5,o_0\}$ is a pair. The neighbours of $e_3$ in $N(o_1)$ are $o_0,o_2,e_4$; so $\{e_3,o_2\}$ is a pair. The neighbours of $e_2$ in $N(o_1)$ are $o_3,o_5$, and $\{e_2,o_5\}$ has no neighbour in $\{e_1,e_4\}$; so $\{e_2,o_3\}$ is a pair. The remaining vertices $o_4,o_5,o_6,s$ form two pairs with $B(s)\in\{\{s,o_4\},\{s,o_6\}\}$ by $C(\{1\})=\{4,6\}$, so the last pair is $\{o_5,o_6\}$ or $\{o_4,o_5\}$. Neither has a neighbour in $\{e_5,o_0\}$.

\emph{$I=\{3\}$ and $I=\{5\}$.} For $i\in\{3,5\}$ the neighbours of $e_5$ in $N(o_i)$ are $e_0,e_2$, so by (W) $\{e_0,e_5\}$ is a pair. The neighbours of $e_6$ in $N(o_i)$ are $o_1,e_0$, so $\{o_1,e_6\}$ is a pair. The vertex $e_2$ lies in $N(o_i)$ and its neighbours outside $N[o_i]$ are $e_3,e_5$, so $\{e_2,e_3\}$ is a pair. The pairs $\{o_1,e_6\}$ and $\{e_2,e_3\}$ are not adjacent.

For $q\ge2$ we use the index sets $N_{o_i}=N(o_i)\setminus S$; by \cref{lem:normal-form}(b) every nonsingleton branch contains exactly one vertex of each.

\emph{$I=\{0,1\}$.} The index sets are $N_{o_0}=\{o_2,e_3,e_4,e_5,e_6,s\}$ and $N_{o_1}=\{o_2,o_3,o_5,e_4,e_6,s\}$. The branches $B(e_3)$ and $B(e_5)$ are distinct and contain no vertex of $\{o_2,e_4,e_6,s\}\subseteq N_{o_0}$, so each contains one of $o_3,o_5$. Neither $e_3$ nor $e_5$ is adjacent to $o_3$ or $o_5$, so both branches have at least three vertices, exceeding the budget $q-1=1$.

\emph{$I=\{1,3\}$.} Here $C(I)=\{4\}$, so $o_4\in B(s)$. The vertex $o_5$ lies in both $N_{o_1}$ and $N_{o_3}$, so $B(o_5)\setminus\{o_5\}$ consists of vertices outside $S\cup N(o_1)\cup N(o_3)$. The only such neighbour of $o_5$ is $o_4$, which lies in $B(s)$, and $B(s)\ne B(o_5)$ because $s,o_5\in N_{o_3}$. So $B(o_5)=\{o_5\}$, contradicting $o_5\notin S$.

\emph{$I=\{3,6\}$.} Again $o_5\in N_{o_3}\cap N_{o_6}$, and every neighbour of $o_5$ outside $S$, namely $o_1,o_4,e_0,e_2,s$, lies in $N(o_3)\cup N(o_6)$. As before $B(o_5)=\{o_5\}$, a contradiction.

\emph{$I=\{3,5\}$.} The index sets are $N_{o_3}=\{o_1,o_6,e_0,e_2,e_4,s\}$ and $N_{o_5}=\{o_1,o_4,o_6,e_0,e_2,s\}$; the free vertices are $o_0,o_2,e_1,e_3,e_5,e_6$. The five vertices in both index sets lie in five distinct branches, so the sixth nonsingleton branch contains $e_4$ and $o_4$. They are not adjacent, so $B(e_4)$ is the unique triple and all other nonsingleton branches are pairs. The free neighbours of $e_2$ are $e_3,e_5$, so $B(e_2)=\{e_2,e_3\}$ by (W). The third vertex of $B(e_4)$ is a free common neighbour of $e_4,o_4$, namely $e_1$ or $e_3$; hence $B(e_4)=\{e_4,o_4,e_1\}$. The free neighbours of $o_6$ are $o_2,e_1$, so $B(o_6)=\{o_6,o_2\}$. The free neighbours of $s$ are $o_0,o_2$, so $B(s)=\{s,o_0\}$. The free neighbours of $o_1$ are $o_0,o_2,e_6$, so $B(o_1)=\{o_1,e_6\}$, which is not adjacent to $\{e_2,e_3\}$.

\emph{$I=\{5,6\}$.} The index sets are $N_{o_5}=\{o_1,o_3,o_4,e_0,e_2,s\}$ and $N_{o_6}=\{o_2,o_3,o_4,e_0,e_1,s\}$, the free vertices are $o_0,e_3,e_4,e_5,e_6$, and $C(I)=\{0\}$, so $o_0\in B(s)$. The two branches other than $B(o_3),B(o_4),B(e_0),B(s)$ each contain one of $o_1,e_2$ and one of $o_2,e_1$. If one contains $e_2$ and $e_1$, it contains a path between them whose interior consists of free vertices other than $o_0$; the shortest is $e_2\,e_3\,e_4\,e_1$, so that branch has at least four vertices. Otherwise the branches contain $\{o_1,e_1\}$ and $\{o_2,e_2\}$, two nonadjacent pairs, and each has at least three vertices. Either way the budget $q-1=1$ is exceeded.

\emph{$I=\{3,5,6\}$.} Here $q=3$. The index sets are $N_{o_3}=\{o_1,e_0,e_2,e_4,s\}$, $N_{o_5}=\{o_1,o_4,e_0,e_2,s\}$, $N_{o_6}=\{o_2,o_4,e_0,e_1,s\}$; the free vertices are $o_0,e_3,e_5,e_6$, and $o_0\in B(s)$ as $C(I)=\{0\}$. The five nonsingleton branches are $B(o_1),B(e_0),B(e_2),B(e_4),B(s)$, one for each vertex of $N_{o_3}$. The branch $B(e_4)$ needs a vertex of $N_{o_5}$, and $o_1,e_0,e_2,s$ lie in other branches, so $o_4\in B(e_4)$. The only free neighbour of $o_4$ is $e_3$, so $e_3\in B(e_4)$ and $|B(e_4)|\ge3$. The vertices $o_2,e_1$ of $N_{o_6}$ lie one each in $B(o_1)$ and $B(e_2)$. If $e_1\in B(e_2)$, a path from $e_2$ to $e_1$ would have to use free vertices other than $o_0,e_3$, that is, among $e_5,e_6$; but $e_2e_5$, $e_1e_6$ are the only such edges and $e_5e_6$ is not an edge. Hence $o_2\in B(e_2)$ and $e_1\in B(o_1)$. Since $o_2e_2$ and $o_1e_1$ are not edges, $B(e_2)$ and $B(o_1)$ have at least three vertices. Together with $B(e_4)$ this exceeds the budget $q-1=2$.

\emph{Step 3: the case $e_0\in S$.} Every other vertex of $S$ is a neighbour of $e_0$ of degree seven, that is, one of $o_3,o_4,o_5,o_6$, and $o_3o_4$ is not an edge. Up to $\sigma$, the set $S\setminus\{e_0\}$ is one of
\[
 \varnothing,\ \{o_3\},\ \{o_5\},\ \{o_3,o_5\},\ \{o_3,o_6\},\ \{o_5,o_6\},\ \{o_3,o_5,o_6\}.
\]
We use \cref{lem:normal-form}(c): every nonsingleton branch contains one or two neighbours of $e_0$, and exactly one contains two. Note that $s$ is not a neighbour of $e_0$ and has no even neighbour, and $o_0,o_1,o_2$ are not neighbours of $e_0$.

\emph{$S=\{e_0\}$.} All seven other branches are pairs. The six vertices $o_0,o_1,o_2,e_1,e_2,s$ outside $N[e_0]$ lie in six distinct pairs, each with one neighbour of $e_0$, and the seventh pair consists of two neighbours of $e_0$. The partner of $e_5$ lies in $N(e_5)\setminus\{e_0\}=\{o_0,o_2,e_2\}$, and by (W) it is $o_0$ or $o_2$. Likewise the partner of $e_6$ is $o_0$ or $o_1$. Since $\sigma$ fixes $e_0,o_0$ and exchanges $e_5,e_6$ and $o_1,o_2$, there are two cases.
\begin{itemize}
\item $\{e_5,o_2\}$ and $\{e_6,o_1\}$ are pairs. The partner of $o_0$ lies in $N(o_0)\cap N(e_0)=\{e_3,e_4,e_5,e_6\}$, and $e_5,e_6$ are used; as $\sigma$ preserves this case and exchanges $e_3,e_4$, we may take $\{o_0,e_3\}$. The partner $y$ of $e_1$ lies in $N(e_1)\cap N(e_0)\setminus\{e_6\}=\{o_4,o_6,e_4\}$. The pair $\{e_1,y\}$ must be adjacent to $\{e_5,o_2\}$ and to $\{o_0,e_3\}$, and $e_1$ has no neighbour in either set; so $y$ has a neighbour in both, which excludes $e_4$ (no neighbour in $\{e_5,o_2\}$) and $o_6$ (no neighbour in $\{o_0,e_3\}$). Hence $\{e_1,o_4\}$ is a pair. Now $e_4$ must be paired with a neighbour, and $o_3$ is the only unused one; but $\{o_3,e_4\}$ has no neighbour in $\{e_5,o_2\}$.
\item $\{e_5,o_0\}$ and $\{e_6,o_1\}$ are pairs. The partner of $e_1$ lies in $\{o_4,o_6,e_4\}$ and, as $e_1$ has no neighbour in $\{e_5,o_0\}$, must have one there; so $\{e_1,e_4\}$ is a pair. The pair formed by two neighbours of $e_0$ uses two of the unused vertices $o_3,o_4,o_5,o_6,e_3$ and must have a neighbour in $\{e_5,o_0\}$, so it contains $e_3$, whose only neighbour among $o_3,o_4,o_5,o_6$ is $o_4$. But $\{e_3,o_4\}$ has no neighbour in $\{e_6,o_1\}$.
\end{itemize}

\emph{$S=\{e_0,o_3\}$.} Here $q=2$. The index set $N_{o_3}=\{o_1,o_5,o_6,e_2,e_4,s\}$ meets the six nonsingleton branches once each; the free vertices are $o_0,o_2,o_4,e_1,e_3,e_5,e_6$. The only free neighbour of $o_5$ is $o_4$, so $o_4\in B(o_5)$, and this branch contains the two neighbours $o_4,o_5$ of $e_0$. The branch $B(s)$ must contain a neighbour of $e_0$; its other vertices are free, the free neighbours of $s$ other than $o_4$ are $o_0,o_2$, and these are not neighbours of $e_0$. Hence $|B(s)|\ge3$, $B(s)$ is the unique triple, and all other nonsingleton branches are pairs. The free neighbours of $e_2$ are $e_3,e_5$, so $B(e_2)=\{e_2,e_3\}$ by (W). The free neighbours of $o_1$ are $o_0,o_2,e_6$, and $B(o_1)$ contains a neighbour of $e_0$, so $B(o_1)=\{o_1,e_6\}$. These two pairs are not adjacent.

\emph{$S=\{e_0,o_5\}$.} Here $N_{o_5}=\{o_1,o_3,o_4,o_6,e_2,s\}$ and the free vertices are $o_0,o_2,e_1,e_3,e_4,e_5,e_6$. The free neighbours of $s$ are $o_0,o_2$, neither a neighbour of $e_0$, so $B(s)$ is the unique triple and the other branches are pairs. As before $B(e_2)=\{e_2,e_3\}$. The only free neighbour of $o_3$ is $e_4$, so $B(o_3)=\{o_3,e_4\}$. The free neighbours of $o_1$ are $o_0,o_2,e_4,e_6$; it needs a neighbour of $e_0$ and $e_4$ is used, so $B(o_1)=\{o_1,e_6\}$, which is not adjacent to $\{e_2,e_3\}$.

\emph{$S=\{e_0,o_3,o_5\}$.} Here $q=3$, $N_{o_3}=\{o_1,o_6,e_2,e_4,s\}$, $N_{o_5}=\{o_1,o_4,o_6,e_2,s\}$, and the free vertices are $o_0,o_2,e_1,e_3,e_5,e_6$. The four vertices in both index sets lie in four distinct branches, so the fifth nonsingleton branch contains $e_4$ and $o_4$, which are not adjacent; thus $|B(e_4)|\ge3$. The free neighbours of $s$ are $o_0,o_2$, neither a neighbour of $e_0$, so $|B(s)|\ge3$. By the budget $q-1=2$ these two branches are triples and the others are pairs. Since $e_2,e_4,s\in N_{o_3}$, the branch $B(e_2)$ is neither triple, so it is a pair; the free neighbours of $e_2$ are $e_3,e_5$, so $B(e_2)=\{e_2,e_3\}$ by (W). Likewise $B(o_1)$ is a pair; it contains a neighbour of $e_0$, and of the free neighbours $o_0,o_2,e_6$ of $o_1$ only $e_6$ is one, so $B(o_1)=\{o_1,e_6\}$. These pairs are not adjacent.

\emph{$S=\{e_0,o_3,o_6\}$.} The vertex $o_5$ lies in $N_{o_3}\cap N_{o_6}$, and its neighbours outside $S$, namely $o_1,o_4,e_2,s$, all lie in $N(o_3)\cup N(o_6)$. Hence $B(o_5)=\{o_5\}$, contradicting $o_5\notin S$.

\emph{$S=\{e_0,o_5,o_6\}$.} Here $q=3$, $N_{o_5}=\{o_1,o_3,o_4,e_2,s\}$, $N_{o_6}=\{o_2,o_3,o_4,e_1,s\}$, and the free vertices are $o_0,e_3,e_4,e_5,e_6$. The only free neighbour of $s$ is $o_0$, which is not a neighbour of $e_0$, so $B(s)\ni o_0$ and $|B(s)|\ge3$. The two branches other than $B(o_3),B(o_4),B(s)$ each contain one of $o_1,e_2$ and one of $o_2,e_1$. If one contains $e_2,e_1$, a path between them through free vertices other than $o_0$ has at least two interior vertices ($e_2\,e_3\,e_4\,e_1$ being shortest), so that branch has at least four vertices. Otherwise both branches contain a nonadjacent pair, $\{o_1,e_1\}$ or $\{o_2,e_2\}$, and have at least three vertices. Either way the budget $q-1=2$ is exceeded.

\emph{$S=\{e_0,o_3,o_5,o_6\}$.} Here $N_{o_3}=\{o_1,e_2,e_4,s\}$, $N_{o_5}=\{o_1,o_4,e_2,s\}$, $N_{o_6}=\{o_2,o_4,e_1,s\}$, and the free vertices are $o_0,e_3,e_5,e_6$. The branch $B(e_4)$ needs a vertex of $N_{o_5}$; $o_1,e_2,s$ lie in other branches, so $o_4\in B(e_4)$. The only free neighbour of $o_4$ is $e_3$, so $e_3\in B(e_4)$. Then $B(e_4)$ contains the three neighbours $e_3,e_4,o_4$ of $e_0$, contrary to \cref{lem:normal-form}(c).

All cases are impossible, so $H(44348)$ has no $K_8$ minor.
\end{proof}

\begin{remark}
The proof covers all sixty-one nonempty cliques of vertices of degree at least seven: the twenty-five containing $s$ in Step~1, the twenty-four contained in $P$ in Step~2, and the twelve containing $e_0$ in Step~3. The symmetry $\sigma$ is used to choose orbit representatives and, within the rooted lemma and the case $S=\{e_0\}$, to make without-loss-of-generality choices; each such use is valid because $\sigma$ fixes $o_0,e_0,s$ and the data of the case in question. The independent finite verification of \cref{thm:root-apex} in \cref{subsec:apex-corroboration} also checks every one of these sixty-one sets separately; it is corroboration, not a premise.
\end{remark}

\section{An even vertex attached to eight double fibres}\label{sec:attachment}
This section determines exactly when an even vertex attached to eight double fibres forces a $K_8$ minor. The proof contracts the attached vertex together with one supported double to the apex of a root-apex graph. All minor models in this section are ordinary: their branches need not meet the odd set.

\subsection{Attachment graphs, cards and apex--even edges}
For a graph $A$ on $\{0,\ldots,7\}$ let $K(A)$ be defined by the interface rule of \cref{sec:root-apex} on eight doubles: its vertices are $o_i,e_i$ ($0\le i\le7$), and for $i<j$ the pair $\{i,j\}$ contributes the single edge $o_io_j$ if $ij\notin E(A)$ and the three edges $o_ie_j,e_io_j,e_ie_j$ if $ij\in E(A)$; there are no other edges. For $S\subseteq\{0,\ldots,7\}$ the \emph{attachment graph} $J(A,S)$ is obtained from $K(A)$ by adding a vertex $x$ adjacent exactly to $o_i$ and $e_i$ for $i\in S$. For an index $i$, let $\rho_i$ be the increasing bijection from $\{0,\ldots,7\}\setminus\{i\}$ onto $\{0,\ldots,6\}$. The \emph{card} $A-i$ is the graph on $\{0,\ldots,6\}$ with edge set $\{\rho_i(j)\rho_i(k):jk\in E(A),\ i\notin\{j,k\}\}$. Since an interface depends only on whether its pair is active, $o_j\mapsto o_{\rho_i(j)}$, $e_j\mapsto e_{\rho_i(j)}$ is an isomorphism from $K(A)-\{o_i,e_i\}$ onto $K(A-i)$ preserving odd vertices.

For a graph $B$ on $\{0,\ldots,6\}$ and $T\subseteq\{0,\ldots,6\}$, let $H_T(B)$ be obtained from the root-apex graph $H(B)$ of \cref{sec:root-apex} by adding the edges $se_j$ for $j\in T$. Thus $H_\varnothing(B)=H(B)$, and $s$ is adjacent in $H_T(B)$ exactly to $o_0,\ldots,o_6$ and to $e_j$ for $j\in T$. A graph on $\{0,\ldots,6\}$ is \emph{exceptional} if it is isomorphic to a graph whose code is one of \eqref{eq:thirtyone}, and \emph{hard} if it is isomorphic to a graph with code $42232$, $44340$ or $44348$. An isomorphism $\varphi$ between graphs $B,B'$ on $\{0,\ldots,6\}$ induces, through $o_i\mapsto o_{\varphi(i)}$, $e_i\mapsto e_{\varphi(i)}$ and $s\mapsto s$, an isomorphism from $H_T(B)$ onto $H_{\varphi(T)}(B')$.

The proofs below use the following inherited input, the positive half of the seven-double classification of the companion archive:
\begin{enumerate}[label=(P7)]
\item\label{input:P7} if a graph $B$ on seven indices is not exceptional, then $K(B)$ has a $K_7$ model whose branches all meet $\{o_0,\ldots,o_6\}$.
\end{enumerate}
Input~\ref{input:P7} is certificate-backed. The archive lists the least mask of each of the $1{,}044$ isomorphism types of seven-index graphs, with one rooted $K_7$ model for each of the $1{,}013$ nonexceptional types and a zero placeholder for each of the $31$ exceptional types. \ref{input:P7} follows from three checks: each listed model is a rooted $K_7$ model; the zero placeholders occur exactly at the thirty-one codes of \eqref{eq:thirtyone}; and the representative orbits together contain all $2^{21}$ seven-index graphs. Indeed, a nonexceptional graph lies in the orbit of a listed representative, which then carries a model, and an index permutation acts on both vertices of every double and therefore maps that model to a rooted model of the graph. That the orbits are disjoint and have the listed sizes was also checked, but is not needed. These checks were carried out independently in revisions~4 and~5 (\cref{sec:reproduce}). The $31$ zero records are not used as negative certificates; the negative half of the seven-double classification is not used in this section.

\subsection{Contraction to a root-apex graph}
\begin{lemma}[Root-apex contraction]\label[lemma]{lem:contraction}
Let $A$ be a graph on $\{0,\ldots,7\}$, let $S\subseteq\{0,\ldots,7\}$ and let $a\in S$. Put
\[
 T=\rho_a\bigl(N_A(a)\cup(S\setminus\{a\})\bigr).
\]
Contracting the set $C=\{x,o_a,e_a\}$ of $J(A,S)$ to a single vertex produces a graph isomorphic to $H_T(A-a)$, in which the new vertex corresponds to $s$ and every other vertex $o_j$ or $e_j$ corresponds to $o_{\rho_a(j)}$ or $e_{\rho_a(j)}$. In particular $H_T(A-a)$ is a minor of $J(A,S)$.
\end{lemma}
\begin{proof}
Since $a\in S$, the vertex $x$ is adjacent to $o_a$ and to $e_a$, so $C$ is connected and may be contracted. Let $c$ be the new vertex. Among the remaining vertices $o_j,e_j$ ($j\neq a$) the edges are those of $K(A)-\{o_a,e_a\}$, which $\rho_a$ maps onto $K(A-a)$. The neighbours of $c$ are the vertices outside $C$ having a neighbour in $C$. For $j\neq a$ the vertex $o_j$ is adjacent to $o_a$ if $aj\notin E(A)$ and to $e_a$ if $aj\in E(A)$, so $c$ is adjacent to every $o_j$. The vertex $e_j$ is adjacent to $o_a$ and to $e_a$ if $aj\in E(A)$ and to neither otherwise, and it is adjacent to $x$ exactly when $j\in S$. Hence $c$ is adjacent to $e_j$ exactly when $j\in N_A(a)\cup(S\setminus\{a\})$. These are precisely the neighbours of $s$ in $H_T(A-a)$ after applying $\rho_a$.
\end{proof}

\subsection{One apex--even edge suffices}
\begin{lemma}[The three hard cores]\label[lemma]{lem:hard-one-edge}
Let $h\in\{42232,44340,44348\}$ and $k\in\{0,\ldots,6\}$. Then $H_{\{k\}}(h)$ has a $K_8$ minor.
\end{lemma}
\begin{proof}
This is \cref{ua:hard} in Appendix~\ref{ua:appendix}, proved by six path and triangle formulas with every internal connection and pairwise crossing justified. It uses no computational premise. The twenty-one explicit models of \cref{tab:hard-models} are independent corroboration, not a premise.
\end{proof}
By \cref{thm:root-apex}, no row of \cref{tab:hard-models} is a $K_8$ model of $H(h)$ itself, so each uses the added edge $se_k$; this is consistent with the separately checked fact that each listed row fails in $H(h)$. For $44340$ and $44348$, and for $42232$ with $k\in\{0,1,6\}$, the listed model has $\{s\}$ as a branch; equivalently its other seven branches form a $K_7$ model of $K(h)$ in which every branch meets $\{o_0,\ldots,o_6,e_k\}$. For $42232$ with $k\in\{2,3,4,5\}$ the listed model places $s$ in a branch with $o_6$. Appendix~\ref{rc:sec} proves that precisely these four marks admit no model with singleton branch $\{s\}$; that additional negative result is structural since revision~7 and is not used in the present attachment proof.

\begin{theorem}[One apex--even edge]\label{thm:one-edge}
Assume \ref{input:P7}. Let $B$ be a graph on $\{0,\ldots,6\}$ and $T\subseteq\{0,\ldots,6\}$. Then $H_T(B)$ has a $K_8$ minor if and only if $T\neq\varnothing$ or $B$ is not hard.
\end{theorem}
\begin{proof}
Suppose first that $B$ is not exceptional. By \ref{input:P7}, $K(B)$ has a $K_7$ model each of whose branches contains some $o_i$. The vertex $s$ is adjacent to every $o_i$, so adding the branch $\{s\}$ gives a $K_8$ model of $H(B)$, which is a spanning subgraph of $H_T(B)$.

Suppose next that $B$ is exceptional but not hard. Then $B$ is isomorphic to a graph whose code $m$ is one of the twenty-eight codes of \eqref{eq:thirtyone} other than $42232,44340,44348$. The induced isomorphism maps $H(B)$ onto $H(m)$, and \cref{tab:positive-models} gives a $K_8$ model of $H(m)$. Again $H(B)$ is a spanning subgraph of $H_T(B)$.

Suppose finally that $B$ is hard, and let $\varphi$ be an isomorphism from $B$ onto the graph with code $h\in\{42232,44340,44348\}$. If $T=\varnothing$, then $H_T(B)=H(B)\cong H(h)$ has no $K_8$ minor by \cref{thm:root-apex}. If $k\in T$, then $\varphi$ induces an isomorphism from $H_{\{k\}}(B)$ onto $H_{\{\varphi(k)\}}(h)$, which has a $K_8$ minor by \cref{lem:hard-one-edge}. Since $H_{\{k\}}(B)$ is a spanning subgraph of $H_T(B)$, the latter has a $K_8$ minor.
\end{proof}

\subsection{The attachment theorem}
\begin{theorem}[Attachment theorem]\label{thm:attachment}
Assume \ref{input:P7}. Let $A$ be a graph on $\{0,\ldots,7\}$ and let $\varnothing\neq S\subseteq\{0,\ldots,7\}$. Then $J(A,S)$ has a $K_8$ minor if and only if it is not the case that $S=\{a\}$ for an index $a$ which is isolated in $A$ and whose card $A-a$ is hard.
\end{theorem}
\begin{proof}
Choose $a\in S$ and let $T$ be as in \cref{lem:contraction}, so that $H_T(A-a)$ is a minor of $J(A,S)$. If $T\neq\varnothing$, or if $A-a$ is not hard, then $H_T(A-a)$ has a $K_8$ minor by \cref{thm:one-edge}, and so does $J(A,S)$. Otherwise $T=\varnothing$ and $A-a$ is hard. Now $T=\varnothing$ means exactly that $S=\{a\}$ and $N_A(a)=\varnothing$, so this is the stated exception. Hence $J(A,S)$ has a $K_8$ minor whenever the exception does not hold.

Conversely, suppose that $S=\{a\}$, that $a$ is isolated in $A$, and that $A-a$ is hard. Every pair containing $a$ is inactive. Therefore $o_a$ is adjacent in $K(A)$ exactly to the vertices $o_j$ with $j\ne a$, and $e_a$ has no neighbour in $K(A)$. The subgraph of $J(A,\{a\})$ induced by $\{o_j,e_j:j\neq a\}\cup\{o_a\}$ is therefore mapped by $\rho_a$, together with $o_a\mapsto s$, isomorphically onto $H(A-a)$. The only other vertices are $x$ and $e_a$, and the only other edges are $o_ax$ and $xe_a$. Thus $J(A,\{a\})$ is obtained from a copy of $H(A-a)$ by attaching a path of length two at $s$.

Let $G$ be a graph, let $\ell$ be a vertex of degree one, and let $t\ge3$. If $G$ has a $K_t$ model, then so does $G-\ell$. Indeed, if $\ell$ lies in no branch, the model is already a model of $G-\ell$. The vertex $\ell$ is not a singleton branch, because a singleton branch is adjacent to $t-1\ge2$ other branches. If $\ell$ lies in a branch $X$ with $|X|\ge2$, its unique neighbour lies in $X$ because $X$ is connected. Hence $X-\ell$ is connected and $\ell$ is an endpoint of no edge between distinct branches, so deleting $\ell$ from $X$ leaves a $K_t$ model. Deleting the leaf $e_a$ and then the leaf $x$ shows that a $K_8$ minor of $J(A,\{a\})$ would give one of $H(A-a)$. This contradicts \cref{thm:root-apex}, since $H(A-a)$ is isomorphic to $H(h)$ for some $h\in\{42232,44340,44348\}$.
\end{proof}

\begin{corollary}[Two supports suffice]\label[corollary]{thm:support-two}
Assume \ref{input:P7}. For every graph $A$ on $\{0,\ldots,7\}$ and every $S\subseteq\{0,\ldots,7\}$ with $|S|\ge2$, the graph $J(A,S)$ has a $K_8$ minor. The bound is sharp: if $A$ has edge set $\{03,04,05,06,14,16,23,25,34\}$, then $J(A,\{7\})$ has no $K_8$ minor.
\end{corollary}
\begin{proof}
The exception in \cref{thm:attachment} requires $|S|=1$. In the example, index~$7$ is isolated and $A-7=A$ restricted to $\{0,\ldots,6\}$ has code $44348$. In the lexicographic encoding on eight indices, in which the pair $i<j$ has bit $\tfrac12 i(15-i)+j-i-1$, the graph $A$ has code $305724$, and $S=\{7\}$ has indicator mask $2^7=128$.
\end{proof}
The case $|S|\ge3$ of \cref{thm:support-two} is the attachment lemma of the earlier versions, there called the support-three lemma.
\begin{corollary}[Support three]\label[corollary]{lem:quadruple}
Assume \ref{input:P7}. For every graph $A$ on $\{0,\ldots,7\}$ and every $S$ with $|S|\ge3$, the graph $J(A,S)$ has a $K_8$ minor.
\end{corollary}
\begin{proof}
This is a special case of \cref{thm:support-two}.
\end{proof}

\begin{remark}[Counting the exceptions]
The automorphism groups of the three hard graphs are determined as follows.
\begin{itemize}
\item In $42232=\Theta(2,3,3)$ the vertices $0,1$ are the only vertices of degree three. The graph is the union of three internally disjoint $0$--$1$ paths, of lengths $2,3,3$. An automorphism permutes $\{0,1\}$ and permutes the three paths, preserving lengths. It is determined by these two actions, since a path whose ends are fixed or exchanged is mapped pointwise in the forced way. All four combinations occur, so there are four automorphisms.
\item In $44340=\Theta(1,3,4)$ the vertices $0,3$ are the only vertices of degree three. The three paths between them have distinct lengths, so each is fixed setwise, and an automorphism is determined by its action on $\{0,3\}$. The exchange $(0\,3)(2\,5)(4\,6)$ is an automorphism, so there are two.
\item In $44348$ the vertex $0$ is the only vertex of degree four, and $3,4$ are the only vertices of degree three. An automorphism therefore fixes $0$ and fixes or exchanges $3$ and $4$. The neighbourhoods $N(3)=\{0,2,4\}$, $N(4)=\{0,1,3\}$, $N(5)=\{0,2\}$ and $N(6)=\{0,1\}$ then force the identity, or the permutation $(1\,2)(3\,4)(5\,6)$, which maps $\{03,04,05,06,14,16,23,25,34\}$ onto itself. Again there are two automorphisms.
\end{itemize}
The three hard graphs are pairwise nonisomorphic (see the discussion after the proof of \cref{thm:root-apex}). Hence there are $5040/4+5040/2+5040/2=6300$ labelled hard graphs on $\{0,\ldots,6\}$. For each index $a$ and each such graph there is exactly one $A$ in which $a$ is isolated and $A-a$ equals that graph. The structural converse in \cref{thm:attachment} shows, without \ref{input:P7}, that each of these $8\cdot6300=50{,}400$ pairs $(A,a)$ gives a graph $J(A,\{a\})$ without a $K_8$ minor. Assuming \ref{input:P7}, \cref{thm:attachment} shows that they are exactly the pairs among the $8\cdot2^{28}$ with this property.
\end{remark}

\subsection{Consequence for the order-twenty profile with a quadruple fibre}
\begin{corollary}[Conditional $(2^8,4)$ step]\label[corollary]{cor:q4}
Assume Input~\ref{input:lower18} and \ref{input:P7}. Let $W$ be an oddomorphism witness to $K_9$ with eight double fibres and one fibre $Q$ of size four. Then $W$, and hence every graph containing $W$, has a $K_8$ minor.
\end{corollary}
\begin{proof}
Suppose that $W$ has no $K_8$ minor. The charge of $Q$ is odd, so $Q$ contains one or three odd vertices and hence an even vertex $x$. Write the doubles as $D_i=\{o_i,e_i\}$ for $0\le i\le7$. Fibres are independent sets, so every neighbour of $x$ lies in some $D_i$. The degree of $x$ into each $D_i$ is even and at most two, so it is zero or two. Let $S$ be the set of indices $i$ for which $x$ has a neighbour in $D_i$.

If $S=\varnothing$, then $x$ is isolated. Deleting it leaves a witness to $K_9$ on nineteen vertices with profile $(2^8,3)$ and the same restricted indicator. By \cref{prop:finalprofile}, whose only hypothesis is Input~\ref{input:lower18} (its proof also uses \cref{thm:main} and the finite \cref{lem:B7}), this subgraph of $W$ has a $K_8$ minor. If $S=\{i\}$, the even-star contraction of \cref{lem:normalization} at $x$ produces a witness to $K_9$ on eighteen vertices that is a minor of $W$. By Input~\ref{input:lower18} it has a $K_8$ minor. Both cases contradict the assumption. Hence $|S|\ge2$. By \cref{lem:double} the subgraph of $W$ induced by $x$ and the eight doubles is exactly $J(A,S)$, where $A$ is the active graph of the doubles. Thus \cref{thm:support-two} gives a $K_8$ minor, again a contradiction.
\end{proof}
This is the support argument of \cref{lem:support19} specialised to the vertex $x$. It replaces the earlier route through a three-odd-vertex reduction of $Q$, a separate proof of $|S|\ge3$ and the support-three certificates. It still depends on Input~\ref{input:lower18} and on \cref{prop:finalprofile}, which belong to the accepted chain of \cref{thm:order19}; that proposition in turn uses \cref{thm:main} and the finite \cref{lem:B7}. It concerns only the profile $(2^8,4)$, which the historical archive had already excluded. It gives a $K_8$ minor, not a $K_9$ minor, and it changes none of the ten order-twenty profiles that remain open (\cref{sec:order19}).

\subsection{Earlier certificate proofs}
Earlier versions proved the attachment statements by finite positive certificates on top of \ref{input:P7}. Version~3 used $9{,}201$ models for $|S|\ge3$, of which the $4{,}010$ with $|S|=3$ suffice by support monotonicity. Revision~4 used $4{,}740$ models for $|S|=2$, one for every element of the residual domain in which both supported cards are exceptional. The residual domain was defined there after relabelling one supported index to $7$ and the other to $t\in\{0,\ldots,6\}$, with $A-7$ equal to one of the thirty-one listed codes. All these models were re-verified for this revision from their files and independently reconstructed domains. They remain in the archive as corroboration. They are not premises of the proofs above.

The $9{,}201$ version-3 models were assembled from companion data. For $9{,}189$ keys, a $K_7$ model of a card $A-i$ with $i\in S$ was used. In that model exactly one branch misses the odd double vertices, and that branch contains $e_a$ for some $a\in S$. It was extended by $\{x,o_i,e_i\}$. For the other twelve keys, a $K_7$ model of the card $A-7$ was used. In that model exactly one branch misses the odd double vertices, and that branch contains $e_a$ for an index $a\in\{0,1,6\}$ adjacent to $7$. It was extended by $\{x,o_7,e_7\}$. This set is connected because $7\in S$ in every record, and it meets the root-free branch through the active pair $a7$. In these twelve records $a\notin S$. This provenance is not used in any proof.

None of the $4{,}740$ revision-4 models has a singleton branch. This is a property of those stored witnesses only. It does not show that the corresponding graphs have no $K_8$ model with a singleton branch. What it illustrates is that eight pairs fit into seventeen vertices, so a $K_8$ model of $J(A,S)$ need not contain a singleton branch. For this reason the fifteen-vertex singleton-pivot reduction of \cref{lem:normal-form} is not a valid negative decision method for $J(A,S)$. No negative search result is used in this section. The negative half of \cref{thm:attachment} rests on the structural exclusions in \cref{thm:root-apex}.

\section{The bounded nine colour result}\label{sec:order19}
This section records a separate computer-assisted theorem and proves its final case. The earlier eighteen- and nineteen-vertex cases are supplied in five historical-dependency topic supplements, with source locations in Appendix~\ref{app:sources} and the accompanying dependency index. They are not re-proved by the excess-six argument. This distinction matters because a source with nineteen vertices and nine fibres has excess ten.

\begin{theorem}[Bounded nine-colour theorem]\label{thm:order19}
If $F$ has at most nineteen vertices and admits a weak oddomorphism to $K_9$, then $F$ has a $K_8$ minor.
\end{theorem}
The proof has three levels: the symbolic reductions below, explicit finite rooted-core statements, and their certificate validations. This section proves the final-profile reduction in detail; the historical archive supplies the other profile proofs and their certificates. The theorem is computer-assisted and its full proof requires those historical materials.

\subsection{How the earlier cases enter}
The minimum-order consequence of \cref{thm:main} says that a $K_8$-minor-free graph cannot contain a witness to $K_8$ on fewer than fifteen vertices. Therefore deleting any fibre from an order-$v$ counterexample to \cref{thm:order19} leaves at least fifteen vertices. For $v=19$, all fibres have size at most four.

If $s,d,t,q$ count singleton, double, triple, and quadruple fibres, respectively, then
\[
 s+d+t+q=9,\qquad s+2d+3t+4q=19.
\]
Equivalently $d=10-2t-3q$ and $s=t+2q-1$; nonnegativity gives exactly thirteen profiles. Their counts are
\[
\begin{array}{c|ccccc}
(s,d,t,q)&(0,8,1,0)&(1,6,2,0)&(2,4,3,0)&(3,2,4,0)&(4,0,5,0)\\[-2pt]
\end{array}
\]
for $q=0$, and
\[
(1,7,0,1),\ (2,5,1,1),\ (3,3,2,1),\ (4,1,3,1),\
(3,4,0,2),\ (4,2,1,2),\ (5,0,2,2),\ (5,1,0,3)
\]
for $q>0$.

The source-mapped preliminary proofs establish the following two inputs:
\begin{enumerate}[label=(C\arabic*)]
\item\label{input:lower18} Every witness to $K_9$ on at most eighteen vertices has a $K_8$ minor.
\item\label{input:profiles} Every order-nineteen profile other than $(2^8,3)$ has a $K_8$ minor.
\end{enumerate}
Input \ref{input:profiles} is not a single search claim. Structural reductions first leave five profiles,
\[
(2^8,3),\quad (1,2^6,3^2),\quad (1^2,2^4,3^3),\quad
(1^3,2^2,3^4),\quad(1^4,2,3^3,4),
\]
then distinct rooted-core arguments exclude the last four. The dependency map identifies each argument. Input \ref{input:lower18} uses the excess-six theorem followed by an order-seventeen argument and five order-eighteen profile exclusions, including a corrected cycle-background proof. A superseded disconnected branch construction is not used.

\begin{lemma}[Ambient support condition]\label[lemma]{lem:support19}
Assuming Input~\ref{input:lower18}, every even vertex in a nineteen-vertex $K_8$-minor-free witness to $K_9$ supports at least two other fibres.
\end{lemma}
\begin{proof}
Deleting an isolated even vertex or performing one even-star contraction from \cref{lem:normalization} gives a smaller witness to the same target in a minor. It has at most eighteen vertices, contradicting Input~\ref{input:lower18}. The charges keep every fibre nonempty.
\end{proof}
The support condition applies to the full ambient witness. It need not hold in a restriction used as a finite lemma.

\subsection{The restricted rooted seven clique lemma}
\begin{lemma}[Finite rooted core]\label[lemma]{lem:B7}
Let $W$ consist of six double fibres $D_i=\{o_i,e_i\}$, $0\le i<6$, and a triple $T=\{t_0,t_1,t_2\}$ whose vertices are all odd. Every double pair has one of the two interfaces in \cref{lem:double}. For each $i$, the root $o_i$ is adjacent to one chosen $t_{c_i}$ and $e_i$ is adjacent to the other two vertices of $T$. There are no further edges. Then $W$ has a $K_7$ model rooted in $\{o_0,\ldots,o_5\}\cup T$.
\end{lemma}
\begin{proof}[Computer-assisted proof]
The fifteen active-pair bits give $b\in\{0,\ldots,2^{15}-1\}$ and the six choices give $c_i\in\{0,1,2\}$. Index the graph by
\[
 I=b+2^{15}\sum_{i=0}^{5}c_i3^i.
\]
This is a bijection onto $\{0,\ldots,23,887,871\}$. Number vertices by $o_i=2i$, $e_i=2i+1$, and $t_a=12+a$; the fixed odd mask is hexadecimal $7555$.

To compress the certificate, first relabel the active graph by a permutation of its six vertices to a minimum representative. Exhaustive application of all $720$ permutations produces $156$ disjoint orbits covering all $32,768$ masks. Carry the six triple choices along this permutation. Next relabel the three triple vertices by first occurrence in the resulting word. These canonical words are the partitions of six positions into at most three blocks, so their number is
\[
 S(6,1)+S(6,2)+S(6,3)=1+31+90=122.
\]
All $156\cdot122=19,032$ product settings are retained; no assumption about compatibility with residual active-graph automorphisms is needed.

The certificate supplies seven unsigned sixteen-bit branch masks for each setting, fourteen bytes per record. The independent checker reconstructs the graph and checks each branch is in range, nonempty, disjoint from the others, connected, and meets the fixed odd mask, as well as all twenty-one crossings. The supplied file has exactly $266,448$ bytes and every record passes. Both relabellings preserve the complete adjacency relation and the odd set, so inverse images of the branches prove every labelled setting.

As an additional full-domain implementation check, the transport validator reconstructs each of the $23,887,872$ original graphs from its own index and directly checks the transported branches there. All pass. The representative list, certificate, independent checker, exact commands, and execution receipts are in the full historical archive; they are inherited and not replayed by the current package. This is a positive-certificate proof; it makes no inference from a failed search.
\end{proof}

\subsection{Two whole fibre branches and the final profile}
\begin{proposition}\label[proposition]{prop:finalprofile}
Assuming Input~\ref{input:lower18}, an order-nineteen witness to $K_9$ of profile $(2^8,3)$ has a $K_8$ minor.
\end{proposition}
\begin{proof}
Suppose otherwise. Write the doubles as $D_i=\{o_i,e_i\}$, $0\le i<8$, and the triple as $T$. By \cref{lem:support19}, every even vertex supports at least two fibres. Restrictions to any six doubles have twelve vertices and six colours, so \cref{thm:main} gives a model rooted in their \emph{unchanged original} odd set.

First suppose $T$ has one odd vertex and two even vertices $x,y$. Each of $x,y$ supports at least two doubles and sees both vertices of any supported double. Choose distinct supported doubles $D_i$ for $x$ and $D_j$ for $y$; this is possible even when the support sets coincide. The disjoint stars
\[
 A=D_i\cup\{x\},\qquad B=D_j\cup\{y\}
\]
are rooted and contain whole fibres. A rooted $K_6$ model on the remaining six doubles is disjoint from both stars. Whole-fibre adjacency joins each star to all six branches, and joins the stars to each other using $o_j$ into $D_i$. This gives $K_8$. Thus all three vertices of $T$ are odd.

Each triple vertex now has exactly one neighbour in any double. The root $o_i$ therefore sees either one or three vertices of $T$, with $e_i$ seeing the complement. If $o_i$ sees all of $T$, then $e_i$ sees none of $T$. It supports some other double $D_j$, whose interface with $D_i$ is active. The sets
\[
 A=T\cup\{o_i\},\qquad B=D_j\cup\{e_i\}
\]
are disjoint connected rooted stars, adjacent through $o_ie_j$. Again the other six doubles supply a rooted $K_6$; the two stars attach to all six by whole-fibre adjacency. This is a $K_8$ model. Consequently every double--triple interface has exactly the restricted transversal form of \cref{lem:B7}.

If no double pair is active, the eight $o_i$ form a $K_8$ subgraph. Otherwise choose an active pair $i,j$. Its four vertices form the path $o_i-e_j-e_i-o_j$; use their union as one branch $A$. The remaining six doubles together with $T$ satisfy \cref{lem:B7}, with their original odd set, and provide seven rooted branches. The whole double $D_i\subseteq A$ is adjacent to each. These eight branches form a $K_8$ model, a contradiction.
\end{proof}
Combining Inputs~\ref{input:lower18} and~\ref{input:profiles} with \cref{prop:finalprofile} proves \cref{thm:order19}. This last proof does not use a universal seven-double rooted $K_7$ assertion: that assertion is false for the known theta example.

\paragraph{Order twenty.} The companion archive also contains computer-assisted exclusions of the order-twenty profiles $(1,2^7,5)$ and $(2^8,4)$, conditional on the inherited chain above and on further reductions proved there. For the second profile, \cref{cor:q4} gives a short proof. It assumes Input~\ref{input:lower18} and \ref{input:P7}, and uses \cref{prop:finalprofile} and the attachment theorem of \cref{sec:attachment}. In this conditional form the exclusion of $(2^8,4)$ is a result of this paper. The exclusion of $(1,2^7,5)$ and the archive's other order-twenty reductions are not reproduced here and are not theorems of this paper. Ten order-twenty profiles remain open in that archive. In particular we do not claim that every weak oddomorphism witness to $K_9$ of order twenty forces a $K_8$ minor, and the order bound of \cref{thm:order19} is not improved.

\clearpage
\section{Reproducibility and limitations}\label{sec:reproduce}
\paragraph{Evidence by statement.}
\begin{itemize}
\item \Cref{thm:main,thm:sharp} (core paper): structural. Computations described below corroborate them and are not premises.
\item \Cref{thm:rooted-four} (core paper): structural, with a complete proof in \cref{sec:rooted-four}. No computation is a premise.
\item \Cref{thm:root-apex}: twenty-eight explicit models (\cref{tab:positive-models}) and three structural exclusions (\cref{prop:noK8,prop:noK8-44340,prop:noK8-44348}). Its reading as a statement about witnesses of profile $(1,2^7)$ uses the inherited seven-double classification.
\item \Cref{lem:contraction}, the negative half of \cref{thm:attachment}, \cref{ua:hard} and \cref{rc:classification,rc:apex}: structural. The negative half of \cref{rc:classification} became structural in revision~7.
\item \Cref{thm:one-edge,thm:attachment,thm:support-two,lem:quadruple}: certificate-backed through Input~\ref{input:P7} ($1{,}013$ rooted models and a complete $1{,}044$-orbit cover of all $2^{21}$ seven-index masks) and the twenty-eight models of \cref{tab:positive-models}.
\item \Cref{cor:q4}: additionally conditional on Input~\ref{input:lower18}, \cref{prop:finalprofile} and, through it, the finite \cref{lem:B7}. The B7 certificate is inherited from the full historical archive and was checked in the earlier order-nineteen review leg; it is not replayed by this package.
\item \Cref{thm:order19}: computer-assisted, through the inherited Inputs~\ref{input:lower18} and~\ref{input:profiles} and \cref{lem:B7}.
\end{itemize}
The $21$ hard-core models (\cref{tab:hard-models}), the $4{,}740$ revision-4 and $9{,}201$ version-3 attachment certificates, the revision-6 rejection stream and the $27$-case refutation forest of Appendix~\ref{rc:sec}, and the exhaustive fifteen-vertex check of Appendix~\ref{app:h15-computation} are valid corroboration, not premises.

\paragraph{Code and certificate archive.}
The supporting release~\cite{CertikSupport2026} has Zenodo DOI \url{https://doi.org/10.5281/zenodo.23603860}. Its six archives separate the active \mbox{checks} (archive 00) from five historical dependency packages. After extracting archive 00, follow its \texttt{START\_HERE.md}. The commands are:
\begin{quote}\small\ttfamily\raggedright
python3 reproduce.py --verify-only\par
python3 reproduce.py --quick --workdir ../replay\_00\_quick\par
python3 reproduce.py --full --workdir ../replay\_00\_full
\end{quote}
The first command verifies the files; the other two run the documented checks. Each historical archive has its own entry instructions. These commands differ from the full historical source-package command below; the public release does not use its \texttt{--with-pdf} option. The historical computations are not all re-executed by the active package.

\paragraph{Current replay in the full historical source package.}
From the package root, the command
\begin{center}\ttfamily python3 reproduce.py --workdir NEWDIR [--with-pdf]\end{center}
verifies the package manifest, copies the sources into an empty directory and runs, with recorded commands, UTC times, exit codes and hashes: the inherited revision-6 checks (P7 orbit cover and models, the twenty-eight and twenty-one printed models, local structural assertions, regeneration of the augmentation formulas, regeneration and validation of the rejection stream, and corruption controls), without any historical minor census; the revision-7 checks: every finite assertion in the structural proof of Appendix~\ref{rc:sec} with mutation controls, a corroborating enumeration of all marked-root spanning partitions, a third independent validation of the rejection stream, independent checkers for Appendices~\ref{ua:appendix} and~\ref{rc:sec}, the native $27$-case forest with its controls, and a comparison of the printed tables with the certificate files; and, with \texttt{--with-pdf}, the builds of the core paper, this companion and the current report, reference checks, and checks that the statement numbers and statements of earlier editions are unchanged. The option \texttt{--quick} runs only the two checkers that read these TeX sources, together with the builds and checks. Standard-library Python~3 and a C++17 compiler suffice; the PDFs need XeLaTeX and BibTeX. No network, solver or graph database is used.

\paragraph{Historical computations.}
Earlier legs are preserved byte for byte in the full provenance package, each with its own driver and receipts, and are not all re-executed here. They include the version-3 root-apex validators and refutation forests, the exhaustive enumeration of $156{,}545$ canonically labelled excess-six interface states ($47{,}259$ normalized), which found an odd-rooted clique model in every state, and the replay of the restricted B7 certificate on all $23{,}887{,}872$ graphs; the revision-4 regeneration of the $4{,}740$ support-two models; and the revision-5 three-checker validation of the hard-core models, the P7 orbit checker and constructive transport tests. Two optional computations remain author-recorded corroboration only: a full search of the $2^{21}$ seven-double masks, whose saved negative lists match the thirty-one exceptional orbits but whose execution ranges are not bound, and a $98{,}768$-instance singleton sweep without persisted per-instance models.

\paragraph{Citations.}
The doubled theta construction is in Kriesell and Mohr, arXiv:1911.09998v2, Section~3, pp.~8--9, immediately before Theorem~2 on p.~9, and their routing theorem for at most six doubled vertices is Theorem~6 in Section~5, p.~12; the earlier version~1 places these in Sections~2 and~4. These locators were established in revision~6 from the version-specific PDF bytes and are recorded with their hashes in \path{CITATION_ERRATUM.md}; a later independent audit re-read the same version-2 PDF and confirmed them. No journal edition of that preprint has been verified. Neuen and Seppelt's lemma numbers refer to arXiv:2601.18602v2; their paper also appeared in LICS~2026.

\paragraph{AI assistance.}
AI tools assisted with mathematical exploration, proof development, computational checks and manuscript preparation. The author takes responsibility for the final content.

\paragraph{Limitations of the evidence.}
Role separation between producing and reviewing AI tools does not establish statistical independence, independent human review, or formally verified correctness. Certificate-backed statements remain subject to the correctness of the supplied programs, their compilation and execution. No proof assistant has verified the symbolic proofs.

\paragraph{Remaining mathematical work.}
The minimum orders for the remaining small targets $n=6,7$ are not determined here; for $n=5$ the closure theorem of Neuen and Seppelt \cite[Theorem~5.22 in arXiv version~2]{NeuenSeppelt2026} rules out every counterexample. \Cref{thm:root-apex} classifies a specified family; whether every fifteen-vertex witness to $K_8$ without a $K_8$ minor has profile $(1,2^7)$ is open here. A structural proof of Input~\ref{input:P7} would remove the largest finite premise of \cref{sec:attachment}; structural replacements for the companion certificate lemmas behind Inputs~\ref{input:lower18} and~\ref{input:profiles} and \cref{lem:B7} would do the same for \cref{sec:order19}. Ten order-twenty profiles remain open in the archive, and no unrestricted $K_9$-to-$K_8$ statement or separation of graph-learning architectures is claimed. The literature comparison is a primary-source search, not a proof that no prior result exists. Related symmetric-difference merging appears in \cite[Definition 2.3 and Lemma 2.4]{JimenezMooreQuirozYoo2026}, and Arsenyan's thesis \cite[Sections 4.3 and 4.5]{Arsenyan2025} reports related small rooted-routing searches, made with a heuristic minor finder; these are context, not premises. The comparison of Arsenyan's Figure~2.3 with the thirty-one active graphs of \cref{sec:root-apex} remains unresolved; we make no claim of originality for individual cores.

\appendix
\section{Proof dependency and source map}\label{app:sources}
The revision-7 package provides a claim ledger, a coverage and dependency ledger, and input hashes; all six inputs of revision~6, including the complete revision-5 chain, are preserved byte for byte inside the revision-6 full provenance package, which is itself preserved in the revision-7 full package. The revision-4 package, with its \path{provenance/SOURCE_MAP.json}, is preserved byte for byte inside the revision-5 full provenance package. Hashes establish byte identity, not mathematical truth. Both incoming version-3 archives and their standalone PDFs are preserved byte for byte in the revision-4 full provenance package's \path{incoming/} directory. Legacy paths mentioned below are archive-relative, not paths at the revision-5 root. In particular the old \path{SOURCE_MAP.json} and structural QA directories occur inside the expanded archive's preserved incoming version-2 archive. The recovered P7 data come from the expanded archive's \path{incoming/order20_q4_exclusion_package.zip}, then its \path{incoming/final_order19_proof_package.zip}. No unverified filename alias is used. The new compact replay uses the extracted, hash-bound P7 files directly; it does not replay the older census.

\subsection{Attachment theorem}
The proof of \cref{thm:attachment} is given in full in \cref{sec:attachment}. Its finite inputs are the inherited rooted models and orbit cover behind \ref{input:P7}, recovered from the expanded archive's nested order-nineteen package (files \path{reps7.txt} and \path{k7_doubles_representatives.bin}) and the twenty-eight models of \cref{tab:positive-models}. Its hard-core step is supplied by the structural augmentation formulas of Appendix~\ref{ua:appendix}. The twenty-one models of \cref{tab:hard-models}, introduced in revision~5, are retained corroboration since revision~6. The enlarged-root classification of Appendix~\ref{rc:sec} is structural since revision~7; its revision-6 rejection stream and the independent $27$-case forest are corroboration. The revision-5 certificate files carry SHA-256 bindings in the package manifest.

\subsection{Main extremal theorem}
The main argument is given in the core paper, with its imported literature results explicitly stated and cited:
\begin{align*}
\text{parity equations}&\Longrightarrow\text{root-preserving normalization}\\
 &\Longrightarrow\text{excess five and seven excess-six patterns}\\
 &\Longrightarrow\text{Theorem~\ref{thm:main}}.
\end{align*}
For the upper bound, the explicit graph construction, the rooted theta obstruction \cref{lem:known-theta} (proved in the core paper by Kriesell and Mohr's counting argument for their arXiv~v2 construction), and the eight singleton cases prove \cref{prop:noK8}. Cores consisting only of doubles are covered by the cited routing theorem of Kriesell and Mohr; the direct proofs are in Appendix~\ref{app:all-doubles}. Singleton extension then gives $m(n)=n+7$ for every $n\ge8$. No order-eighteen or order-nineteen claim, or computational assertion, is used for the main result. The source map distinguishes the earlier mathematical material, the consolidated proof, and the corroborating verifiers. The accepted H15 structural supplement has SHA-256 \texttt{\seqsplit{ddbbaf9bee3c65569eee348eb4ddc2706899326878d3b30167adc471e40b6c40}}; it is retained with its audits under the archive-qualified legacy \path{qa/structural_h15} directory.

\subsection{Version-3 sources}
The new material is bound to the following inputs by SHA-256 (full digests in the package manifest).
\begin{sloppypar}
\begin{itemize}
\item The recovered final source of the preceding review leg, \path{gpt_final_reconstruction_source.md} ($65{,}259$ bytes, \texttt{6c2035cc\ldots 93511b21}), preserved unchanged. It supplies the thirty-one codes, the twenty-eight models of \cref{tab:positive-models}, the three compressed forests with their declared hashes and one declared transcription correction, the recovered validator, the proof text for $H(44340)$ adapted in \cref{prop:noK8-44340}, and the two scripts that build and check \path{q4_j8.txt}.
\item The materialized transport checkpoint \path{apex_materialized_transport_verified.zip} (\texttt{c4c46978\ldots dbc8f400}); every file used from it was re-derived from the recovered source and compared byte for byte.
\item The companion order-twenty package \path{order20_q4_exclusion_package.zip} (\texttt{007b5118\ldots 1e74c05a}); the files \path{data/q4_prune.txt} and \path{data/des_models.txt} are used through byte-identical copies.
\item The version-2 article archive and PDF, and the native Q4 referee supplement, preserved unchanged.
\end{itemize}
\end{sloppypar}
The structural proofs of \cref{lem:rooted-44348,prop:noK8-44348} and the version-3 verification programs were introduced in version~3.

\subsection{Companion dependency chain}
For reading the larger companion proof, the archive contains the following complete mathematical sources and certificate specifications. Short filenames below are resolved unambiguously by the machine-readable map.
\begin{enumerate}
\item \textbf{Orders at most eighteen.} The excess-six proof yields the fourteen-vertex $K_8$ witness bound. The subsequent absorption and all-double constructions are in \path{update_absorption.tex}. The five order-eighteen profiles are assembled in \path{order18_closure_report.tex}; its two finite mixed-profile lemmas have their domain and coverage records in the same archive. The corrected connected cycle argument is \path{c9_symbolic_boundary_and_core_proof.md}; disconnected cycle backgrounds are proved in \path{three_singleton_disconnected_backgrounds_proof_audit.md}.
\item \textbf{Reduction of the thirteen order-nineteen profiles.} \path{order19_large_fibres.md} proves the support and quadruple-fibre restriction. \path{rooted_triangle.md} and \path{pure_triple_profile_proof.md} supply the rooted triangle and five-triple arguments. The remaining quadruple profile is reduced to the explicitly stated core in \path{quad_refinement.md}; its rooted $K_4$ proof is \path{rooted_k4_core_proof.md}. The associated double--triple interface justification is in the QREF addendum.
\item \textbf{Three singleton fibres.} The independent three-singleton audit specifies the complete orbit cover and validates the rooted-core certificates. Its $589,072$ certificate records cover $51,118,080$ fixed-pattern states by the proved orbit partition. This count is not a claim of that many directly stored records. The residual contraction check is corroboration, not a substitute for the stated certificate domain.
\item \textbf{Two singleton fibres.} \path{NEW_RESULTS.md} and the rooted-$K_6$ report give the arbitrary rooted-$K_5$ lemma on $(2^3,3^2)$, the forced interfaces, the restricted rooted-$K_5$ lemma on $(2^2,3^3)$, and the complete no-active-pair constructions. The arbitrary lemma has $2,097,152$ named settings; the restricted lemma has $8,748$ normal forms covering $314,928$ fixed-label settings by explicitly root-preserving transports.
\item \textbf{One singleton fibre.} \path{ONE_SINGLETON_PROOF.md} gives the complete absorber and forced-interface argument. Its arbitrary rooted-$K_6$ premise on $(2^5,3)$ has $2,097,152$ settings. Its restricted premise on four doubles and two all-odd, perfectly matched triples has $419,904$ normalized settings, covering $2,519,424$ labelled settings by matching relabelling. Both use unchanged original odd sets.
\item \textbf{No singleton fibre.} The remaining $(2^8,3)$ case is proved in Section~\ref{sec:order19}; its finite restricted B7 premise belongs to the full historical archive and was checked in the earlier order-nineteen review leg. It is inherited here and is not replayed by the current compact attachment supplement.
\end{enumerate}
Every finite positive-model checker verifies range, nonemptiness, disjointness, connectedness, intersection with the specified original odd set, and all required pair adjacencies, together with the declared graph domain. Domain exhaustion or a proved covering map is essential: an unverified success count alone is not an exhaustive proof. The source map is a navigation aid; it does not replace these proofs and checks.

\section{Exhaustive corroboration for the fifteen vertex graph}\label{app:h15-computation}
The following independent line of verification was completed before the structural exclusion. Its finite counts are retained for reproducibility.

By \cref{lem:h-witness,lem:pivot}, a singleton pivot must be one of $s,o_0,\ldots,o_6$, and each has exactly seven neighbours. Fix such a pivot $x$ and order its neighbours $r_0,\ldots,r_6$ by the numerical encoding $o_i=2i$, $e_i=2i+1$, $s=14$. Order the seven remaining vertices as $f_0,\ldots,f_6$. A candidate model is fully described by a map
\[
 a:\{f_0,\ldots,f_6\}\longrightarrow\{\bot,0,\ldots,6\}.
\]
Its seven other branches are $B_i=\{r_i\}\cup a^{-1}(i)$; $\bot$ means unused. This includes every model with pivot $x$, because \cref{lem:pivot} permits no second neighbour in any branch.

There are $8^7=2,097,152$ assignments for each pivot. For every assignment the verifier tests connectedness of each $H[B_i]$ and adjacency of every pair $B_i,B_j$. Branches are automatically nonempty and disjoint. Their adjacency to $\{x\}$ is automatic through $r_i$. A passing assignment would therefore be exactly a $K_8$ model with pivot $x$, and every such model would pass.

For transparency, connectedness is computed as the fixed point of
\[
 R_0=\{\min B_i\},\qquad
 R_{k+1}=R_k\cup\bigl(N_H(R_k)\cap B_i\bigr).
\]
It passes exactly when the fixed point is $B_i$. Pair adjacency is the direct test $N_H(B_i)\cap B_j\ne\varnothing$. No search pruning or isomorphism reduction is used. The complete source appears in Appendix~\ref{app:verifier}. The executed counts are in \cref{tab:h15}; all $8\cdot8^7=16,777,216$ assignments were checked and none passed. The reduction above proves exhaustion, and the direct predicates prove correctness of the test. This gives a separate computer-assisted verification of \cref{prop:noK8}. It is corroboration and is not a premise of the structural proof in Section~\ref{sec:sharpness}.

\begin{table}[ht]
\centering
\begin{tabular}{lrrr}
\toprule
Pivot & Assignments & All seven branches connected & Models\\
\midrule
$o_0$ & 2,097,152 & 53,927 & 0\\
$o_1$ & 2,097,152 & 53,927 & 0\\
$o_2$ & 2,097,152 & 36,010 & 0\\
$o_3$ & 2,097,152 & 53,927 & 0\\
$o_4$ & 2,097,152 & 53,927 & 0\\
$o_5$ & 2,097,152 & 78,230 & 0\\
$o_6$ & 2,097,152 & 78,230 & 0\\
$s$   & 2,097,152 & 43,544 & 0\\
\bottomrule
\end{tabular}
\caption{Fresh exhaustive verifier results. All pivots are checked individually; equal counts are not used as a symmetry argument.}\label{tab:h15}
\end{table}

\section{Root-apex data and finite corroboration}\label{app:root-apex-data}
\subsection{Explicit models for the twenty-eight positive codes}
\Cref{tab:positive-models} lists, for each of the twenty-eight codes in \eqref{eq:thirtyone} other than $42232,44340,44348$, eight vertex sets of $H(m)$ forming a $K_8$ model. In each row every set is connected, the sets are disjoint, and every two of them are joined by an edge; this is checked directly from the definition of $H(m)$. The rows are transcribed from the recovered source file \path{models.json} (blocked vertex numbering $o_i=i$, $e_i=7+i$, $s=14$).

{\small
\begin{longtable}{>{\raggedright\arraybackslash}p{1.25cm}>{\raggedright\arraybackslash}p{13.9cm}}
\caption{Explicit $K_8$ models in the twenty-eight positive root-apex graphs.}\label{tab:positive-models}\\
\toprule
Code&Branch sets\\
\midrule
\endfirsthead
\toprule
Code&Branch sets\\
\midrule
\endhead
42234 & $\{e_{0}\}\;\{o_{0},e_{4}\}\;\{o_{1},e_{2}\}\;\{o_{2},e_{5}\}\;\{o_{3},o_{5}\}\;\{o_{4},e_{3}\}\;\{o_{6},s\}\;\{e_{1},e_{6}\}$\\
42238 & $\{e_{0}\}\;\{o_{0},e_{4}\}\;\{o_{1},e_{2}\}\;\{o_{2},e_{5}\}\;\{o_{3},o_{5}\}\;\{o_{4},e_{3}\}\;\{o_{6},s\}\;\{e_{1},e_{6}\}$\\
42486 & $\{e_{0}\}\;\{o_{0},e_{3}\}\;\{o_{1},e_{2}\}\;\{o_{2},e_{5}\}\;\{o_{3},e_{4}\}\;\{o_{4},o_{5}\}\;\{o_{6},s\}\;\{e_{1},e_{6}\}$\\
44280 & $\{o_{0}\}\;\{o_{1}\}\;\{o_{2},o_{4}\}\;\{o_{3},o_{5}\}\;\{o_{6},e_{0},s\}\;\{e_{1},e_{6}\}\;\{e_{2},e_{5}\}\;\{e_{3},e_{4}\}$\\
44281 & $\{o_{0}\}\;\{o_{2}\}\;\{e_{1}\}\;\{o_{1},o_{4},e_{6}\}\;\{o_{3},o_{6}\}\;\{o_{5},e_{0},s\}\;\{e_{2},e_{5}\}\;\{e_{3},e_{4}\}$\\
44283 & $\{o_{0}\}\;\{o_{3}\}\;\{e_{1}\}\;\{e_{2}\}\;\{o_{1},o_{5},e_{6}\}\;\{o_{2},o_{6},e_{5}\}\;\{o_{4},e_{0},s\}\;\{e_{3},e_{4}\}$\\
44287 & $\{o_{3}\}\;\{e_{0}\}\;\{e_{1}\}\;\{e_{2}\}\;\{o_{0},e_{4}\}\;\{o_{1},o_{5},e_{6}\}\;\{o_{2},o_{6},e_{5}\}\;\{o_{4},e_{3},s\}$\\
44341 & $\{o_{0}\}\;\{o_{2}\}\;\{o_{1},e_{6}\}\;\{o_{3},s\}\;\{o_{4},o_{5}\}\;\{o_{6},e_{1}\}\;\{e_{0},e_{5}\}\;\{e_{2},e_{3},e_{4}\}$\\
44403 & $\{o_{0}\}\;\{o_{4}\}\;\{e_{1}\}\;\{o_{1},e_{4},s\}\;\{o_{2},e_{5}\}\;\{o_{3},o_{6}\}\;\{o_{5},e_{0},e_{6}\}\;\{e_{2},e_{3}\}$\\
44404 & $\{o_{0}\}\;\{o_{1},o_{5}\}\;\{o_{2},o_{6}\}\;\{o_{3},s\}\;\{o_{4},e_{1}\}\;\{e_{0},e_{6}\}\;\{e_{2},e_{5}\}\;\{e_{3},e_{4}\}$\\
44405 & $\{e_{0}\}\;\{e_{1}\}\;\{o_{0},e_{3}\}\;\{o_{1},e_{6}\}\;\{o_{2},e_{2},e_{5}\}\;\{o_{3},e_{4}\}\;\{o_{4},o_{5}\}\;\{o_{6},s\}$\\
44412 & $\{o_{0}\}\;\{o_{1},o_{5}\}\;\{o_{2},o_{6}\}\;\{o_{3},e_{4}\}\;\{o_{4},s\}\;\{e_{0},e_{5}\}\;\{e_{1},e_{6}\}\;\{e_{2},e_{3}\}$\\
44472 & $\{o_{0}\}\;\{o_{1}\}\;\{o_{2},o_{6}\}\;\{o_{3},o_{5}\}\;\{o_{4},s\}\;\{e_{0},e_{4}\}\;\{e_{1},e_{6}\}\;\{e_{2},e_{3},e_{5}\}$\\
44476 & $\{o_{3}\}\;\{o_{5}\}\;\{e_{0}\}\;\{o_{0},e_{5},s\}\;\{o_{1},e_{4}\}\;\{o_{2},o_{6}\}\;\{o_{4},e_{1},e_{6}\}\;\{e_{2},e_{3}\}$\\
44536 & $\{o_{0}\}\;\{o_{3}\}\;\{o_{1},o_{5}\}\;\{o_{2},o_{6}\}\;\{o_{4},s\}\;\{e_{0},e_{4}\}\;\{e_{1},e_{6}\}\;\{e_{2},e_{3},e_{5}\}$\\
44856 & $\{o_{0}\}\;\{o_{1}\}\;\{o_{2}\}\;\{o_{3},o_{6}\}\;\{o_{4},e_{0},e_{6}\}\;\{o_{5},s\}\;\{e_{1},e_{5}\}\;\{e_{2},e_{3},e_{4}\}$\\
44858 & $\{o_{2}\}\;\{e_{0}\}\;\{o_{0},e_{5}\}\;\{o_{1},e_{4}\}\;\{o_{3},o_{6}\}\;\{o_{4},e_{1},e_{6}\}\;\{o_{5},s\}\;\{e_{2},e_{3}\}$\\
44922 & $\{o_{0}\}\;\{o_{1}\}\;\{o_{3}\}\;\{e_{2}\}\;\{o_{2},o_{6},e_{5}\}\;\{o_{4},e_{0},s\}\;\{o_{5},e_{1},e_{6}\}\;\{e_{3},e_{4}\}$\\
44924 & $\{o_{3}\}\;\{e_{0}\}\;\{o_{0},e_{4}\}\;\{o_{1},e_{5}\}\;\{o_{2},o_{6}\}\;\{o_{4},s\}\;\{o_{5},e_{1},e_{6}\}\;\{e_{2},e_{3}\}$\\
48564 & $\{o_{0}\}\;\{o_{1}\}\;\{e_{3}\}\;\{o_{2},o_{6}\}\;\{o_{3},e_{0},s\}\;\{o_{4},o_{5}\}\;\{e_{1},e_{6}\}\;\{e_{2},e_{4},e_{5}\}$\\
48628 & $\{o_{0}\}\;\{o_{1}\}\;\{e_{3}\}\;\{o_{2},e_{4}\}\;\{o_{3},o_{4},o_{5}\}\;\{o_{6},e_{0},s\}\;\{e_{1},e_{6}\}\;\{e_{2},e_{5}\}$\\
48756 & $\{o_{0}\}\;\{o_{1}\}\;\{o_{2},o_{3},e_{4}\}\;\{o_{4},o_{6}\}\;\{o_{5},s\}\;\{e_{0},e_{5}\}\;\{e_{1},e_{6}\}\;\{e_{2},e_{3}\}$\\
48820 & $\{o_{0}\}\;\{o_{1}\}\;\{o_{2}\}\;\{e_{3}\}\;\{o_{3},e_{0},e_{4},e_{6}\}\;\{o_{4},o_{6}\}\;\{o_{5},e_{1},s\}\;\{e_{2},e_{5}\}$\\
48824 & $\{o_{0}\}\;\{o_{2}\}\;\{o_{1},o_{4}\}\;\{o_{3},o_{6}\}\;\{o_{5},s\}\;\{e_{0},e_{4}\}\;\{e_{1},e_{3},e_{6}\}\;\{e_{2},e_{5}\}$\\
48828 & $\{o_{4}\}\;\{e_{0}\}\;\{e_{3}\}\;\{o_{0},o_{2},o_{6}\}\;\{o_{1},e_{6}\}\;\{o_{3},e_{4},s\}\;\{o_{5},e_{1}\}\;\{e_{2},e_{5}\}$\\
61112 & $\{o_{0}\}\;\{o_{1}\}\;\{o_{2}\}\;\{o_{3},o_{4},o_{5}\}\;\{o_{6},s\}\;\{e_{0},e_{5}\}\;\{e_{1},e_{6}\}\;\{e_{2},e_{3},e_{4}\}$\\
111980 & $\{o_{0}\}\;\{o_{1}\}\;\{o_{2},o_{3}\}\;\{o_{4},s\}\;\{o_{5},o_{6}\}\;\{e_{0},e_{4}\}\;\{e_{1},e_{6}\}\;\{e_{2},e_{3},e_{5}\}$\\
112110 & $\{o_{2}\}\;\{o_{3}\}\;\{o_{6}\}\;\{e_{0}\}\;\{o_{0},o_{5},e_{4}\}\;\{o_{1},e_{2},s\}\;\{o_{4},e_{3},e_{5}\}\;\{e_{1},e_{6}\}$\\
\bottomrule
\end{longtable}
}

\subsection{Explicit models with one apex--even edge}\label[appendix]{subsec:hard-models}
\Cref{tab:hard-models} lists, for each of the three hard codes $h$ and each index $k$, eight vertex sets of $H_{\{k\}}(h)$, the graph $H(h)$ with the single added edge $se_k$ (\cref{sec:attachment}). In each row every set is connected, the sets are disjoint, and every two of them are joined by an edge. The models for $k$ in one orbit of the automorphism group of $h$ were obtained from one representative model by that automorphism, and every row was then checked directly; no symmetry argument is needed to read the table. The orbits are $\{0,1\},\{2,3,4,5\},\{6\}$ for $42232$, $\{0,3\},\{1\},\{2,5\},\{4,6\}$ for $44340$ and $\{0\},\{1,2\},\{3,4\},\{5,6\}$ for $44348$, eleven representatives in all.

{\small
\begin{longtable}{>{\raggedright\arraybackslash}p{1.0cm}>{\raggedright\arraybackslash}p{0.45cm}>{\raggedright\arraybackslash}p{13.4cm}}
\caption{Explicit $K_8$ models in $H_{\{k\}}(h)$ for the three hard codes.}\label{tab:hard-models}\\
\toprule
Code&$k$&Branch sets\\
\midrule
\endfirsthead
\toprule
Code&$k$&Branch sets\\
\midrule
\endhead
42232 & 0 & $\{s\} \;\{e_{0}\} \;\{o_{4}\} \;\{o_{5}\} \;\{o_{6}\} \;\{o_{1},e_{1},e_{6}\} \;\{o_{2},e_{2},e_{5}\} \;\{o_{3},e_{3},e_{4}\}$\\
42232 & 1 & $\{s\} \;\{e_{1}\} \;\{o_{2}\} \;\{o_{3}\} \;\{o_{6}\} \;\{o_{0},e_{0},e_{6}\} \;\{o_{4},e_{3},e_{4}\} \;\{o_{5},e_{2},e_{5}\}$\\
42232 & 2 & $\{o_{0}\} \;\{o_{1}\} \;\{o_{6},s\} \;\{o_{2},o_{4}\} \;\{o_{3},o_{5}\} \;\{e_{0},e_{6}\} \;\{e_{2},e_{5}\} \;\{e_{1},e_{3},e_{4}\}$\\
42232 & 3 & $\{o_{0}\} \;\{o_{1}\} \;\{o_{6},s\} \;\{o_{3},o_{5}\} \;\{o_{2},o_{4}\} \;\{e_{0},e_{6}\} \;\{e_{3},e_{4}\} \;\{e_{1},e_{2},e_{5}\}$\\
42232 & 4 & $\{o_{1}\} \;\{o_{0}\} \;\{o_{6},s\} \;\{o_{2},o_{4}\} \;\{o_{3},o_{5}\} \;\{e_{1},e_{6}\} \;\{e_{3},e_{4}\} \;\{e_{0},e_{2},e_{5}\}$\\
42232 & 5 & $\{o_{1}\} \;\{o_{0}\} \;\{o_{6},s\} \;\{o_{3},o_{5}\} \;\{o_{2},o_{4}\} \;\{e_{1},e_{6}\} \;\{e_{2},e_{5}\} \;\{e_{0},e_{3},e_{4}\}$\\
42232 & 6 & $\{s\} \;\{o_{2}\} \;\{o_{0},e_{4}\} \;\{o_{3},o_{5}\} \;\{o_{4},e_{3}\} \;\{o_{6},e_{0}\} \;\{e_{1},e_{6}\} \;\{o_{1},e_{2},e_{5}\}$\\
44340 & 0 & $\{s\} \;\{o_{1}\} \;\{o_{3}\} \;\{o_{5}\} \;\{o_{2},o_{6}\} \;\{e_{0},e_{6}\} \;\{o_{0},e_{2},e_{3}\} \;\{o_{4},e_{1},e_{4}\}$\\
44340 & 1 & $\{s\} \;\{o_{0}\} \;\{o_{1}\} \;\{o_{2}\} \;\{o_{4},o_{5}\} \;\{o_{3},e_{2},e_{5}\} \;\{o_{6},e_{0},e_{6}\} \;\{e_{1},e_{3},e_{4}\}$\\
44340 & 2 & $\{s\} \;\{o_{0}\} \;\{o_{1}\} \;\{o_{2}\} \;\{o_{4},o_{5}\} \;\{o_{3},e_{0},e_{5}\} \;\{o_{6},e_{1},e_{6}\} \;\{e_{2},e_{3},e_{4}\}$\\
44340 & 3 & $\{s\} \;\{o_{1}\} \;\{o_{0}\} \;\{o_{2}\} \;\{o_{4},o_{5}\} \;\{e_{3},e_{4}\} \;\{o_{3},e_{0},e_{5}\} \;\{o_{6},e_{1},e_{6}\}$\\
44340 & 4 & $\{s\} \;\{o_{0}\} \;\{o_{1}\} \;\{o_{2}\} \;\{o_{4},o_{5}\} \;\{e_{3},e_{4}\} \;\{o_{3},e_{0},e_{5}\} \;\{o_{6},e_{1},e_{6}\}$\\
44340 & 5 & $\{s\} \;\{o_{3}\} \;\{o_{1}\} \;\{o_{5}\} \;\{o_{2},o_{6}\} \;\{o_{0},e_{2},e_{3}\} \;\{o_{4},e_{1},e_{4}\} \;\{e_{0},e_{5},e_{6}\}$\\
44340 & 6 & $\{s\} \;\{o_{3}\} \;\{o_{1}\} \;\{o_{5}\} \;\{o_{2},o_{6}\} \;\{e_{0},e_{6}\} \;\{o_{0},e_{2},e_{3}\} \;\{o_{4},e_{1},e_{4}\}$\\
44348 & 0 & $\{s\} \;\{o_{1}\} \;\{o_{3}\} \;\{o_{5}\} \;\{o_{2},o_{6}\} \;\{e_{0},e_{6}\} \;\{o_{0},e_{2},e_{3}\} \;\{o_{4},e_{1},e_{4}\}$\\
44348 & 1 & $\{s\} \;\{o_{1}\} \;\{o_{3}\} \;\{o_{5}\} \;\{o_{2},o_{6}\} \;\{o_{4},e_{3},e_{4}\} \;\{o_{0},e_{2},e_{5}\} \;\{e_{0},e_{1},e_{6}\}$\\
44348 & 2 & $\{s\} \;\{o_{2}\} \;\{o_{4}\} \;\{o_{6}\} \;\{o_{1},o_{5}\} \;\{o_{3},e_{3},e_{4}\} \;\{o_{0},e_{1},e_{6}\} \;\{e_{0},e_{2},e_{5}\}$\\
44348 & 3 & $\{s\} \;\{o_{2}\} \;\{e_{3}\} \;\{o_{6},e_{0}\} \;\{o_{4},e_{1}\} \;\{o_{1},e_{4}\} \;\{o_{0},e_{6}\} \;\{o_{5},e_{2},e_{5}\}$\\
44348 & 4 & $\{s\} \;\{o_{1}\} \;\{e_{4}\} \;\{o_{5},e_{0}\} \;\{o_{3},e_{2}\} \;\{o_{2},e_{3}\} \;\{o_{0},e_{5}\} \;\{o_{6},e_{1},e_{6}\}$\\
44348 & 5 & $\{s\} \;\{o_{2}\} \;\{o_{4}\} \;\{o_{6}\} \;\{o_{1},o_{5}\} \;\{e_{0},e_{5}\} \;\{o_{3},e_{2},e_{3}\} \;\{o_{0},e_{1},e_{4}\}$\\
44348 & 6 & $\{s\} \;\{o_{1}\} \;\{o_{3}\} \;\{o_{5}\} \;\{o_{2},o_{6}\} \;\{e_{0},e_{6}\} \;\{o_{4},e_{1},e_{4}\} \;\{o_{0},e_{2},e_{3}\}$\\
\bottomrule
\end{longtable}
}

\subsection{Finite corroboration of \texorpdfstring{\cref{thm:root-apex}}{the classification}}\label[appendix]{subsec:apex-corroboration}
\Cref{thm:root-apex} is proved by \cref{tab:positive-models} and the structural \cref{prop:noK8,prop:noK8-44340,prop:noK8-44348}. The following finite checks are independent corroboration and are not premises.

\emph{Refutation forests.} For each negative host $G$, a forest has one root for every nonempty clique $S$ of vertices of degree at least seven ($53$, $53$ and $61$ cliques for $42232$, $44340$, $44348$; the last set includes all twelve cliques containing $e_0$). A node records the uncovered set $R$, the number $r$ of branches still to be chosen and the branches already chosen. It is a valid leaf if $|R|<2r$ or $|R|>(|S|+1)r$, or if no permissible branch contains the least vertex of $R$; otherwise its children must be exactly all permissible branches containing that vertex. A branch is permissible if it is connected, has between two and $|S|+1$ vertices, lies in $R$, leaves at least two vertices for each later branch, is adjacent to every chosen branch, contains exactly one neighbour of each degree-seven vertex of $S$ and at least one neighbour of $e_0$ if $e_0\in S$. By \cref{lem:normal-form} every spanning model with singleton set $S$ follows one root-to-node path through such a forest to a state with $R=\varnothing$ and $r=0$, which a valid forest never reaches. The three forests have $1{,}062$, $1{,}224$ and $2{,}683$ nodes. They were recovered from a transcribed compressed stream; the source declares one transcription correction in the $44348$ stream (deleting the two characters \texttt{u8} from a unique fourteen-character substring), and only that correction was applied. The uncorrected stream is not valid Base64, and the corrected stream reproduces the declared compressed and decompressed SHA-256 values.

\emph{Validators.} The recovered validator was replayed in a clean directory: it accepts the $28$ models and the three forests, agrees with its own search on exactly three negatives, rejects ten forest corruptions and four model corruptions, and finds the required model in a subdivided-$K_8$ positive control. A separately written checker binds every datum to the recovered source by SHA-256, recomputes the permissible branches by its own subset enumeration, validates all $4{,}969$ forest nodes and all $167$ singleton sets, rechecks the $28$ models and the oddomorphism equations, and decides all thirty-one graphs by an exhaustive search over spanning partitions into eight connected parts that uses only the conditions that singleton parts have degree at least seven and that parts are pairwise adjacent. It finds exactly the three negatives. Regenerating each forest from the branching rule above reproduces the recovered forest exactly, up to the order of children; applying the same rule to two positive codes produces trees that reach a model state, which the checker rejects. Fifteen further forest and model corruptions, including omission of the singleton sets $\{e_0\}$ and $\{e_0,o_3,o_5\}$, are rejected.

\emph{Local claims.} Every partner list, index set, free-vertex set, adjacency and nonadjacency fact, path-length bound, symmetry and orbit list quoted in the proofs of \cref{lem:rooted-44348,prop:noK8-44348} was re-derived mechanically from the definition of $H(44348)$ ($233$ recorded assertions), and likewise for \cref{lem:rooted-44340,prop:noK8-44340} ($59$ assertions). These checks confirm the stated facts; the deductions connecting them are the written proofs.

\paragraph{Revision-4 audit.} A separately authored exact-cover checker freshly verifies all $167$ exact singleton sets across the three hosts, including all $61$ for $44348$. A second checker enumerates spanning connected partitions without the singleton reduction and also excludes all three hosts. The original local-claim program contains one empty-message tautological assertion; it contributes no mathematical evidence. The $233$ count refers to its nonempty-message checks.

\section{Complete source of the fifteen vertex verifier}\label{app:verifier}
The following source implements the exhaustive argument of Appendix~\ref{app:h15-computation}, which corroborates \cref{prop:noK8}. It uses no solver, heuristic search, graph database, or symmetry reduction. Its positive control adds the missing edges among the seven odd double vertices, producing an explicit $K_8$ with the singleton.
\begin{lstlisting}[language=C++,basicstyle=\ttfamily\scriptsize]
// Exhaustive ordinary K8-minor exclusion for the explicitly defined H15.
// C++17, standard library only. No solver or graph-isomorphism software.
#include <array>
#include <cstdint>
#include <iostream>
#include <vector>
#include <string>
using Mask = std::uint32_t;
using Graph = std::array<Mask,15>;
static bool connected(const Graph& g, Mask b) {
  if (!b) return false;
  Mask reach=b & (~b+1U);
  for (;;) {
    Mask next=reach;
    for(unsigned v=0;v<15;++v) if(reach & (1U<<v)) next |= g[v]&b;
    if(next==reach) return reach==b;
    reach=next;
  }
}
int main(int argc, char** argv) {
  const bool positive = argc==2 && std::string(argv[1])=="--positive-control";
  if(argc>1 && !positive) return 9;
  Graph g{};
  auto edge=[&](unsigned u,unsigned v){g[u]|=1U<<v;g[v]|=1U<<u;};
  bool theta[7][7]{};
  const unsigned es[8][2]={{0,3},{0,5},{1,4},{1,5},{2,5},{2,6},{3,6},{4,6}};
  for(const auto& e:es) theta[e[0]][e[1]]=theta[e[1]][e[0]]=true;
  // o_i=2i, e_i=2i+1, s=14.
  for(unsigned i=0;i<7;++i) {
    edge(2*i,14);
    for(unsigned j=i+1;j<7;++j) {
      if(theta[i][j]) {edge(2*i,2*j+1);edge(2*i+1,2*j);edge(2*i+1,2*j+1);}
      else edge(2*i,2*j);
    }
  }
  // Positive control: the eight vertices s,o_0,...,o_6 induce K8.
  if(positive) for(unsigned i=0;i<7;++i) for(unsigned j=i+1;j<7;++j) edge(2*i,2*j);
  unsigned edges=0;
  for(unsigned i=0;i<15;++i) {
    if(g[i]&(1U<<i)) return 2;
    for(unsigned j=i+1;j<15;++j) edges += bool(g[i]&(1U<<j));
    for(unsigned j=0;j<15;++j) if(bool(g[i]&(1U<<j))!=bool(g[j]&(1U<<i))) return 3;
  }
  if(edges!=(positive?52U:44U)) return 4;
  std::uint64_t total=0,total_models=0;
  std::cout << "{\"vertices\":15,\"edges\":" << edges << ",\"pivots\":[";
  bool first=true;
  for(unsigned pivot=0;pivot<15;++pivot) {
    std::vector<unsigned> roots,free;
    for(unsigned v=0;v<15;++v) if(v!=pivot) {
      if(g[pivot]&(1U<<v)) roots.push_back(v); else free.push_back(v);
    }
    if(roots.size()<7) continue;
    if(roots.size()!=7 || free.size()!=7) { if(positive) continue; return 5; }
    std::array<std::array<Mask,128>,7> branch{},neigh{};
    std::array<std::array<bool,128>,7> good{};
    for(unsigned r=0;r<7;++r) for(unsigned set=0;set<128;++set) {
      Mask b=1U<<roots[r];
      for(unsigned j=0;j<7;++j) if(set&(1U<<j)) b |= 1U<<free[j];
      branch[r][set]=b; good[r][set]=connected(g,b);
      Mask n=0; for(unsigned v=0;v<15;++v) if(b&(1U<<v)) n |= g[v];
      neigh[r][set]=n;
    }
    std::uint64_t connected_assignments=0,models=0;
    std::array<std::uint64_t,7> rejected_connectivity{};
    std::array<std::uint64_t,21> rejected_adjacency{};
    // Seven base-eight digits: zero means unused; 1..7 chooses a branch.
    for(unsigned a=0;a<(1U<<21);++a) {
      std::array<unsigned,7> owned{}; unsigned digits=a;
      for(unsigned j=0;j<7;++j,digits>>=3) {unsigned r=digits&7U; if(r) owned[r-1]|=1U<<j;}
      ++total; bool bad=false;
      for(unsigned r=0;r<7;++r) if(!good[r][owned[r]]) {++rejected_connectivity[r];bad=true;break;}
      if(bad) continue;
      ++connected_assignments; unsigned pair=0;
      for(unsigned r=0;r<7&&!bad;++r) for(unsigned t=r+1;t<7;++t,++pair)
        if(!(neigh[r][owned[r]]&branch[t][owned[t]])) {++rejected_adjacency[pair];bad=true;break;}
      if(!bad) ++models;
    }
    total_models+=models;
    if(!first) std::cout << ',';
    first=false;
    std::cout << "{\"pivot\":"<<pivot<<",\"assignments\":2097152,\"connected_assignments\":"<<connected_assignments<<",\"models\":"<<models<<",\"first_disconnected_branch\":[";
    for(unsigned i=0;i<7;++i) std::cout << (i?",":"") << rejected_connectivity[i];
    std::cout << "],\"first_nonadjacent_pair\":[";
    for(unsigned i=0;i<21;++i) std::cout << (i?",":"") << rejected_adjacency[i];
    std::cout << "]}";
  }
  std::cout << "],\"total_assignments\":"<<total<<",\"total_models\":"<<total_models<<"}\n";
  if(positive) return total_models>0 ? 0 : 1;
  return total==16777216 && total_models==0 ? 0 : 1;
}
\end{lstlisting}
\section{A structural augmentation construction}\label{ua:appendix}
This appendix replaces the twenty-one separate positive certificates in
Lemma~\ref{lem:hard-one-edge} by six formulas adapted to paths and a triangle.
The formulas establish the same ordinary-minor assertion; the twenty-one
original certificates remain separately checkable corroboration. The central formula
is valid for a partially specified active graph with four arbitrary pairs,
not merely for the three named hard graphs. None of the arguments below uses
Input~\ref{input:P7} or a negative search.

We use the interface definition of $K(B)$ and $H_T(B)$ from
Section~\ref{sec:attachment}. In the following constructions $o_i,e_i$
refer to the two vertices at the indicated index. An edge between two odd
vertices certifies an inactive index pair, while an edge with at least one
even end certifies an active index pair. Consequently, each displayed
crossing edge is checked directly against the stated path structure.

\subsection{A common routing formula for two ears}
\begin{lemma}[Two-ear routing]\label[lemma]{ua:twoear}
Let $B$ have distinct indices $u,v,w,a,b,c,d$. Require the pairs
\[
 uv,ua,ab,bv,uc,cd,dw,wv
\]
to be active and the pairs
\[
 ad,av,dv,bc,bd,cv,ac,du,aw
\]
to be inactive. Leave the four remaining pairs $uw,ub,wb,wc$ arbitrary.
For each $k\in\{u,c,a,d\}$, $H_{\{k\}}(B)$ has a $K_8$ model
with $\{s\}$ as a branch. More precisely, choose $z\in\{a,d\}$
with $k\in\{u,c,z\}$. The other seven branches are
\[
 D=\{o_d\},\quad V=\{o_v\},\quad A=\{o_a\},\quad
 R=\{o_b,o_c\},\quad Q=\{e_u,e_c,e_z\},\quad U,\quad W,
\]
where
\[
\begin{array}{c|cc}
z&U&W\\\hline
d&\{o_u,e_b,e_a\}&\{o_w,e_v,e_w\}\\
a&\{o_u,e_b,e_v\}&\{o_w,e_d,e_w\}.
\end{array}
\]
Thus $Q$ is the unique branch without an odd vertex; the mark $k$ supplies
exactly its missing adjacency to $\{s\}$.
\end{lemma}
\begin{proof}
The branches partition the fourteen vertices of $K(B)$. The branch $R$
is an edge because $bc$ is inactive. The branch $Q$ contains the edge
$e_ue_c$, and its third vertex attaches to $e_u$ when $z=a$ and to $e_c$
when $z=d$. For $z=d$, the branches $U,W$ contain the paths
$o_u e_a e_b$ and $o_w e_v e_w$; for $z=a$, they contain
$o_u e_v e_b$ and $o_w e_d e_w$. All these edges use required active
pairs.

The following array gives one edge for every pair among $D,V,A,R,Q,U,W$.
An entry with two alternatives uses the first for $z=d$ and the second
for $z=a$. No listed edge depends on any of the four arbitrary pairs.
\[
\begin{array}{c|cccccc}
 &V&A&R&Q&U&W\\\hline
D&o_do_v&o_do_a&o_do_b&o_de_c&o_do_u&o_de_w\\
V&&o_vo_a&o_vo_c&o_ve_u&o_ve_b&o_ve_w\\
A&&&o_ao_c&o_ae_u&o_ae_b&o_ao_w\\
R&&&&o_ce_u&o_be_a\,/\,o_be_v&o_be_v\,/\,o_ce_d\\
Q&&&&&e_ue_a\,/\,e_ue_v&e_ue_v\,/\,e_ce_d\\
U&&&&&&e_be_v\,/\,e_ve_w
\end{array}
\]
All entries are edges by the eight required active pairs and nine required
inactive pairs. Hence these seven branches form a $K_7$ model.
The vertex $s$ sees the six branches containing odd vertices, and the
added edge $se_k$ joins it to $Q$. This proves the result simultaneously
for every choice of the four arbitrary pairs.
\end{proof}

The mechanism is a routing choice: the even vertex at $a$ or $d$ is moved
into $Q$, while the two mixed branches exchange their even paths.
This preserves all branch adjacencies and determines which additional
mark can repair the sole missing adjacency to $s$.

\begin{corollary}\label[corollary]{ua:theta134}
If $B$ is the union of the internally disjoint $u$--$v$ paths
$uv$, $uabv$, and $ucdwv$, then $H_{\{k\}}(B)$ has a $K_8$ model
with branch $\{s\}$ for every index $k$.
\end{corollary}
\begin{proof}
Lemma~\ref{ua:twoear} covers $u,c,a,d$. The automorphism
$(u\ v)(a\ b)(c\ w)$ fixes $d$ and reverses all three paths.
Its images of the covered positions include $v,w,b$; together these
are all seven indices. Applying the same index permutation to odd and
even vertices, and fixing $s$, transports the models.
\end{proof}

\subsection{The triangle with two length-three ears}
Suppose now that $B$ is the triangle $uvw$ together with the two
internally disjoint paths $uabv$ and $ucdw$, with no other edges.
Lemma~\ref{ua:twoear} applies, since the optional pair $uw$ is active.
The automorphism $(v\ w)(a\ c)(b\ d)$ exchanges the ears.
The lemma and its image cover the five indices $u,a,b,c,d$.
The remaining two positions have the following different construction.

\begin{lemma}[Triangle pivot]\label[lemma]{ua:triangle}
For the triangle-with-ears graph just described,
$H_{\{v\}}(B)$ has the model $\{s\},X_1,\ldots,X_7$, where
\[
\begin{array}{c|l}
1&\{o_b\}\\
2&\{e_v\}\\
3&\{o_c,e_u\}\\
4&\{o_w,e_d\}\\
5&\{o_d,e_w\}\\
6&\{o_u,e_c\}\\
7&\{o_a,e_b,e_a\}.
\end{array}
\]
The corresponding model at $w$ is obtained by exchanging the ears.
\end{lemma}
\begin{proof}
The pair branches are edges because $uc,dw,dw,uc$ are active, and
$X_7$ contains the path $o_a e_b e_a$. All sets are disjoint.
The following triangular array gives their twenty-one crossing edges:
\[
\begin{array}{c|cccccc}
 &2&3&4&5&6&7\\\hline
1&o_be_v&o_bo_c&o_bo_w&o_bo_d&o_bo_u&o_be_a\\
2&&e_ve_u&e_vo_w&e_ve_w&e_vo_u&e_ve_b\\
3&&&e_uo_w&e_ue_w&e_ue_c&e_ue_a\\
4&&&&e_de_w&e_de_c&o_wo_a\\
5&&&&&e_wo_u&o_do_a\\
6&&&&&&o_ue_a
\end{array}
\]
Each entry follows from a triangle edge, an ear edge, or an inactive
pair between distinct ears. Six branches contain odd vertices; the
remaining branch is $\{e_v\}$. The added edge $se_v$ therefore gives
all seven adjacencies of the branch $\{s\}$. Exchanging the ears
proves the assertion at $w$.
\end{proof}

\subsection{The theta with path lengths two, three and three}
Let $B$ be the union of the three internally disjoint $u$--$v$ paths
\[
 uzv,\qquad uabv,\qquad ucdv,
\]
with no other edges. Its index positions fall into the endpoints
$u,v$, the short-path interior $z$, and the four long-path interiors
$a,b,c,d$. We give one construction for each position type.

\begin{lemma}[Three theta positions]\label[lemma]{ua:theta233}
For every index $k$ of this theta graph, $H_{\{k\}}(B)$ has a
$K_8$ minor. When $k=u,v,z$, it has such a model with branch $\{s\}$.
\end{lemma}
\begin{proof}
For $k=u$, take $\{s\}$ together with
\[
 \{e_u\},\ \{o_a\},\ \{o_c\},\ \{o_z\},\
 \{o_v,e_v,e_z\},\ \{o_b,e_b,e_a\},\ \{o_d,e_d,e_c\}.
\]
The triples contain the paths $o_v e_z e_v$, $o_b e_a e_b$ and
$o_d e_c e_d$. The three odd singleton branches form a clique, and
$e_u$ sees those three branches through the active pairs $ua,uc,uz$.
It sees the triples through $e_z,e_a,e_c$, respectively. Each odd
singleton sees its corresponding triple through an even vertex and
each other triple through its odd vertex. Among the triples, crossing
edges are $e_vo_b$, $e_vo_d$, and $o_bo_d$.
Thus they form a $K_7$; $se_u$ gives its even singleton the missing
adjacency to $s$. Reversing all paths proves the case $k=v$.

For $k=z$, take $\{s\}$ and the following seven branches:
\[
\begin{array}{c|l}
1&\{o_b\}\\
2&\{o_u,e_c\}\\
3&\{o_d,o_a\}\\
4&\{o_c,e_d\}\\
5&\{o_z,e_u\}\\
6&\{e_v,e_z\}\\
7&\{o_v,e_b,e_a\}.
\end{array}
\]
Each pair is an edge, and the triple contains $o_v e_b e_a$.
The following array supplies every crossing edge:
\[
\begin{array}{c|cccccc}
 &2&3&4&5&6&7\\\hline
1&o_bo_u&o_bo_d&o_bo_c&o_bo_z&o_be_v&o_be_a\\
2&&e_co_d&e_ce_d&e_ce_u&o_ue_z&o_ue_a\\
3&&&o_ao_c&o_ae_u&o_de_v&o_ae_b\\
4&&&&o_ce_u&e_de_v&e_do_v\\
5&&&&&e_ue_z&e_ue_a\\
6&&&&&&e_ve_b
\end{array}
\]
Six branches meet the odd set and the other contains $e_z$, so
adding $\{s\}$ works.

Finally, for $k\in\{a,b\}$, take
\[
\begin{split}
 X_1&=\{o_u\},\quad X_2=\{o_v\},\quad X_3=\{s,o_z\},\\
 X_4&=\{o_b,o_c\},\quad X_5=\{o_d,o_a\},\quad
 X_6=\{e_u,e_z\},\\
 X_7&=\{e_a,e_b\},\quad X_8=\{e_v,e_d,e_c\}.
\end{split}
\]
Their internal edges are $so_z$, $o_bo_c$, $o_do_a$, $e_ue_z$,
$e_ae_b$, and the path $e_v e_d e_c$. The singletons $X_1,X_2$
are adjacent to one another. In the branch order $X_3,\ldots,X_8$,
their crossing edges can be chosen respectively as
\[
\begin{array}{c|cccccc}
X_1&o_us&o_uo_b&o_uo_d&o_ue_z&o_ue_a&o_ue_c\\
X_2&o_vs&o_vo_c&o_vo_a&o_ve_z&o_ve_b&o_ve_d.
\end{array}
\]
Among the other six branches the crossing edges are
\[
\begin{array}{c|ccccc}
 &4&5&6&7&8\\\hline
3&o_zo_b&o_zo_a&o_ze_u&se_k&o_ze_v\\
4&&o_bo_d&o_ce_u&o_be_a&o_be_v\\
5&&&o_ae_u&o_ae_b&o_de_c\\
6&&&&e_ue_a&e_ue_c\\
7&&&&&e_be_v.
\end{array}
\]
The only missing crossing before augmentation is $X_3X_7$;
the edge $se_k$ supplies it for either $k=a$ or $k=b$.
Exchanging the two length-three paths handles $c,d$.
Unlike the preceding constructions, this one absorbs the odd vertex
at the short-path interior into the apex branch. It is an ordinary
model and makes no original-rooted assertion.
\end{proof}

\subsection{Identification and scope}
\begin{theorem}[Structural hard-core augmentation]\label[theorem]{ua:hard}
For $h\in\{42232,44340,44348\}$ and every index $k$,
$H_{\{k\}}(h)$ has an ordinary $K_8$ minor. This assertion has a
structural proof independent of the twenty-one stored models.
\end{theorem}
\begin{proof}
For $42232$, use Lemma~\ref{ua:theta233} with
\[
 (u,v,z,a,b,c,d)=(0,1,6,5,2,4,3).
\]
For $44340$, use Corollary~\ref{ua:theta134} with
\[
 (u,v,w,a,b,c,d)=(0,3,4,5,2,6,1).
\]
The same latter identification maps $44348$ to the triangle with two
ears, so Lemmas~\ref{ua:twoear} and~\ref{ua:triangle} apply.
Indeed the active edge sets are
$\{04,05,06,12,13,16,25,34\}$ for $42232$, the images of $uz,zv,ua,ab,bv,uc,cd,dv$;
$\{03,05,06,14,16,23,25,34\}$ for $44340$, the images of $uv,ua,ab,bv,uc,cd,dw,wv$;
and these together with $04$, the image of $uw$, for $44348$.
\end{proof}

The theorem is confined to these three cores; it is not a structural
proof of Input~\ref{input:P7}. The two-ear lemma additionally gives
a precisely stated partial-edge family, with four unrestricted pairs.
The proof uses six explicit path formulas (two routing formulas, one
triangle formula, and three theta-position formulas), with all
connections and crossings displayed. It does not infer nonexistence
from any unsuccessful search. The $1{,}013$ P7 positive models, their
$1{,}044$-orbit coverage, and the twenty-eight nonhard root-apex
models remain the finite premises of the general attachment theorem.
The root-control converse for the four long-path interiors of
$42232$ is a separate assertion and is not required here.

\section{The enlarged-root boundary for the three hard cores}\label{rc:sec}
For a seven-index graph $B$, write $O=\{o_0,\ldots,o_6\}$ and
\[
 R_k=O\cup\{e_k\}\qquad(0\le k\le6).
\]
An \emph{$R_k$-meeting $K_7$ model} in $K(B)$ means seven pairwise disjoint,
nonempty, connected, pairwise adjacent branch sets, each intersecting $R_k$.
There are eight allowed roots and seven branches. No prescribed root is assigned
to any branch, and a branch may contain two allowed roots.

\begin{theorem}[Exact enlarged-root classification]\label[theorem]{rc:classification}
For $h\in\{42232,44340,44348\}$, the graph $K(h)$ has an $R_k$-meeting
$K_7$ model exactly in the following cases:
\[
\begin{array}{c|c|c}
h&\text{marks admitting a model}&\text{marks admitting no model}\\\hline
42232&0,1,6&2,3,4,5\\
44340&0,1,2,3,4,5,6&\varnothing\\
44348&0,1,2,3,4,5,6&\varnothing.
\end{array}
\]
The positive assertions follow structurally from Appendix~\ref{ua:appendix}; explicit branch certificates below corroborate them. The negative
assertions are proved structurally below, using \cref{lem:known-theta} and no computation;
the finite certificates of revision~6 are retained as corroboration.
Neither direction assumes Input~\ref{input:P7} or any negative exploratory search.
\end{theorem}

\begin{lemma}[Eight-root spanning normalization]\label[lemma]{rc:normalization}
Let $G$ be a connected graph on fourteen vertices and let $R\subseteq V(G)$
have size eight. Every $R$-meeting $K_7$ model extends to a partition of $V(G)$
into seven branches still forming an $R$-meeting $K_7$ model. Such a partition
has exactly one branch containing two vertices of $R$ and six containing one.
For fixed orders on $R$ and $V(G)\setminus R$, its branches have a unique
description by one of $\binom82$ root pairs and one of $7^6$ allocations.
\end{lemma}
\begin{proof}
If vertices remain unused, connectedness gives an edge from a used vertex to
an unused vertex. Add the latter to the former's branch. This preserves
connectedness, disjointness, pairwise adjacency, and intersection with $R$.
Iteration gives a spanning partition. Its seven positive root counts sum to
eight, so they are $2,1,1,1,1,1,1$.

Label the branch containing the root pair by $0$. Label the remaining six
branches $1,\ldots,6$ by the increasing order of their unique roots. Each of
the six nonroots belongs to exactly one of these seven branches. Conversely,
each root pair and allocation specifies one partition meeting $R$, although
it need not be connected or pairwise adjacent. The description is unique.
Thus the exact candidate domain has
\[
 \binom82\,7^6=28\cdot117649=3294172
\]
elements, including every possible spanning model.
\end{proof}

\paragraph{Positive constructions.}
The table below gives models for one representative of each positive marked
index orbit. A permutation of the indices acts on both $o_i$ and $e_i$.
For $42232$, the automorphisms
\[
 g=(2\ 3)(4\ 5),\qquad h=(0\ 1)(2\ 4)(3\ 5)
\]
have index orbits $\{0,1\},\{2,3,4,5\},\{6\}$.
For $44340$, the automorphism $(0\ 3)(2\ 5)(4\ 6)$ has orbits
$\{0,3\},\{1\},\{2,5\},\{4,6\}$.
For $44348$, $(1\ 2)(3\ 4)(5\ 6)$ has orbits
$\{0\},\{1,2\},\{3,4\},\{5,6\}$.
Each displayed permutation preserves the relevant active edge set, so it
preserves $K(h)$ and carries $R_k$ to the corresponding marked root set.
Every row is checked directly for connectivity, disjointness, crossing
edges, and intersections with $R_k$. The complete seventeen positive models
are stored in \path{positive_models.json}; \path{positive_edge_witnesses.json}
names their internal spanning trees and all $17\cdot21=357$ crossing edges.
These positive models were transcribed from the supplied revision-5 rooted
exploration log and freshly verified; the log's negative entries play no role.

\begingroup\small
\begin{longtable}{r r p{0.72\textwidth}}
\caption{Enlarged-root positive representatives. Each row lists seven branches.}
\label{rc:positive-table}\\
\toprule $h$&$k$&Branches\\\midrule\endfirsthead
\toprule $h$&$k$&Branches\\\midrule\endhead
42232 & 0 & $\{o_{0},o_{1},e_{6}\}$, $\{o_{2},e_{5}\}$, $\{o_{3},e_{4}\}$, $\{o_{4},e_{3}\}$, $\{o_{5},e_{2}\}$, $\{o_{6},e_{1}\}$, $\{e_{0}\}$ \\
42232 & 6 & $\{o_{0},e_{4}\}$, $\{o_{1},e_{2},e_{5}\}$, $\{o_{2}\}$, $\{o_{3},o_{5}\}$, $\{o_{4},e_{3}\}$, $\{o_{6},e_{0}\}$, $\{e_{1},e_{6}\}$ \\
44340 & 0 & $\{o_{0}\}$, $\{o_{1}\}$, $\{o_{2}\}$, $\{o_{3},e_{2},e_{5}\}$, $\{o_{4},o_{5}\}$, $\{o_{6},e_{1},e_{6}\}$, $\{e_{0},e_{3},e_{4}\}$ \\
44340 & 1 & $\{o_{0}\}$, $\{o_{1}\}$, $\{o_{2}\}$, $\{o_{3},e_{2},e_{5}\}$, $\{o_{4},o_{5}\}$, $\{o_{6},e_{0},e_{6}\}$, $\{e_{1},e_{3},e_{4}\}$ \\
44340 & 2 & $\{o_{0}\}$, $\{o_{1}\}$, $\{o_{2}\}$, $\{o_{3},e_{0},e_{5}\}$, $\{o_{4},o_{5}\}$, $\{o_{6},e_{1},e_{6}\}$, $\{e_{2},e_{3},e_{4}\}$ \\
44340 & 4 & $\{o_{0}\}$, $\{o_{1}\}$, $\{o_{2}\}$, $\{o_{3},e_{0},e_{2},e_{5}\}$, $\{o_{4},o_{5}\}$, $\{o_{6},e_{1},e_{6}\}$, $\{e_{3},e_{4}\}$ \\
44348 & 0 & $\{o_{0},o_{1},e_{6}\}$, $\{o_{2},e_{2},e_{5}\}$, $\{o_{3}\}$, $\{o_{4},e_{3},e_{4}\}$, $\{o_{5}\}$, $\{o_{6},e_{1}\}$, $\{e_{0}\}$ \\
44348 & 1 & $\{o_{0},e_{2},e_{5}\}$, $\{o_{1}\}$, $\{o_{2},o_{6}\}$, $\{o_{3}\}$, $\{o_{4},e_{3},e_{4}\}$, $\{o_{5}\}$, $\{e_{0},e_{1},e_{6}\}$ \\
44348 & 3 & $\{o_{0},e_{4}\}$, $\{o_{1},e_{6}\}$, $\{o_{2}\}$, $\{o_{3},o_{6}\}$, $\{o_{4},e_{1}\}$, $\{o_{5},e_{2},e_{5}\}$, $\{e_{0},e_{3}\}$ \\
44348 & 5 & $\{o_{0},e_{3}\}$, $\{o_{1},e_{6}\}$, $\{o_{2},o_{6}\}$, $\{o_{3},e_{4}\}$, $\{o_{4},e_{1}\}$, $\{o_{5},e_{2}\}$, $\{e_{0},e_{5}\}$ \\
\bottomrule
\end{longtable}
\endgroup

\paragraph{Negative assertions.}
For $42232$ the automorphisms $1,g,h,hg$ carry index~$2$ to $2,3,4,5$, so the four negative marks form one index orbit, and the stabiliser of each is trivial. The structural proof below treats one representative without any further symmetry reduction.

\begin{corollary}[Exact singleton-apex boundary]\label[corollary]{rc:apex}
The augmented graph $H_{\{k\}}(h)$ has a $K_8$ model with singleton branch
$\{s\}$ exactly in the positive cases of \cref{rc:classification}.
For $h=42232$ and $k\in\{2,3,4,5\}$, every $K_8$ model that uses $s$ in a
branch uses at least two vertices in that branch. Two are sufficient, as
witnessed by the previously verified models having branch $\{s,o_6\}$.
\end{corollary}
\begin{proof}
The neighbours of $s$ in $H_{\{k\}}(h)$ are exactly $R_k$.
Deleting the branch $\{s\}$ from a $K_8$ model leaves an $R_k$-meeting
$K_7$ model in $K(h)$. Conversely, adjoining $\{s\}$ to any such $K_7$
model gives a $K_8$ model. Apply \cref{rc:classification}; the stated
two-vertex models give the upper bound. No assertion is made that every
nontrivial apex branch must contain $o_6$.
\end{proof}

\subsection*{A structural proof of the negative assertions}
\addcontentsline{toc}{subsection}{Structural proof of the negative assertions}
Throughout this subsection let $\Theta$ be the union of the three internally disjoint $u$--$v$ paths
\[
 uzv,\qquad uabv,\qquad ucdv .
\]
With $(u,v,z,a,b,c,d)=(0,1,6,5,2,4,3)$ this is the active graph of code $42232$; the four marks $2,3,4,5$ are $b,d,c,a$. In $K=K(\Theta)$ the odd vertex $o_x$ is adjacent to $o_y$ for $y\notin N[x]$ and to $e_y$ for $y\in N(x)$, so it has degree six, and $e_x$ is adjacent to $o_y$ and $e_y$ for $y\in N(x)$, so $e_u,e_v$ have degree six and the other even vertices degree four. Here and below $N$ denotes neighbourhoods in $\Theta$; $N(I)$ is the set of vertices outside $I$ with a neighbour in $I$, and $N[x]=N(x)\cup\{x\}$.

\begin{lemma}[The theta graph]\label[lemma]{rc:theta}
The graph $\Theta$ has the following properties.
\begin{enumerate}[label=(\alph*)]
\item The cycles of $\Theta$ have lengths $5,5,6$. Hence adjacent vertices of $\Theta$ have no common neighbour, and any two vertices have at most one.
\item Every nonempty independent set $I$ of $\Theta$ satisfies $|N(I)|\ge|I|+1$. The independent sets of size three are $ubd$, $vac$, $zac$, $zad$, $zbc$, $zbd$; each has $|N(I)|=4$, and there is no independent set of size four. The independent pairs with $|N(I)|=3$ are $za$, $zb$, $zc$, $zd$, $ac$, $bd$.
\item The nonadjacent pairs at distance three are $\{a,d\}$ and $\{b,c\}$. Each other nonadjacent pair has exactly one common neighbour, namely $uv\colon z$, $ub\colon a$, $ud\colon c$, $va\colon b$, $vc\colon d$, $za\colon u$, $zb\colon v$, $zc\colon u$, $zd\colon v$, $ac\colon u$, $bd\colon v$.
\item The automorphism group of $\Theta$ is $\{1,(u\,v)(a\,b)(c\,d),(a\,c)(b\,d),(u\,v)(a\,d)(b\,c)\}$; it acts simply transitively on $\{a,b,c,d\}$.
\end{enumerate}
\end{lemma}
\begin{proof}
(a) Each cycle is the union of two of the three paths.

(b) The vertices $u,a,b,v,z$ form a $5$-cycle $Q$, and the remaining vertices $c,d$ are adjacent, with $N(c)=\{u,d\}$ and $N(d)=\{c,v\}$. Let $J\subseteq V(Q)$ be nonempty and independent. For a cyclic orientation of $Q$ the successor map and the predecessor map both inject $J$ into $N_Q(J)$. If $|N_Q(J)|=|J|$, both are bijections onto $N_Q(J)$, so $J$ is invariant under the rotation by two steps, which generates all rotations of a $5$-cycle; then $J=V(Q)$, which is not independent. Hence $|N_Q(J)|\ge|J|+1$. An independent set $I$ meets $\{c,d\}$ in at most one vertex. If $I\cap\{c,d\}=\varnothing$, then $|N(I)|\ge|N_Q(I)|\ge|I|+1$. If $c\in I$, then $d\in N(I)\setminus V(Q)$, so $|N(I)|\ge|N_Q(I\setminus\{c\})|+1\ge|I|+1$ when $I\ne\{c\}$, and $N(\{c\})=\{u,d\}$; the case $d\in I$ is symmetric. An independent set has at most two vertices in $Q$ and at most one in $\{c,d\}$. The independent pairs in $Q$ are $ub,uv,va,za,zb$; adding $c$ (adjacent to $u$) or $d$ (adjacent to $v$) when possible gives exactly the six triples listed, and their neighbourhoods $\{z,a,c,v\}$, $\{z,b,d,u\}$, $\{u,v,b,d\}$, $\{u,v,b,c\}$, $\{u,v,a,d\}$, $\{u,v,a,c\}$ have four vertices. For a nonadjacent pair $\{x,y\}$ we have $|N(\{x,y\})|=\deg x+\deg y-|N(x)\cap N(y)|\ge3$, with equality exactly when $x,y$ have degree two and a common neighbour. By (c) these are the six pairs listed.

(c) This is a direct check of the neighbourhoods $N(u)=\{z,a,c\}$, $N(v)=\{z,b,d\}$, $N(z)=\{u,v\}$, $N(a)=\{u,b\}$, $N(b)=\{a,v\}$, $N(c)=\{u,d\}$, $N(d)=\{c,v\}$; uniqueness is part of (a).

(d) An automorphism fixes the set $\{u,v\}$ of vertices of degree three and permutes the three $u$--$v$ paths preserving their lengths; it is determined by these two actions, and all four combinations occur.
\end{proof}

By \cref{rc:theta}(d), an automorphism $\varphi$ of $\Theta$, acting on both vertices of every double, is an automorphism of $K$ mapping $R_k$ to $R_{\varphi(k)}$. It therefore suffices to prove that $K$ has no $R_a$-meeting $K_7$ model. The stabiliser of $a$ in $\operatorname{Aut}(\Theta)$ is trivial, and no symmetry is used in the rest of the argument.

\begin{lemma}[Case B constraints]\label[lemma]{rc:caseB}
Suppose that $K$ has a $K_7$ model whose branches partition $V(K)$, each meeting $R_a$, in which one branch $Y$ contains exactly two odd vertices $o_p,o_q$, the branch $X$ containing $e_a$ contains no odd vertex, and each other branch contains exactly one odd vertex. Write $C_x$ for the branch containing $o_x$ when $x\notin\{p,q\}$, put $W=\{w:e_w\in X\}$, let $I$ be the set of $s$ with $C_s=\{o_s\}$, and let $T=N(I)\setminus\{p,q\}$. Then:
\begin{enumerate}[label=(\roman*)]
\item every singleton branch is $\{o_s\}$ with $s\in I$, $I$ is independent in $\Theta$, and each $o_s$ with $s\in I$ has exactly one neighbour in every other branch;
\item $|N(s)\cap W|=1$ for every $s\in I$;
\item $pq\notin E(\Theta)$, and every $s\in I$ has exactly one neighbour in $\{p,q\}$;
\item with $M=N(p)\cap N(q)$ if $|Y|=2$ and $M=\varnothing$ if $|Y|\ge3$, the set $M$ is disjoint from $I\cup N(I)$ and
\begin{equation}\label{rc:budget}
 (|X|-2)+(|Y|-2)+|T|+|M|\le|I|;
\end{equation}
\item every branch $C_t$ with $t\in T\cup M$ has at least three vertices.
\end{enumerate}
\end{lemma}
\begin{proof}
(i) A singleton branch must have six neighbours. The branch $X$ is not a singleton, because its vertices are even and $e_a$ has degree four, and $Y$ has at least two vertices; so every singleton branch is $\{o_s\}$ for some $s\notin\{p,q\}$. Singleton branches are pairwise adjacent, and $o_so_{s'}$ is an edge only if $ss'\notin E(\Theta)$. Each $o_s$ has degree six and there are six other branches.

(ii) The neighbours of $o_s$ in $X$ are the vertices $e_w$ with $w\in N(s)\cap W$, and by (i) there is exactly one.

(v) for $T$. Let $s\in I$ and $t\in N(s)\setminus\{p,q\}$. The only odd vertex of $C_t$ is $o_t$, which is not adjacent to $o_s$, so $C_t$ contains some $e_y$ with $y\in N(s)$. If $C_t=\{o_t,e_y\}$, then $o_te_y$ is an edge, so $ty\in E(\Theta)$ with $y\neq t$, and $s,t,y$ would form a triangle.

The budget. The fourteen vertices form seven branches, so the excesses $|C|-2$ of the nonsingleton branches sum to $|I|$. By the previous paragraph,
\begin{equation}\label{rc:budget-first}
 (|X|-2)+(|Y|-2)+|T|\le|I|.
\end{equation}

(iii) If $I=\varnothing$, \eqref{rc:budget-first} gives $|X|=|Y|=2$. If $I\ne\varnothing$, \cref{rc:theta}(b) gives $|T|\ge|N(I)|-2\ge|I|-1$, so $(|X|-2)+(|Y|-2)\le1$. In either case $|Y|\le3$. If $pq\in E(\Theta)$, then $o_p,o_q$ are not adjacent, every inner vertex of an $o_p$--$o_q$ path inside $Y$ is even, and a path $o_pe_wo_q$ would need $w\in N(p)\cap N(q)$, a triangle; hence $|Y|\ge4$, a contradiction. Let $s\in I$. If $s$ has no neighbour in $\{p,q\}$, then $o_s$ is adjacent to both $o_p$ and $o_q$, contrary to (i). If $p,q\in N(s)$, then $o_s$ is adjacent to neither $o_p$ nor $o_q$, so its neighbour in $Y$ is some $e_y$ with $y\in N(s)$. If $y\notin\{p,q\}$, an edge from $e_y$ to $o_p$ or $o_q$ would give a triangle with $s$. If $y=p$, then $e_p$ is not adjacent to $o_p$, which lies in the same double, nor to $o_q$, because the two neighbours $p,q$ of $s$ are nonadjacent; the case $y=q$ is the same. Thus $e_y$ has no neighbour among $o_p,o_q$, and the connected set $Y$ has at least four vertices, again a contradiction.

(v) for $M$. Let $|Y|=2$, so $Y=\{o_p,o_q\}$ and $pq\notin E(\Theta)$, and let $m\in N(p)\cap N(q)$. Then $o_m$ is adjacent to neither $o_p$ nor $o_q$, so $C_m$ contains some $e_y$ with $y\in N(p)\cup N(q)$, say $y\in N(p)$. If $C_m=\{o_m,e_y\}$, then $y\in N(m)$, and $m,p,y$ would form a triangle unless $y=p$; but $e_p$ is adjacent to neither $o_p$ (same double) nor $o_q$ (as $pq\notin E(\Theta)$).

(iv) Let $m\in M$. By (iii), $m\notin I$. If $m\in N(s)$ for some $s\in I$, then $s$, $m$ and the neighbour of $s$ in $\{p,q\}$ would form a triangle. Thus $M\cap(I\cup N(I))=\varnothing$, in particular $M\cap T=\varnothing$, and the excess count with (v) gives \eqref{rc:budget}.
\end{proof}

\begin{lemma}[Residual configurations]\label[lemma]{rc:residual}
In the situation of \cref{rc:caseB}, $X=\{e_a,e_w\}$ with $w\in\{u,b\}$, equality holds in \eqref{rc:budget}, and $(I,\{p,q\},w)$ is one of the following, where in each case $|Y|=2$ except that Q1 also allows $|Y|=3$:
\[
\begin{array}{c|c|c|c}
\text{name}&I&\{p,q\}&w\\\hline
\mathrm{P1}&\varnothing&\{a,d\}&u\text{ or }b\\
\mathrm{P2}&\varnothing&\{b,c\}&u\text{ or }b\\
\mathrm{P3}&\{b\}&\{a,d\}&u\text{ or }b\\
\mathrm{P4}&\{c\}&\{a,d\}&u\\
\mathrm{P5}&\{b,c\}&\{a,d\}&u\\
\mathrm{P6}&\{a\}&\{b,c\}&u\text{ or }b\\
\mathrm{Q1}&\{a,c\}&\{b,d\}&u
\end{array}
\]
In each configuration every branch size is forced: $C_t$ is a triple for $t\in T\cup M$, $Y$ has the stated size, and every other nonsingleton branch is a pair. A pair containing $o_x$ is $\{o_x,e_y\}$ with $y\in N(x)$.
\end{lemma}
\begin{proof}
By \cref{rc:theta}(c), $|M|=1$ when $|Y|=2$ and $\{p,q\}\notin\{\{a,d\},\{b,c\}\}$, and $M=\varnothing$ otherwise.

\emph{First, $|Y|=3$, or $|Y|=2$ and $|M|=1$.} Then $(|Y|-2)+|M|=1$, so \eqref{rc:budget} gives $|T|\le|I|-1$, and $I\ne\varnothing$ because $|T|\ge0$. As $|T|\ge|N(I)|-2\ge|I|-1$, all these inequalities are equalities. Thus $|X|=2$; because $X$ is connected and contains $e_a$, we can write $X=\{e_a,e_w\}$ with $w\in N(a)=\{u,b\}$, so that $W=\{a,w\}$. Moreover $\{p,q\}\subseteq N(I)$ and $|N(I)|=|I|+1$. By \cref{rc:caseB}(iii) no $s\in I$ is adjacent to both $p$ and $q$, so $|I|\ne1$; and $|I|\le3$ by \cref{rc:theta}(b). For each of the six independent triples, every nonadjacent pair in $N(I)$ has a vertex of $I$ adjacent to both of its members, or, where marked `none', to neither; the table lists one such vertex for each pair:
\[
\begin{array}{c|c|l}
I&N(I)&\text{nonadjacent pairs in }N(I)\text{ with a violating vertex of }I\\\hline
zac&uvbd&uv\colon z,\ ub\colon a,\ ud\colon c,\ bd\colon z\text{ (none)}\\
zbd&uvac&uv\colon z,\ va\colon b,\ vc\colon d,\ ac\colon z\text{ (none)}\\
zad&uvbc&uv\colon z,\ ub\colon a,\ vc\colon d,\ bc\colon z\text{ (none)}\\
zbc&uvad&uv\colon z,\ va\colon b,\ ud\colon c,\ ad\colon z\text{ (none)}\\
ubd&zacv&za\colon u,\ zc\colon u,\ ac\colon u,\ av\colon b,\ cv\colon d\\
vac&zbdu&zb\colon v,\ zd\colon v,\ bd\colon v,\ ub\colon a,\ ud\colon c
\end{array}
\]
Hence $|I|=2$ and, since $|N(I)|=3$, $I$ is one of $za,zb,zc,zd,ac,bd$ by \cref{rc:theta}(b). For $za$, $N(I)=\{u,v,b\}$ and $uv,ub$ fail at $z,a$; similarly $zb$ ($uv,va$), $zc$ ($uv,ud$) and $zd$ ($uv,vc$) leave nothing. For $bd$, $N(I)=\{a,v,c\}$, and $av,vc$ fail at $b,d$, so $\{p,q\}=\{a,c\}$; but $N(d)=\{c,v\}$ is disjoint from $W=\{a,w\}$, contrary to \cref{rc:caseB}(ii). For $ac$, $N(I)=\{u,b,d\}$, and $ub,ud$ fail at $a,c$, so $\{p,q\}=\{b,d\}$; \cref{rc:caseB}(ii) at $c$, whose neighbours are $u,d$, forces $w=u$. This is Q1; it occurs with $|Y|=3$, or with $|Y|=2$ and $M=\{v\}$.

\emph{Second, $|Y|=2$ and $\{p,q\}\in\{\{a,d\},\{b,c\}\}$.} Now $M=\varnothing$, and by \cref{rc:caseB}(iii) every $s\in I$ lies in $N(p)\cup N(q)$, which is $\{u,b,c,v\}$ for $\{a,d\}$ and $\{a,v,u,d\}$ for $\{b,c\}$. Inequality \eqref{rc:budget} reads $(|X|-2)+|N(I)\setminus\{p,q\}|\le|I|$. The independent subsets of these four-sets and the corresponding sets $N(I)\setminus\{p,q\}$ are
\[
\begin{array}{c|l}
\{p,q\}=\{a,d\}&\varnothing\colon\varnothing;\ u\colon zc;\ v\colon zb;\ b\colon v;\ c\colon u;\ ub\colon zcv;\ uv\colon zcb;\ bc\colon uv;\ cv\colon uzb\\
\{p,q\}=\{b,c\}&\varnothing\colon\varnothing;\ v\colon zd;\ u\colon za;\ a\colon u;\ d\colon v;\ av\colon uzd;\ uv\colon zad;\ ad\colon uv;\ ud\colon zav
\end{array}
\]
The inequality fails exactly for $u,v,ub,uv,cv$ and for $v,u,av,uv,ud$, respectively, and holds with $|X|=2$ in the remaining cases. So $X=\{e_a,e_w\}$ with $w\in N(a)=\{u,b\}$. Finally \cref{rc:caseB}(ii) excludes $I=\{d\}$ and $I=\{a,d\}$, since $N(d)=\{c,v\}$ misses $\{a,u,b\}$, and forces $w=u$ whenever $c\in I$, since $N(c)=\{u,d\}$. What remains is P1--P6.

In every case \eqref{rc:budget} holds with equality, which forces the sizes. The last assertion holds because $C_x$ is connected, has $o_x$ as its only odd vertex, and $o_x$ is adjacent to $e_y$ exactly when $y\in N(x)$.
\end{proof}

\begin{proof}[Proof of the negative assertions of \cref{rc:classification}]
By \cref{rc:theta}(d) it suffices to treat $k=a$. Suppose that $K$ has an $R_a$-meeting $K_7$ model. Since $K$ is connected, \cref{rc:normalization} extends it to a spanning one, in which exactly one branch $D$ contains two roots. If $e_a\in D$, then every branch contains exactly one odd vertex, so the model is rooted in $O$; this contradicts \cref{lem:known-theta}, because the relabelling $(u,v,z,a,b,c,d)\mapsto(5,6,2,0,3,1,4)$ maps $\Theta$ onto the active graph of \cref{sec:sharpness} and hence $K$ onto $H-s$, preserving odd vertices. Otherwise $D$ contains two odd vertices, and we are in the situation of \cref{rc:caseB}; one of the configurations of \cref{rc:residual} occurs. We use three observations.
\begin{enumerate}[label=(E\arabic*)]
\item If $w=b$, then the neighbourhood of $X$ in $K$ (not in $\Theta$) is $N_K(X)=\{o_u,o_a,o_b,o_v,e_u,e_v\}$; it has six vertices, so each of the six other branches contains exactly one of them.
\item If $w=u$ and $C_z$ is a pair, then $C_z=\{o_z,e_v\}$, because $N(z)=\{u,v\}$ and $e_u\in X$.
\item For $s\in I$ the six neighbours of $o_s$ lie in the six other branches, one in each (\cref{rc:caseB}(i)).
\end{enumerate}
\emph{P1} ($Y=\{o_a,o_d\}$; pairs $C_u,C_v,C_z,C_b,C_c$). If $w=u$, then $C_z=\{o_z,e_v\}$ by (E2), and the pair $C_b=\{o_b,e_y\}$ needs $y\in N(b)=\{a,v\}$, where $e_a\in X$ and $e_v\in C_z$. If $w=b$, then by (E1) the pairs $C_c$ and $C_z$, whose odd vertices are not in $N_K(X)$, contain $e_u$ or $e_v$; as $o_c$ is not adjacent to $e_v$, $C_c=\{o_c,e_u\}$ and $C_z=\{o_z,e_v\}$, and $C_b$ fails as before.

\emph{P2} ($Y=\{o_b,o_c\}$; pairs $C_u,C_v,C_z,C_a,C_d$). If $w=u$, then $C_z=\{o_z,e_v\}$ by (E2), $C_a=\{o_a,e_b\}$ as $N(a)\setminus W=\{b\}$, $C_d=\{o_d,e_c\}$ as $N(d)=\{c,v\}$ and $e_v\in C_z$, and then $C_u=\{o_u,e_z\}$ as $N(u)=\{z,a,c\}$, $e_a\in X$ and $e_c\in C_d$. But $C_u$ and $C_a$ are not adjacent: $o_uo_a$ is not an edge since $ua\in E(\Theta)$, and $ub,za,zb\notin E(\Theta)$. If $w=b$, then by (E1) the pairs $C_d$ and $C_z$ contain $e_u$ or $e_v$; as $o_d$ is not adjacent to $e_u$, $C_d=\{o_d,e_v\}$ and $C_z=\{o_z,e_u\}$, and the pair $C_a$ needs $e_y$ with $y\in N(a)\setminus W=\{u\}$, which is used.

\emph{P3} ($I=\{b\}$, $Y=\{o_a,o_d\}$, triple $C_v$). The neighbours of $o_b$ are $e_a\in X$, $o_d\in Y$, $o_u\in C_u$, $o_z\in C_z$, $o_c\in C_c$ and $e_v$, so (E3) gives $e_v\in C_v$. If $w=u$, (E2) forces $C_z=\{o_z,e_v\}$, but $e_v\in C_v$. If $w=b$, then $C_v$ contains both $o_v$ and $e_v$, contrary to (E1).

\emph{P4} ($I=\{c\}$, $Y=\{o_a,o_d\}$, $w=u$, triple $C_u$). The pair $C_b=\{o_b,e_y\}$ has $y\in N(b)\setminus W=\{v\}$, so (E2) leaves no even vertex for the pair $C_z$.

\emph{P5} ($I=\{b,c\}$, $Y=\{o_a,o_d\}$, $w=u$, triples $C_u,C_v$). The neighbours of $o_b$ are $e_a\in X$, $o_d\in Y$, $o_u\in C_u$, $o_z\in C_z$, the singleton $o_c$, and $e_v$, so (E3) gives $e_v\in C_v$, and (E2) leaves no even vertex for the pair $C_z$.

\emph{P6} ($I=\{a\}$, $Y=\{o_b,o_c\}$, triple $C_u$). The neighbours $o_c\in Y$, $o_v\in C_v$, $o_z\in C_z$, $o_d\in C_d$ of $o_a$ leave $e_u,e_b$ for the branches $X$ and $C_u$, one each, by (E3). If $w=u$, then $e_b\in C_u=\{o_u,e_b,e_y\}$; as $o_ue_b$ is not an edge, $e_y$ must be adjacent to both, so $y\in N(u)\cap N(b)=\{a\}$, but $e_a\in X$. If $w=b$, then $e_u\in C_u$, so $C_u$ contains both $o_u$ and $e_u$, contrary to (E1).

\emph{Q1} ($I=\{a,c\}$, $\{p,q\}=\{b,d\}$, $w=u$, triple $C_u$). The neighbours of $o_a$ are $e_u\in X$, the singleton $o_c$, $o_d\in Y$, $o_v\in C_v$, $o_z\in C_z$ and $e_b$; those of $o_c$ are $e_u\in X$, the singleton $o_a$, $o_b\in Y$, $o_v\in C_v$, $o_z\in C_z$ and $e_d$. By (E3), $e_b,e_d\in C_u$, so $C_u=\{o_u,e_b,e_d\}$. This set is independent, because $ub,ud,bd\notin E(\Theta)$, so it is not connected.

Every configuration is impossible. Hence $K$ has no $R_a$-meeting $K_7$ model, and by the transport above none for the marks $b,c,d$ either.
\end{proof}

\subsection*{Finite corroboration}
\addcontentsline{toc}{subsection}{Finite corroboration}
Revision~6 proved the same negative assertions by a complete finite certificate, which is retained unchanged in \path{research/root_control/} as independent corroboration. Its domain is the one described in \cref{rc:normalization}: for code $42232$ and mark $2$, with blocked numbering $o_i=i$, $e_i=7+i$, the $28$ root pairs and $7^6$ allocations of the six nonroots give $3294172$ candidates in a fixed order. The file \path{rejection_proof.bin.gz} expands to one byte per candidate, naming either a disconnected branch or an anticomplete pair of branches; the uncompressed stream has SHA-256
\begin{center}\ttfamily\small
\seqsplit{84ec2894c7a29240d425def4da1ae4c87d231fbae7b013fe213053d0124d79fe}.
\end{center}
Three separately written validators (revision~6's Python and C++ checkers, and a bitmask checker written for revision~7) accept every record and reject corrupted streams and the stream applied to the positive mark~$0$. A second, independently produced certificate refutes mark $2$ by $27$ singleton cases, $183$ nodes and $90$ leaves; its standalone verifier reconstructs every permitted branch. Revision~7 also mechanically checks every finite assertion used above (\cref{rc:theta}, both tables of \cref{rc:residual}, and the forced placements) and, as a test of the reductions, enumerates all spanning partitions compatible with \cref{rc:caseB}. None of these computations is a premise of the proof.

\paragraph{Scope.}
The normalization, the positive constructions of Appendix~\ref{ua:appendix}, the orbit transport and, since revision~7, the negative assertions are structural; \cref{rc:classification,rc:apex} have no computational premise. The negative proof uses the rooted theta obstruction \cref{lem:known-theta} of the core paper, whose counting argument is due to Kriesell and Mohr. None of this changes the premises of \cref{thm:main}, Input~\ref{input:P7}, the attachment theorem, or the conditional $(2^8,4)$ argument.

\section{Direct proofs of the all-double cases}\label{app:all-doubles}
The core paper cites the routing theorem of Kriesell and Mohr \cite[Theorem~6, Section~5, p.~12]{KriesellMohr2022} for witnesses whose core consists of three, four, five or six double fibres (in the proofs of \cref{thm:excess-five} and \cref{lem:six-doubles}). This appendix keeps the direct proofs of those cases from revision~6; they do not use that theorem. Throughout, a normalized core consisting of double fibres $\{o_i,e_i\}$ is the graph described by \cref{lem:double} for its active graph $B$, and normalization means $\delta(B)\ge2$. Every branch below contains an odd vertex.

\subsection{Three and four doubles}

For three doubles, $B=K_3$. The branches
$\{o_1,e_2\},\{o_2,e_3\},\{o_3,e_1\}$ have active internal edges and
pairwise crossing edges $e_2e_3,e_3e_1,e_1e_2$.
For four doubles, $\Delta(\overline B)\le1$, so $\overline B$ is a matching
with isolated vertices; $B$ consequently contains a spanning cycle
$1,2,3,4,1$. Put $A_i=\{o_i,e_{i+1}\}$, cyclically. The four internal
edges are active cycle edges. The four consecutive pairs are joined by
the corresponding even--even cycle edge. The two opposite pairs are
joined by $o_1e_4$ and $o_2e_1$. These six crossing edges depend only on
the fixed cycle, so additional active chords cannot remove any edge used.
Every branch has its odd $o_i$.

\subsection{Five doubles}

Suppose $|V(B)|=5$ and $\delta(B)\ge2$. If $B$ has a spanning cycle,
take $A_i=\{o_i,e_{i+1}\}$ cyclically. Consecutive pairs meet through
their even vertices. For distance-two pairs $A_i,A_{i+2}$, the edge
$e_{i+1}o_{i+2}$ crosses. Every pair of indices on a five-cycle is
consecutive or at distance two, so this checks all ten crossings.

For completeness, the classification without a five-cycle is valid. A graph of minimum degree two has a cycle. If
there is a four-cycle $1,2,3,4,1$ with fifth vertex $w$, any consecutive
neighbours of $w$ would yield a five-cycle. Thus its at least two
neighbours are exactly one opposite pair, say $1,3$. Edge $24$ would
give the five-cycle $w,1,2,4,3,w$ and is absent, while edge $13$ is
optional. This is exactly $K_{2,3}$ with optional edge between the two
degree-three hubs. If there is no four- or five-cycle, choose a triangle
$a,b,c$. Each remaining vertex has at most one triangle neighbour:
two such neighbours and the third triangle vertex create a four-cycle.
Minimum degree therefore forces the two remaining vertices to be adjacent
to each other and to have exactly one triangle neighbour each. Distinct
triangle neighbours yield a five-cycle through the third triangle vertex;
their common neighbour instead gives exactly a bowtie. No other
configuration is possible.

For a bowtie with centre $1$ and leaf pairs $23,45$, use
\[
 \{o_1\},\quad\{o_2,e_3\},\quad\{o_3,e_2\},\quad
 \{o_4,e_5\},\quad\{o_5,e_4\}.
\]
Each of four centre crossings is an odd--even edge from an active spoke;
the two within-pair crossings are $e_2e_3$ and $e_4e_5$; the four
between-pair crossings are the odd--odd edges for the inactive cross-pairs.
For $K_{2,3}$ with hubs $1,2$ and leaves $3,4,5$, use
\[
 \{o_1,e_1,e_3\},\quad\{o_2,e_2,e_4\},\quad
 \{o_3\},\quad\{o_4\},\quad\{o_5\}.
\]
The two nonsingleton branches have paths $o_1,e_3,e_1$ and $o_2,e_4,e_2$.
They meet through $e_3e_2$. Each meets every leaf root through its hub
even vertex, giving six crossings; the three leaf--leaf pairs use inactive
odd--odd edges. The optional hub edge affects none of these edges. All
branches in all three constructions are disjoint and odd-rooted.

\subsection{Six doubles}
\begin{proof}[Direct proof of \cref{lem:six-doubles}]
Its active graph $B$ has six vertices and minimum degree at least two.
We first give three constructions in $B$.

\emph{Independent-set construction.} If $S$ is independent and there
is a matching of $C=V(B)\setminus S$ into $S$, write $s(c)$ for the
mate of $c\in C$. Use
\[
 \{o_s\}\ (s\in S),\qquad
 \{o_c,e_c,e_{s(c)}\}\ (c\in C).
\]
The latter sets are paths containing whole doubles. Absorption handles
their adjacencies, and the singletons are pairwise adjacent because $S$
is independent in $B$.

\emph{Two-singleton construction.} Suppose $ab$ is a nonedge, $uv$ is
an edge disjoint from $\{a,b\}$, and $c,d$ are the remaining vertices.
Assume $a$ has the same adjacency status to $u,v$, as does $b$, and
that $ac,bd$ are edges, after exchanging $c,d$ if necessary. Use
\[
 \{o_a\},\ \{o_b\},\ \{o_c,e_c,e_a\},\ \{o_d,e_d,e_b\},\
 \{o_u,e_v\},\ \{o_v,e_u\}.
\]
The middle branches absorb $D_c,D_d$. The last two are adjacent by
$e_ue_v$. Each singleton sees both last branches, through their odd
vertices if its two statuses are zero and through their even vertices
if they are one.

\emph{Matching construction.} Suppose a perfect matching $M$ has
no pair of edges with exactly one active edge between their endpoint
sets. With $m(i)$ the mate of $i$, use $\{o_i,e_{m(i)}\}$ for all $i$.
Matched branches meet through their even vertices. For indices $i,j$
in different matching edges, an inactive $ij$ gives $o_io_j$; if $ij$
is active, another active edge among
$i m(j),m(i)j,m(i)m(j)$ gives one of the other three possible crossings.

Two special active graphs will be useful. Suppose two triangles
$\{a,x_1,x_2\}$ and $\{b,y_1,y_2\}$ have crossing set
$E\subseteq\{x_1,x_2\}\times\{y_1,y_2\}$. If $|E|=0$ or
$|E|\geq2$, take
\[
 \{o_a\},\ \{o_b\},\ \{o_{x_1},e_{x_2}\},\
 \{o_{x_2},e_{x_1}\},\ \{o_{y_1},e_{y_2}\},\
 \{o_{y_2},e_{y_1}\}.
\]
The singleton adjacencies follow from the triangles and absence of
crossings at $a,b$. Across the two swapped pairs, an inactive root
pair gives an odd--odd edge; an active root pair is accompanied by a
second member of $E$, giving one of the other three crossing edges.
If $E=\{x_1y_1\}$, the matching
$\{ax_2,x_1y_1,by_2\}$ has respectively two, two and zero crossing
edges between its pairs, so the matching construction applies. Also,
for the triangular prism, its three crossing matching edges satisfy
the same construction, with two crossings between each pair.

It remains to show that these constructions exhaust $B$. Minimum
degree gives $\alpha(B)\leq4$. If $\alpha(B)=4$, an independent
four-set is complete to its two-vertex complement, and the independent-set
construction applies. If $\alpha(B)=3$, take an independent triple
$S$ and its complement $C$. Every vertex of $S$ sees at least two
vertices of $C$, and every vertex of $C$ sees $S$ by maximality.
Hall's condition for matching $C$ into $S$ follows: two vertices of
$C$ cannot together see at most one vertex of $S$, since then the
other two vertices of $S$ would each have degree at most one.

Now let $\alpha(B)\leq2$. If $B$ is complete, use any perfect matching.
Otherwise choose a nonedge $ab$. Partition the other vertices into
$X,Y,Z$, according as they see only $a$, only $b$, or both. There
is no fourth type, since it would make an independent triple with
$a,b$. Writing their sizes as $x,y,z$, we have
\[
 x+y+z=4,\qquad x+z\geq2,\qquad y+z\geq2.
\]
The sets $X,Y$ are cliques. If $z\geq3$, choose an edge $uv$ in
$Z$; the remaining pair can be matched to $a,b$, giving the
two-singleton construction. The same construction applies if
$z=2$ and $(x,y)=(2,0)$ or $(0,2)$, using the two-element clique
$X$ or $Y$ as $\{u,v\}$; and if $z=1$, using the larger of $X,Y$.
If $z=0$, then $x=y=2$, giving the two-triangle family above.

The only unresolved case is $z=2$, $x=y=1$. Write $X=\{x\}$,
$Y=\{y\}$ and $Z=\{u,v\}$. If $uv$ is an edge, the
two-singleton construction again applies. Otherwise $x$ and $y$
each see at least one of $u,v$. Repeating the partition with nonedge
$uv$ settles the case if either sees both, or if both see the same
one. Otherwise relabel so $xu,yv$ are edges and $xv,yu$ are not.
Then $\{a,x,u\}$ and $\{b,y,v\}$ are triangles whose crossings
are $av,ub$ and possibly $xy$. Without $xy$ this is the two-triangle
family; with $xy$ it is the triangular prism. All constructions are
disjoint, connected and rooted at the six displayed odd vertices.
\end{proof}

\clearpage
\phantomsection\addcontentsline{toc}{part}{References for both parts}
\bibliographystyle{plain}
\bibliography{references}
\end{document}